%% file: I2m.tex
\documentclass[11pt]{amsart}
\usepackage[hmargin=3cm,vmargin=3cm]{geometry}

\usepackage[latin1]{inputenc}
\usepackage{xspace,amssymb,amsfonts,euscript}
\usepackage{amsthm,amsmath,amscd,stmaryrd,latexsym}

\usepackage{graphicx}
\usepackage[all]{xy}
\SelectTips{cm}{}

\usepackage{tikz, tikz-cd}
\usetikzlibrary{matrix, arrows, calc}

\usepackage[dvips]{epsfig}
\usepackage{pinlabel}
\usepackage{psfrag}

\newcommand{\ig}[2]{\vcenter{\xy (0,0)*{\includegraphics[scale=#1]{fig/#2}} \endxy}}

\newcommand{\igc}[2]{\begin{center} \includegraphics[scale=#1]{fig/#2} \end{center}}

\IfFileExists{comments.sty}{\usepackage{comments}}

\RequirePackage{color}
\definecolor{myred}{rgb}{0.75,0,0}
\definecolor{mygreen}{rgb}{0,0.5,0}
\definecolor{myblue}{rgb}{0,0,0.65}

\RequirePackage{ifpdf}
\ifpdf
  \IfFileExists{pdfsync.sty}{\RequirePackage{pdfsync}}{}
  \RequirePackage[pdftex,
   colorlinks = true,
   urlcolor = myblue, 
   citecolor = mygreen, 
   linkcolor = myred, 
   pagebackref,
   bookmarksopen=true]{hyperref}
\else
  \RequirePackage[hypertex]{hyperref}
\fi

\newtheorem{thm}{Theorem}[section]
\newtheorem{lemma}[thm]{Lemma}

\newtheorem{prop}[thm]{Proposition}
\newtheorem{cor}[thm]{Corollary}
\newtheorem{claim}[thm]{Claim}

\newtheorem*{prop*}{Proposition}

\theoremstyle{definition}
\newtheorem{defn}[thm]{Definition}

\newtheorem{notation}[thm]{Notation}

\newtheorem{example}[thm]{Example}
\newtheorem{assumption}[thm]{Assumption}

\theoremstyle{remark}
\newtheorem{remark}[thm]{Remark}
\newtheorem{rmk}[thm]{Remark}

\numberwithin{equation}{section}

    \def\BM{{\mathbb{B}}}
    \def\CM{{\mathbb{C}}}
\def\DG{{\mathfrak D}}    
    
\def\FG{{\mathfrak F}}    \def\FM{{\mathbb{F}}}
    
\def\HG{{\mathfrak H}}  \def\hg{{\mathfrak h}}

    \def\LM{{\mathbb{L}}}

    \def\RM{{\mathbb{R}}}
    \def\SM{{\mathbb{S}}}
    \def\TM{{\mathbb{T}}}

    \def\ZM{{\mathbb{Z}}}

    \def\CC{{\mathcal{C}}}
    \def\DC{{\mathcal{D}}}
    
    \def\FC{{\mathcal{F}}}
    
\def\HB{{\mathbf H}}

    \def\LC{{\mathcal{L}}}

    \def\SC{{\mathcal{S}}}

\def\a{\alpha}

\def\d{\delta}

\def\e{\varepsilon}

\def\l{\lambda}

\def\w{\omega}

\def\z{\zeta}

\let\phi=\varphi

\usepackage{bbm}
\def\C{{\mathbbm C}}
\def\N{{\mathbbm N}}
\def\R{{\mathbbm R}}
\def\Z{{\mathbbm Z}}
\def\Q{{\mathbbm Q}}
\def\1{\mathbbm{1}}

\newcommand{\ul}{\underline}
\newcommand{\ubr}[2]{\underbrace{#1}_{#2}}
\newcommand{\mf}[1]{\mathfrak{#1}}

\newcommand{\ot}{\otimes}
\newcommand{\pa}{\partial}
\newcommand{\co}{\colon}

\renewcommand{\to}{\rightarrow}
\newcommand{\into}{\hookrightarrow}

\newcommand{\sumset}{\stackrel{\scriptstyle{\oplus}}{\scriptstyle{\subset}}}

\newcommand{\define}{\stackrel{\mbox{\scriptsize{def}}}{=}}

\renewcommand{\sl}{\mathfrak{sl}}

\newcommand{\Bim}{\textbf{Bim}}
\newcommand{\Kar}{\textbf{Kar}}
\newcommand{\Hom}{{\rm Hom}}
\newcommand{\HOM}{{\rm HOM}}
\newcommand{\End}{{\rm End}}

\newcommand{\Res}{{\rm Res}}
\newcommand{\Ind}{{\rm Ind}}

\newcommand{\grdrk}{{\rm grdrk}}

\renewcommand{\Im}{\textrm{Im}}

\newcommand{\ii}{\underline{\textbf{\textit{i}}}}

\newcommand{\mTL}{\mathcal{TL}}
\newcommand{\mTTL}{\mathfrak{2}\mathcal{TL}}
\newcommand{\DCti}{\DC(\infty)}
\newcommand{\DGti}{\DG(\infty)}

\newcommand{\qtwo}{(v+v^{-1})}
\newcommand{\qJ}{\left[J\right]}

\newcommand{\qW}{\left[W\right]}

\newcommand{\negl}{{\rm negl}}

\newcommand{\Zvv}{\ensuremath{\mathbbm{Z}\left[v^{\pm 1}\right]}}

\newcommand{\Zqq}{\ensuremath{\mathbbm{Z}\left[q^{\pm 1}\right]}}

\newcommand{\SBim}{\SM\textrm{Bim}}
\newcommand{\BSBim}{\BM\SM\textrm{Bim}}
\newcommand{\SSBim}{\SC\SM\textrm{Bim}}
\newcommand{\SBSBim}{\SC\BM\SM\textrm{Bim}}
\newcommand{\fooBim}{g\BM\SM\textrm{Bim}}

\newcommand{\DCfoo}{g\DC}

\renewcommand{\hat}{\widehat}

\psfrag{Suma1ton}{$\displaystyle \sum_{a=1}^n$}
\psfrag{JW1}{$JW_1$}
\psfrag{JW2}{$JW_2$}
\psfrag{JW3sm}{\tiny{$JW_3$}}
\psfrag{JW3}{$JW_3$}
\psfrag{JWn}{$JW_n$}
\psfrag{JWnm1}{$JW_{n-1}$}
\psfrag{JWmm1}{$JW_{m-1}$}
\psfrag{JWnp1}{$JW_{n+1}$}
\psfrag{1over2}{\Large{$\frac{1}{[2]}$}}
\psfrag{1over3}{\Large{$\frac{1}{[3]}$}}
\psfrag{2over3}{\Large{$\frac{[2]}{[3]}$}}
\psfrag{novernp1}{\Large{$\frac{[n]}{[n+1]}$}}
\psfrag{aovernp1}{\Large{$\frac{[a]}{[n+1]}$}}
\psfrag{1overx}{\Large{$\frac{1}{x}$}}
\psfrag{1overxym1}{\Large{$\frac{1}{xy-1}$}}
\psfrag{xoverxym1}{\Large{$\frac{x}{xy-1}$}}
\psfrag{yoverxym1}{\Large{$\frac{y}{xy-1}$}}
\psfrag{a}{$a$}
\psfrag{x}{$x$}
\psfrag{b}{$b$}
\psfrag{y}{$y$}
\psfrag{f}{$f$}
\psfrag{fs}{$\a_s$}
\psfrag{ft}{$\a_t$}
\psfrag{sif}{$s(f)$}
\psfrag{paif}{$\pa_s(f)$}
\psfrag{pasf}{$\pa_s(f)$}
\psfrag{1half}{\Large{$\frac{1}{2}$}}
\psfrag{Deltas}{$\Delta_s$}
\psfrag{DeltasW}{$\Delta^s_W$}
\psfrag{DeltaW}{$\Delta_W$}
\psfrag{pasWf}{$\pa^s_W(f)$}
\psfrag{musW}{$\frac{\LM}{\a_s}$}
\psfrag{paDeltast}{$\pa\Delta_{st}$}
\psfrag{musWpat}{$\frac{\LM}{\a_s \a_t}$}
\psfrag{vk}{$v_k$}
\psfrag{vkm1}{$v_{k-1}$}
\psfrag{vkp1}{$v_{k+1}$}
\psfrag{vm}{$v_m$}
\psfrag{vmp1}{$v_{m+1}$}
\psfrag{cm}{$C_m$}
\psfrag{qtwo}{$[2]$}
\psfrag{mqtwo}{$-[2]$}

\title[The Dihedral Cathedral]{The Two-Color Soergel Calculus}
\author[Ben Elias]{or: The Dihedral Cathedral \\ \\ Ben Elias}
 
\begin{document}

\maketitle

\begin{abstract} We give a diagrammatic presentation for the category of Soergel bimodules for the dihedral group $W$. The (two-colored) Temperley-Lieb category is embedded inside this
category as the degree $0$ morphisms between color-alternating objects. The indecomposable Soergel bimodules are the images of Jones-Wenzl projectors. When $W$ is infinite, the parameter
$q$ of the Temperley-Lieb algebra may be generic, yielding a quantum version of the geometric Satake equivalence for $\sl_2$. When $W$ is finite, $q$ must be specialized to an appropriate
root of unity, and the negligible Jones-Wenzl projector yields the Soergel bimodule for the longest element of $W$.\end{abstract}

\setcounter{tocdepth}{1}
\tableofcontents

\newpage

\section{Introduction} 

\input Introduction.tex

\section{The dihedral group and its Hecke algebra}
\label{sec-dihedral}

\input DihedralAndHecke.tex

\section{Frobenius extensions and the Soergel categorification}
\label{sec-frobenius}

 \input SoergelCatfn.tex

\section{Temperley-Lieb categories}
\label{sec-TL}

\input TemperleyLieb.tex

\section{Dihedral diagrammatics: $m=\infty$}
\label{sec-didi}

 \input InftyDiag.tex

\section{Dihedral diagrammatics: $m<\infty$}
\label{sec-didifinite}

\input FiniteDiag.tex

\appendix
\section{Non-symmetric and unbalanced Cartan matrices}
\label{sec-exoticdihedral}

\input ExoticDihedral.tex

%
%

\bibliographystyle{plain}
\bibliography{everyone}{}

\vspace{0.1in}
 
\noindent
{\textsl \small Ben Elias, Department of Mathematics, University of Oregon, Eugene, OR, USA}

\noindent 
{\tt \small email: belias@uoregon.edu}

\end{document}

%% file: Introduction.tex
\subsection{{\bf Overview}}

Let $(W,S)$ be any Coxeter group, and let $\HB = \HB_W$ be its associated Hecke algebra. Kazhdan and Lusztig \cite{KaLu1} introduced a particular basis of $\HB$, now known as the
Kazhdan-Lusztig basis or the \emph{KL basis}. This basis was conjectured to have certain positivity properties. One way to prove this positivity would be to construct an additive
monoidal category whose Grothendieck ring is isomorphic to $\HB$, and whose indecomposable objects descend to the KL basis. When $W$ is a Weyl group, Kazhdan and Lusztig \cite{KaLu2}
constructed such a categorification using geometric techniques, by considering perverse sheaves on the flag variety. Similar methods can be used for other crystallographic Coxeter
groups \cite{Harterich}, but for a general Coxeter group there are as yet no geometric tools available.

In the early 1990s, Soergel constructed an algebraic categorification of the Hecke algebra, using certain bimodules over the coordinate ring $R$ of the reflection representation of $W$.
These bimodules, now called \emph{Soergel bimodules}, can be defined for any Coxeter group, and can be studied using combinatorial and algebraic methods. When $W$ is a Weyl group, the
category of Soergel bimodules agrees with the (equivariant) hypercohomology of the perverse sheaves used in the categorification of Kazhdan and Lusztig, and thus geometric techniques can
be used to study Soergel bimodules as well. Soergel conjectured that the indecomposable Soergel bimodules should descend to the KL basis, when defined over a field of characteristic zero,
though in the absence of geometric tools there is no a priori reason this should be true. This conjecture was recently proven by the author and Geordie Williamson in \cite{EWHodge}. We
refer the reader to \cite{Soe5} for a purely algebraic account of Soergel bimodules, and to numerous other papers \cite{Soe1,Soe2,Soe3,Soe4} for the complete story.

In this paper we study Soergel bimodules for dihedral groups in great detail, and present several new results. An arbitrary Coxeter group is in some sense built out of dihedral groups, so
this is an appropriate place to begin. The Kazhdan-Lusztig theory of dihedral groups is well-understood: Soergel himself proved that the indecomposable Soergel bimodules descend to the KL
basis (under certain assumptions). Perhaps because of the simplicity of dihedral groups, there are few resources available for their study. Our priority in this paper is to provide a
thorough discussion and understanding of dihedral groups. This paper will be a springboard for future works, including several in progress by the author and G. Williamson.

\subsection{{\bf Soergel bimodules and Bott-Samelson bimodules}}

To ease the introduction, we postpone some of the subtleties to the next section. For now, let us define $\hg$ to be the reflection representation of $W$ over $\RM$, as defined say in
\cite{Humphreys}. Soon, we will allow other similar representations, called realizations, to take the place of $\hg$.

Let $R$ be the symmetric algebra of $\hg$, graded with linear terms in degree $2$. It is equipped with an action of $W$, and therefore an action of each simple reflection $s \in S$. Let
$R^s$ denote the subring of $s$-invariants. We define a graded $R$ bimodule by \[B_s \define R \ot_{R^s} R (1) \] where $(1)$ denotes a grading shift. Tensor products of the bimodules
$B_s$ for various $s \in S$ are known as \emph{Bott-Samelson bimodules}, and they form a monoidal category $\BSBim$. By definition, a \emph{Soergel bimodule} is an element of the graded,
additive, Karoubian category $\SBim$ generated by the Bott-Samelson bimodules. Concretely, an indecomposable Soergel bimodule is an indecomposable summand of a Bott-Samelson bimodule.

Soergel has proven that there is an indecomposable bimodule $B_w$ for $w \in W$, appearing as a direct summand (with multiplicity one) inside \[B_w \sumset B_{s_1} \ot \cdots \ot B_{s_d}
\] when $s_1 \cdots s_d$ is a reduced expression for $w$. Up to grading shift, these indecomposable bimodules $\{B_w\}_{w \in W}$ form a complete list of non-isomorphic indecomposable
objects. The Grothendieck group of $\SBim$ is isomorphic to the Hecke algebra of $W$. Soergel also gave a formula for the graded dimensions of morphism spaces between Soergel bimodules.
These results are collectively packaged as the \emph{Soergel Categorification Theorem} or \emph{SCT}. In fact, Soergel proved the SCT for representations $\hg$ which are
``reflection-faithful," a class of representations which need not include the reflection representation. Libedinsky \cite{LibRR} showed that the SCT holds for the reflection representation
regardless. Remember that the SCT does not entail the stronger statement known as the \emph{Soergel conjecture}, saying that the indecomposable bimodules descend in the Grothendieck group
to the KL basis.

The primary result of this paper is a presentation of the morphisms in the category $\BSBim$ by generators and relations, in the case when $W$ is a dihedral group. That is, we define a
monoidal category $\DC$ by generators and relations, whose morphisms are linear combinations of planar diagrams. We construct a monoidal functor $\DC \to \BSBim$, and prove that the
functor is an equivalence. We also give an explicit description of the idempotents which pick out each indecomposable $B_w$, thus implying the SCT and the Soergel conjecture.

Let $S=\{s,t\}$ be the set of simple reflections, and $m=m_{s,t}$ be the order of $st$. The same presentation has been given before by Libedinsky \cite{LibRA} for the right-angled cases
$m=2,\infty$. His work is complimentary, as he does not discuss idempotents or connections to the Temperley-Lieb algebra (see below), and his proofs are entirely different.

A morphism in $\DC$ will be represented by a graph with boundary, properly embedded in the planar strip $\R \times [0,1]$. The edges of this graph are labeled by elements of $S$, which we
call ``colors." The only vertices appearing are univalent vertices, trivalent vertices joining three edges of the same color, and if $m$ is finite, vertices of valence $2m$ whose edge
labels alternate between the two colors. We call these \emph{Soergel graphs}. A number of relations are placed on Soergel graphs, which (after some abstraction) can be represented in a way
independent of $m$ (when $m$ is finite).

A more significant goal would be to find a diagrammatic presentation for $\BSBim$ in the case of an arbitrary Coxeter group. For type $A$, the presentation was found by the author and M.
Khovanov in \cite{EKh}. The general case is accomplished in forthcoming work between the author and Geordie Williamson \cite{EWGR4SB}, which relies heavily on this paper. The form of the
presentation is revealing. There is one generating object for each $s \in S$, or more verbosely, for each rank 1 finitary parabolic subgroup. The generating morphisms are associated to
finitary subgroups of rank 1 (univalent and trivalent vertices) and rank 2 ($2m$-valent vertices), and the relations are associated to finitary subgroups of rank $\le 3$. This paper
tackles the generators and relations of rank $\le 2$.

Having a diagrammatic presentation in type $A$ has led the author to numerous other results, such as: \begin{itemize} \item Categorifications of induced trivial modules \cite{EInduced}.
\item A ``thick calculus" for partial idempotent completions \cite{EInduced}. \item A categorification of the Temperley-Lieb quotient of $\HB$ \cite{ETemperley}. \item A conjectural
presentation of the 2-category of Singular Soergel bimodules, joint with G. Williamson. \end{itemize} The dihedral analogs of these results will be proven in this paper as well. The
exposition of these results will be self-contained, so the reader will not need to consult these other works. The final result, a presentation of Singular Soergel bimodules, is essential
to this work, and is described in full detail. The remaining proofs are only sketched, but the details are easy to fill in after consulting the other papers.

\subsection{{\bf Diagrams are better than bimodules}}

A \emph{realization} of a Coxeter group $W$ is a certain kind of representation $\hg$ of $W$ over some commutative base ring $\Bbbk$, equipped with a choice of simple roots and co-roots.
For example, the reflection representation yields a faithful realization over $\RM$. However, $\Bbbk$ can be an arbitrary commutative ring, and the realization need not be faithful. One of
our goals in this paper is to extend ``the study of Soergel bimodules" to arbitrary realizations (caveat: satisfying very minor assumptions).

Here is a crucial example to keep in mind. Let $W$ be the infinite dihedral group. In this case, the additive category of Soergel bimodules for the reflection representation over $\RM$ is
equivalent to some additive category of semisimple equivariant perverse sheaves on a Kac-Moody group, and both categorify the Hecke algebra of $W$. Suppose that we work instead over a field
of finite characteristic. The appropriate geometric object of study is now the category of parity sheaves, as defined in \cite{JMW}, and it also categorifies the Hecke algebra of $W$. On
the other hand, the reflection representation itself now factors through a finite dihedral quotient $W \to W_m$, and the algebraically-defined Soergel bimodule category only depends on the
representation (not the additional choices of roots and coroots). By Soergel's results, one expects $\SBim$ to categorify the Hecke algebra of $W_m$ instead. To give a morphism-theoretic
statement, there is an extra morphism $B_s \ot B_t \ot \ldots \to B_t \ot B_s \ot \ldots$ between the Bott-Samelson bimodules for the two reduced expressions of the longest element of
$W_m$; there is no corresponding morphism between parity sheaves. The category $\SBim$ is no longer the correct object of study for $W$; in general, the SCT could not possibly hold for a
realization of $W$ which factors through a (non-trivial) quotient Coxeter group.

However, the category $\DC$ (which depends on the realization, not just the representation) is a more natural object of study for degenerate realizations where $\BSBim$ does not behave
well. For crystallographic Coxeter groups, $\DC$ will be equivalent to the corresponding category of parity sheaves. The appropriate analog of the SCT will hold for $\DC$ even when it fails
for $\BSBim$. This is proven in the next paper \cite{EWGR4SB}, and is an important motivation for the diagrammatic approach. In less degenerate realizations where the categories $\DC$ and
$\BSBim$ are equivalent, and both satisfy the SCT, it can still be the case that the indecomposable Soergel bimodules do not descend to the KL basis. For instance, this occurs even for Weyl
groups when the realization is defined over a field of positive characteristic. The diagrammatic approach will allow one to study algebraically which indecomposables have the wrong size,
which idempotents are missing, and so forth (see \cite{WilPcan}).

On a different note, the author has also constructed a quantum version of the geometric Satake equivalence in type $\tilde{A}$, coming from a realization defined over $\Zqq$. The case
$\tilde{A}_1$ is discussed in this paper.

By pairing the simple roots and co-roots, one obtains the Cartan matrix of a realization (which need not have integer coefficients). Most familiar Cartan matrices over $\ZM$ have the
property that whenever $m_{st}$ is odd for two simple reflections $s,t \in S$, the corresponding Cartan entries $a_{st}$ and $a_{ts}$ are equal and negative (so that the angle between
simple roots is obtuse). To define $\DC$ in the utmost generality, one should also consider more degenerate, ``unbalanced" realizations. The author's quantum Satake equivalence will
require an unbalanced Cartan matrix over $\Zqq$ in order to study $\tilde{A_n}$ for $n \ge 2$ (although the case $n=1$ is balanced). Using unbalanced Cartan matrices adds a great deal of
bookkeeping and complexity; as such, we develop this theory in the Appendix. In the main body of the paper, we will only work with symmetric, balanced realizations.

For a dihedral group with a balanced symmetric Cartan matrix, we write the off-diagonal entry as $-\d$, living in an algebra $\Bbbk$ over the polynomial ring $\Z[\d] = \Z[q+q^{-1}]$. The
representation will factor through the finite dihedral group $W_m$ of size $2m$ precisely when the $m$-th quantum number, a polynomial in $\d$, vanishes. This is essentially the statement
that $q$ is a $2m$-th root of unity, with $q \ne \pm 1$. It will be a faithful representation of $W_m$ when $q$ is a primitive $2m$-th root of unity.

We devote a great deal of energy to defining $\DC$ and discussing $\SBim$ for arbitrary realizations of dihedral groups, and working in this natural level of generality. For this reason,
this paper is not the easiest introduction to Soergel theory in general, or to the diagrammatic style of the results we present. The novice should perhaps begin by reading about Soergel
bimodules in type $A$ in \cite{EKh}.

The literature about Soergel bimodules cares mostly about the reflection representation (or similar representations), and the interesting choice in this context is what field to define the
representation over. That is, the literature phrases its results in terms of assumptions about the base ring $\Bbbk$, such as its characteristic. The situation in this paper is different.
The properties of $\DC$ we discuss here will depend on conditions intrinsic to the realization as a whole, and not intrinsic to the choice of base ring $\Bbbk$. For example, one assumption
we will make is that the realization satisfies what we call Demazure Surjectivity, which is to say that pairing against each co-root (resp. root) is a surjective map from $\hg^*$ (resp.
$\hg$) to $\Bbbk$. Demazure surjectivity can always be achieved by enlarging $\hg$ and $\hg^*$ without changing the Cartan matrix; this is independent of the choice of base ring $\Bbbk$ or
its characteristic.

We never assume that $\Bbbk$ has a given characteristic, or even that it is a domain. We do not assume that the realization is faithful. However, we expect many readers to prefer domains or
faithful realizations, so we make various comments about the simplifications that occur under these assumptions, but they are not essential. We will need to assume that the realization is
even-balanced, a technical condition discussed in the Appendix. When we go beyond the results discussed in this introduction to define the 2-category $\DG_m$ for a finite dihedral group, we
will need to assume faithfulness of the realization, though this is not required for the monoidal category $\DC$. Whenever we discuss results which connect diagrammatic categories like
$\DC$ to algebraic categories like $\SBim$, we will need to make further assumptions on the realization in order that the algebraic setting (e.g. $\SBim$) is well-behaved.

Let us briefly mention the division of labor and ideas between this paper and its sequel \cite{EWGR4SB}. This paper is based on the author's PhD thesis \cite{EThesis}, which was mostly
concerned with faithful realizations where the Soergel conjecture holds. However, many of the technical results necessary to study more general realizations were already present in
\cite{EThesis}, needing only some reformatting and new terminology to become more generally relevant. Thus we have reformatted the paper, borrowing a lot of terminology from
\cite{EWGR4SB}, and expanding the background section slightly so that it contains all the results needed eventually by \cite{EWGR4SB}. The notion of a realization and the proper approach
to studying general realizations owe a great deal to the wisdom of G. Williamson. As in the original thesis, the approach taken here (to proving the SCT and the Soergel conjecture for
dihedral groups) will work only for some faithful realizations. In \cite{EWGR4SB}, additional technology is developed to prove the SCT in the correct generality.

It is also worth mentioning that, while this paper tackles the algebraic theory of realizations, this is not the end of the story. For example, one can not distinguish algebraically
between $q+q^{-1}$ and $u+u^{-1}$ for two primitive $2m$-th roots of unity. In a realization over $\RM$, however, the positivity properties of quantum numbers depend strongly on the choice
of primitive root of unity. These positivity properties will play a key role in \cite{EWHodge}, and are discussed further in that paper.

\subsection{{\bf Connections with Temperley-Lieb theory}}

The \emph{Temperley-Lieb category} $\mTL$ is a monoidal category governing the representations of the quantum group $U_q(\sl_2)$. It first appeared in \cite{TemLie}, and was used for the
study of subfactors by Jones in \cite{Jon2}. Most useful is the diagrammatic description given by Kauffman \cite{Kau}. Let $\d$ be an indeterminate. In Kauffman's description, the objects
are $n \in \N$, and the morphisms from $n$ to $m$ are the $\Z[\d]$-linear span of the set of $(n,m)$-\emph{crossingless matchings}. There are no morphisms unless $n$ and $m$ have the same
parity. The endomorphism ring of an object $n$ is known as the Temperley-Lieb algebra $TL_n$, and is a quotient of the Hecke algebra in type $A_{n-1}$. See section \ref{sec-TL} for more
details.

It is well-known that, after base change to $\Q(q)$ under the map $\d \mapsto q+q^{-1}$, the category $\mTL$ is equivalent to the full subcategory of $U_q(\sl_2)$-representations whose
objects are tensor powers of the standard representation $V$. Any indecomposable representation $V_n$ appears as a direct summand (with multiplicity one) inside $V^{\ot n}$, so that there
is a canonical idempotent \[ JW_n \in \End_{\mTL}(n) \ot \Q(q) = \End_{U_q(\sl_2)}(V^{\ot n}) \] which projects to this summand. This is called the \emph{Jones-Wenzl projector}
\cite{Jon3,Wenzl}, and it can be defined so long as the $n$-th quantum number and certain other quantum binomial coefficients are invertible. In any $\Z[\d]$-algebra $\Bbbk$ where all
quantum numbers are invertible, the Karoubi envelope of $\mTL \ot \Bbbk$ will be equivalent to (a $\Bbbk$-form of) the category of representations of $U_q(\sl_2)$.

One can draw an immediate analogy: we study an interesting category ($\SBim$ or $U_q(\sl_2)$-rep) by looking inside it at a subcategory ($\BSBim$ or $\mTL$) which admits a combinatorial
and diagrammatic description. We recover the original category by taking the Karoubi envelope. In fact, this analogy is almost perfect.

\begin{prop} \label{introTLprop} Let $W$ be the infinite dihedral group with simple reflections $\{s,t\}$, and let $\DC$ be defined as above for some realization over a $\Z[\d]$-algebra
$\Bbbk$. We will define a faithful $\Bbbk$-linear (non-monoidal) functor $\FC_s \co \mTL \to \BSBim$. It sends the object $n$ to $BS(\ul{\hat{n+1}}_s)$, defined to be the tensor product
$\cdots \ot B_s \ot B_t \ot B_s$ which alternates, ends with $s$, and has length $n+1$. This functor will be defined diagrammatically. There is a separate functor $\FC_t$ which reverses
the roles of $s$ and $t$, sending $n$ to $BS(\ul{\hat{n+1}}_t)$. Moreover, the following facts hold for morphisms in $\DC$: \begin{itemize} \item The graded vector space
$\Hom(BS(\ul{\hat{n}}_s),BS(\ul{\hat{k}}_s))$ is concentrated in non-negative degrees (for $n,k \ge 1$). When $n$ and $k$ share the same parity, every degree zero morphism is in the image
of the functor from $\mTL$. When $n$ and $k$ have different parities, there are no degree zero maps. \item The same holds with $s$ and $t$ reversed. \item The graded vector space
$\Hom(BS(\ul{\hat{n}}_s),BS(\ul{\hat{k}}_t))$ is concentrated in strictly positive degrees (for $n,k \ge 1$). \end{itemize} \end{prop}

The moral is that every degree zero morphism between color-alternating Bott-Samelson bimodules comes from the Temperley-Lieb category, and in particular so does every idempotent. Therefore,
the Jones-Wenzl projectors (when they exist) yield idempotents inside $\DC$, whose images are indecomposable. These images will be the Soergel bimodules $B_w$.

Proposition \ref{introTLprop} is awkwardly stated, using a pair of non-monoidal functors. This is only because we have avoided describing in this introduction the 2-categories which
underlie both sides of the story. In truth, we define a 2-functor $\FG$ from a two-colored version of the Temperley-Lieb category to the 2-category of Singular Soergel bimodules (or its
diagrammatic version $\DG$), which is fully faithful onto degree zero 2-morphisms.

The 2-functor $\FG$ is a (quantum) algebraic version of the geometric Satake equivalence. We mean that certain Singular Soergel bimodules for the infinite dihedral group correspond to a
($q$-analog of a) 2-category of equivariant perverse sheaves on loop group of $SL_2$ (or more precisely, on the Kac-Moody group for affine $SL_2$). Setting $q = 1$ or equivalently $\d =
2$, we recover the geometric Satake equivalence. We will not discuss this quantum algebraic Satake equivalence any further in this paper; see \cite{EQAGS} for more details, and for
generalizations in type $A$. However, we do study the 2-functor $\FG$ in detail.

Now consider a base ring $\Bbbk$ containing a primitive $2m$-th root of unity $q$, and make $\Bbbk$ a $\Z[\d]$-algebra via $\d \mapsto q + q^{-1} \in \Bbbk$. The Karoubi envelope of $\mTL
\ot \Bbbk$ is now equivalent to the category of tilting modules over a form of $U_q(\sl_2)$ at that root of unity. The $m$-th quantum number vanishes, and $JW_m$ is not well-defined. More
interestingly, $JW_{m-1}$ is well-defined, and is \emph{negligible}, i.e. it is in the kernel of a certain invariant form on $\mTL$. In fact, it generates the ideal of negligible morphisms.
It is also common to study the (Karoubi envelope of the) category $\mTL_{\negl}$ obtained by killing all negligible maps. This category is semisimple, and is equivalent to the fusion
category attached to $U_q(\sl_2)$ at a root of unity. Jones' original application of the Temperley-Lieb category to subfactors \cite{Jon2} also used the negligible quotient. For more on
killing negligible morphisms in general, see \cite[Chapter 2]{BarWes} and \cite{Tur}.

When $q$ is a primitive $2m$-th root of unity, the negligible Jones-Wenzl projector $JW_{m-1}$ is actually rotation-invariant. This fact, though fairly trivial, is crucial in this paper.
Rotation in the Temperley-Lieb algebra has been studied before (see \cite{Jon1,GraLehAffineTL}), but typically in the negligible quotient. Other Jones-Wenzl projectors $JW_k$ for $k \le
m-2$ are not rotation-invariant.

\begin{prop} Let $W_m$ be the finite dihedral group of size $2m$, and $\BSBim$ be its category of Bott-Samelson bimodules defined over a $\Z[\d]$-algebra $\Bbbk$ where $\d \mapsto
q+q^{-1}$ for $q$ a primitive $2m$-th root of unity. We can define a functor $\FC_s \co \mTL \ot \Bbbk \to \BSBim$ as before. The facts stated in Proposition \ref{introTLprop} hold for
alternating tensors $BS(\ul{\hat{n}}_s)$ of length $n \le m-1$. However, there is a new degree zero morphism $BS(\ul{\hat{m}}_s) \to BS(\ul{\hat{m}}_t)$ in $\BSBim$, and another in the
reverse direction. These are the $2m$-valent vertices mentioned previously, and they descend to inverse isomorphisms on the images of the respective Jones-Wenzl projectors $JW_{m-1}$.
These maps, in conjunction with the images of $\FC_s$ and $\FC_t$, generate all the morphisms of degree $0$ between alternating tensors of length $\le m$, and there are no negative degree
morphisms. \end{prop}

There is a nice slogan for this proposition, though it is mathematically nonsensical: singular Soergel bimodules for the finite dihedral group are obtained from the two-colored
Temperley-Lieb 2-category by adjoining ``square roots" of the negligible Jones-Wenzl projectors.

There is another slogan, equally nonsensical but more intriguing: (an extension of the) non-semisimplified representations of quantum $\sl_2$ at a $2m$-th root of unity are
Satake-equivalent to perverse sheaves on the flag variety of the finite dihedral group (though such a geometric object does not exist for $m \notin \{2,3,4,6\}$).

There is another slogan, which makes sense but is ill-fated nonetheless. The category $\mTL_{\negl}$ maps fully-faithfully (in degree zero) to the quotient of $\BSBim$ by the ideal of
morphisms which factor through $B_{w_0}$, where $w_0 \in W_m$ is the longest element. This categorifies the quotient of $\HB$ by the KL basis element $b_{w_0}$. The literature refers to
this quotient as the \emph{generalized Temperley-Lieb algebra} associated to the finite dihedral group (when $m>2$) \cite{GraThesis, GreTL}. This terminology is unfortunate for us, but
leads to the slogan: the Temperley-Lieb algebra categorifies the Temperley-Lieb algebra! More precisely, the (two-color) Temperley-Lieb algebra (in type $A$, at a root of unity, modulo
negligible morphisms) categorifies the (generalized) Temperley-Lieb algebra (of the dihedral group).

\subsection{Structure of the paper}

This paper is intended to be an omnibus of all things dihedral: a Dihedral Cathedral. We provide a reasonable level of detail, leaving some simple calculations to the reader. Slightly more
computational detail can be found in the author's thesis \cite{EThesis}, which contains some minor errors and uses slightly different conventions.

We assume little outside knowledge. An introduction to diagrammatics for 2-categories can be found in \cite{LauSL2}, chapter 4. An introduction to Karoubi envelopes can be found in
\cite{BarMor}. We draw upon \cite{EWFrob} heavily for general facts about Frobenius extensions, but that paper is quite short. References to the author's earlier work occur only when the
computation is simple enough to be left as an exercise.

Chapters \ref{sec-dihedral} through \ref{sec-TL} are background material with a lot of elaboration. In chapter \ref{sec-dihedral} we discuss presentations of the Hecke algebra and the
Hecke algebroid in terms of the Kazhdan-Lusztig generators, as well as numerous other features. We also fix some basic notation. Notable is section \ref{techniques} where we discuss
potential categorifications of the Hecke algebra, and some of the standard tricks played in categorification theory. In chapter \ref{sec-frobenius} we begin by discussing the
technicalities of realizations. Then we describe the Soergel bimodules and the Frobenius hypercube structure on invariant subrings of $R$. Chapter \ref{sec-TL} contains an introduction to
Jones-Wenzl idempotents and their analogs for the two-colored Temperley-Lieb category. Counting colored regions in a Jones-Wenzl projector will yield a polynomial which cuts out all the
reflection lines in $\hg$.

In sections \ref{sec-singinfty} and \ref{sec-mdti} we provide diagrammatics for (singular) Bott-Samelson bimodules for the case $m=\infty$. In sections \ref{sec-singfinite} and
\ref{sec-mdm} we provide the additional generators and relations for the case $m<\infty$. Sections \ref{sec-thick} and \ref{sec-TLcatsTL} are simple consequences, giving respectively a
diagrammatic presentation for the so-called generalized Bott-Samelson bimodules, and a categorification of the generalized Temperley-Lieb algebra of the dihedral group.

Finally, the appendix explains how to modify the constructions above to handle non-symmetric, unbalanced, and unfaithful realizations. We include enough detail to deal with this situation
for arbitrary Coxeter groups, not just dihedral groups. It is designed to be read in parallel with the corresponding sections in Chapter \ref{sec-frobenius} and \ref{sec-TL}. Most of the
work goes into defining the Frobenius hypercube structure on invariant subrings, when it exists. Once this is accomplished, the rest of the paper will apply almost verbatim.

\vspace{0.06in}

{\bf Acknowledgments.}

These results are repackaged from the author's PhD thesis.

The author would like to thank Noah Snyder and Geordie Williamson for fruitful conversations without number; Geordie Williamson and Nicolas Libedinsky for comments on an earlier version of
this paper; and Mikhail Khovanov for his unflagging support. The author was supported by NSF grants DMS-524460 and DMS-524124 and DMS-1103862.

%% file: DihedralAndHecke.tex
%

We refer the reader to \cite{Humphreys} and \cite{LusUnequal} for additional background information, and for the proofs of any uncited statements in this chapter.

\subsection{{\bf Notation for the dihedral group}} \label{introtocoxeter}

The \emph{infinite dihedral group} $W_\infty$ is the group freely generated by two involutions $s$ and $t$. It has a length function $\ell$ and a Bruhat order $\le$.

The words \emph{index} and \emph{color} refer to an element of the set of \emph{simple reflections} $S=\{s,t\}$. An \emph{expression} is a finite sequence of indices. Our convention is that an expression will be denoted
by an underlined symbol $\ul{w}$, and removing that underline indicates the corresponding element $w \in W_\infty$. We use shorthand for certain expressions of length $k \ge 0$:
\begin{equation} {}_s \ul{\hat{k}} = \ubr{sts\ldots}{k}, \qquad \qquad {}_t \ul{\hat{k}} = \ubr{tst\ldots}{k}. \label{eq:def-1k-2k} \end{equation} Such an expression will be called
\emph{alternating} when $k>0$. An expression beginning with $s$, such as ${}_s \ul{\hat{k}}$ for $k>0$, will be called \emph{left-$s$-aligned}. In similar fashion, we write
$\ul{\hat{k}}_s$ for the alternating length $k$ expression which is right-$s$-aligned. Without the underline, $\hat{k}_s$ represents the corresponding element in $W$. We write $e=\hat{0}_s
= \hat{0}_t$ for the identity of $W$.

For any integer $m \ge 2$, the \emph{finite dihedral group} $W_m$ is the quotient of $W_\infty$ by the relation \begin{equation} \hat{m}_s = \hat{m}_t. \label{eq:braid} \end{equation} It
is a finite group of size $2m$, and the longest element $\hat{m}_s = \hat{m}_t$ will also be denoted $w_0$.

In this paper, the letter $m$ will always be either $\infty$ or an integer in $\Z_{\ge 2}$, and will refer to (half) the size of the dihedral group $W=W_m$. Our conventions and notation will apply to infinite and finite dihedral groups alike.

The \emph{Poincare polynomial} $\tilde{\pi}(W)$ of a Coxeter group $W$ is $\sum_{w \in W} v^{2\ell(w)}$, an element of $\Z[[v]]$. For finite Coxeter groups, the \emph{balanced Poincare
polynomial} $\qW$ is $\frac{\tilde{\pi}(W)}{v^{\ell(w_0)}}$, an element of $\Z[v^\pm]$ which is invariant under flipping $v$ and $v^{-1}$. A \emph{parabolic subset} is a subset $J \subset
S$, and it is \emph{finitary} when the corresponding parabolic subgroup $W_J$ is finite. For $J$ finitary, we write $\qJ$ for the balanced Poincare polynomial of $W_J$, and $\ell(J)$ for
the length of the longest element $w_J \in W_J$. Note that $\qtwo$ is the balanced Poincare polynomial of any singleton.

By convention, Poincare polynomials like $\qJ$ will always use the variable $v$. The quantum numbers $[n]$ for $n \ge 0$ in this chapter will also use the variable $v$. In subsequent
chapters, quantum numbers $[n]$ will always use the variable $q$. (The variable $v$ is in the Grothendieck group, i.e. the Hecke algebra. The variable $q$ is a scalar in the
categorification.)

Any statement in this paper will hold with the ``colors reversed," that is, with $s$ and $t$ switched.

\subsection{{\bf The Hecke Algebra}} \label{introtohecke}

\subsubsection{{\bf Definitions}}

The Hecke algebra $\HB = \HB_m$ is a $\Zvv$-algebra with several useful presentations. The \emph{standard presentation} has generators $T_i$ for $i \in \{s,t\}$. For an expression
$\ul{w}=i_1 i_2 \ldots i_d$ we let $T_{\ul{w}}$ denote the product $T_{i_1} T_{i_2} \cdots T_{i_d}$. The relations are \begin{subequations} \begin{eqnarray} T_i^2 & = & (v^{-2} - 1) T_i +
v^{-2} 1, \label{tisq}\\ T_{\ul{\hat{m}}_s} & = & T_{\ul{\hat{m}}_t} \label{coxreln} \end{eqnarray} \end{subequations} This second relation is suppressed when $m=\infty$. We define \[T_w
\define T_{\ul{w}}\] whenever $\ul{w}$ is a reduced expression, and note that this does not depend on the choice of reduced expression. The identity of $\HB$ is $T_e$. These $T_w$, for $w
\in W$, form the \emph{standard basis} of $\HB$ as a free $\Zvv$-module. (A related basis is $H_w = v^{\ell(w)} T_w$, which we do not use in this paper, but use in \cite{EWGR4SB}.)

A $\Zvv$-linear map $\mu \colon \HB \to \Zvv$ satisfying $\mu(ab)=\mu(ba)$ is called a \emph{trace}. We also allow traces to take values in $\Z((v))$. One can show that the
map $\epsilon$ given by $\epsilon(T_w) = \delta_{w,1}$ is a trace, called the \emph{standard trace}.

The Hecke algebra also has a \emph{KL basis} $\{b_w\}$, which is defined implicitly as the unique basis satisfying certain criteria. For dihedral groups, the solution to these criteria is easy.

\begin{claim} \label{simpleKLbasis} For all $w \in W$, $b_w = v^{l(w)} \sum_{x \le w} T_x$. This holds for $m$ finite or infinite. \end{claim}

Clearly $\epsilon(b_w)=v^{\ell(w)}$. For $i \in \{s,t\}$, we call $b_i = v(T_i + 1)$ a \emph{KL generator}. When $W$ is finite, we have \[b_{w_0}=v^m \sum_{w \in W} T_w.\]

Let $\omega$ be the $v$-antilinear antiinvolution defined by $\omega(b_i)=b_i$ for $i \in \{s,t\}$. This allows one to define the \emph{standard pairing} on $\HB$ via $(x,y) \define
\epsilon(\omega(x)y)$. Conversely, $\epsilon(x)=(1,x)$. Note that $(b_i x,y) = (x,b_i y)$ and $(x b_i,y) = (x,y b_i)$, so that the KL generator $b_i$ is \emph{self-biadjoint} with respect to the standard pairing.

\begin{rmk} Arbitrary traces $\mu$ are in bijection with semi-linear pairings for which $b_i$ is self-biadjoint, by replacing $\epsilon$ with $\mu$ in the above formulas. Since a trace is
determined by its values at each $b_w$, the corresponding semi-linear pairing is determined by the values $(1,b_w)$. \end{rmk}

For an expression $\ul{w}=i_1 \ldots i_d$ we write $b_{\ul{w}}$ for the product $b_{i_1} \cdots b_{i_d}$. Note that $b_{\ul{w}} \ne b_w$ in general. We write $\omega(\ul{w})$ for the sequence in reverse, so that $b_{\omega(\ul{w})} = \omega(b_{\ul{w}})$. Clearly $b_{\ul{w}}$ is biadjoint to $b_{\omega(\ul{w})}$.

The KL generators do, in fact, generate $\HB$ as a $\Zvv$-algebra, according to the \emph{KL presentation}. The quadratic relation, analogous to \eqref{tisq}, is
\begin{equation} b_i^2 = \qtwo b_i. \label{bisq} \end{equation}
When $m=\infty$ this relation suffices. In the finite case there is one more relation, analogous to \eqref{coxreln}.  We shall give this relation below in \eqref{twocolorbreln}. 

\subsubsection{{\bf Three related recursions}}\label{threerecursions}

Remember that $b_{\hat{k}_s}$ denotes a KL basis element, while $b_{\ul{\hat{k}}_s}$ denotes a product of KL generators. The following formulas indicate how the KL generators act on the KL basis. The product of $b_{\hat{k}_s}$ with $b_s$ is relatively boring.
\begin{claim} \label{boringproduct} For $m \ge k \ge 1$ we have $b_{\hat{k}_s} b_s = (v+v^{-1}) b_{\hat{k}_s}$. However, $b_{\hat{0}_s} b_s = b_s$. \end{claim}
The product of $b_{\hat{k}_s}$ with $b_t$ is more interesting.
\begin{claim} \label{KLrecursion} For $m > k \ge 2$ we have $b_{\hat{k}_s} b_t = b_{\hat{k+1}_t} + b_{\hat{k-1}_t}$. However, $b_{\hat{1}_s} b_t = b_{\hat{2}_t}$. \end{claim}

From Claim \ref{KLrecursion} one could determine a recursive formula to express $b_{\ul{\hat{k}}_s}$ as a linear combination of $b_{\hat{n}_s}$ for $n \le k$. This same recursion appears in
several other places, and by no accident. Let $V=V_1$ denote the standard two-dimensional representation of $\mf{sl}_2$ (or its quantum analog), and let $V_n$ denote the $n+1$-dimensional
irreducible (assuming, for this motivational digression, that we are in the semisimple setting). The reader should compare Claim \ref{KLrecursion} with the following two claims.

\begin{claim} \label{sl2recursion} For $n \ge 1$ we have $V_n \ot V \cong V_{n+1} \oplus V_{n-1}$. However, $V_0 \ot V \cong V_1$. \end{claim}

\begin{claim} \label{qnumrecursion} For $n \ge 2$ we have $[n][2]=[n+1] + [n-1]$. However, $[1][2] = [2]$. \end{claim}

Therefore, the same combinatorics governs the decomposition of $V^{\ot k}$ into irreducibles as governs the decomposition of $b_{\ul{\hat{k+1}}_s}$ into KL basis elements. Tensoring with
$V$ is like multiplying by either $b_s$ or $b_t$, whichever is next in an alternating expression. The decomposition numbers are easily encoded in ``truncated Pascal triangles."

\begin{defn} \label{coeffs1} Let the integer $c^n_k$ be determined from the following table, which is populated by letting each entry be the sum of the one or two
entries diagonally below.
\igc{1}{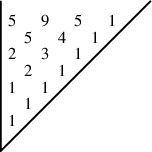}
Then $c^n_k$ is the entry in the $k$-th row and $n$-th column. By convention, $c^n_k = 0$ unless $0 < n \le k$ and $k-n$ is even (i.e. each row only has every other column). For example,
$c^1_1=1$, $c^2_2=1$, $c^1_3=c^3_3=1$, and $c^2_4=2$. The column $n=2$ consists of Catalan numbers.
\end{defn}

\begin{claim} \label{allrecursions1} \begin{itemize} \item For $1 \le k$ we have $V^{\ot (k-1)} \cong \oplus_{n} V_{n-1}^{\oplus c^n_k}$. \item For $1 \le k$ we have
$[2]^{k-1}= \sum_{n} c^n_k [n]$. \item  For $1 \le k \le m$ we have $b_{\ul{\hat{k}}_s} = \sum_{n} c^n_k b_{\hat{n}_s}$. \end{itemize} \end{claim}

\begin{example} $b_sb_tb_sb_tb_sb_t = b_{ststst} + 4 b_{stst} + 5 b_{st}$ when $m \ge 6$. \end{example}

Together with its color-reversed version, this claim covers all alternating expressions, giving two zigzag paths up the Bruhat chart. Note that this claim entirely ignores $b_e =
b_{\hat{0}_s} = b_{\hat{0}_t}$, which never appears in the decomposition of $b_{\ul{w}}$ for $\ul{w}$ nontrivial.

Now we give the inverse matrix.

\begin{defn} \label{coeffs2} Let the integer $d^n_k$ be determined from the following table (with the same conventions as before), which is populated by letting $d^n_k=d^{n-1}_{k-1} -
d^n_{k-2}$.

\igc{1}{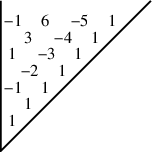} \end{defn}

\begin{claim} \label{allrecursions2} \begin{itemize} \item For $1 \le k$, in the Grothendieck group of $\mf{sl}_2$ representations, we have $[V_{k-1}] = \sum_{n} d^n_k
[V^{\ot (n-1)}]$. \item For $1 \le k$ we have $[j] = \sum_{n} d^n_j [2]^{n-1}$. \item For $1 \le k \le m$ we have $b_{\hat{k}_s} = \sum_{n} d^n_k
b_{\ul{\hat{n}}_s}$. \end{itemize} \end{claim}

\subsubsection{{\bf The finite case}}\label{mfiniterelations}

When $m < \infty$, we have $\hat{m}_s = \hat{m}_t = w_0$ so that Claim \ref{allrecursions2} gives two distinct formula for $b_{w_0}$ in terms of the KL generators. This gives an algebraic
relation on KL generators, which is the replacement for the braid relation \eqref{coxreln}.

\begin{equation} \label{twocolorbreln} \sum_{n} d^n_m b_{\ul{\hat{n}}_s} = b_{w_0} = \sum_{n} d^n_m b_{\ul{\hat{n}}_t}. \end{equation}

\begin{example} When $m=3$, $b_s b_t b_s - b_s = b_t b_s b_t - b_t$. \end{example}

We leave the reader to confirm that \eqref{bisq} and \eqref{twocolorbreln} give an alternate presentation for the Hecke algebra.  Finally, let us record two additional equalities.

\begin{equation} b_s b_{w_0} = b_{w_0} b_s = (v+v^{-1}) b_{w_0} \label{bibw0} \end{equation}
\begin{equation} b_{w_0} b_{w_0} = \qW b_{w_0} \label{bw0sq} \end{equation}

\subsection{{\bf Potential categorifications}}\label{techniques}

In this section we introduce one of the key tricks of the trade, which allows one to use the existence of objects categorifying the KL basis to deduce facts about the categorification.
This trick was used by Soergel for general Coxeter groups in \cite{Soe5} and elsewhere. To see these proofs in more detail (for an analogous case), see section 3.3 of \cite{ETemperley}. Let $\Bbbk$ be a commutative domain.

\begin{defn} Let $\CC$ be a $\Bbbk$-linear graded additive monoidal category, whose morphism spaces are finite rank over $\Bbbk$ in each degree. Let $(1)$ denote the grading shift. Suppose
it has objects $B_i$ for $i = s,t$, satisfying \begin{equation}\label{BiBi1} B_i \ot B_i \cong B_i(1) \oplus B_i(-1), \end{equation} and that $\CC$ is contained in the Karoubi envelope of
the subcategory monoidally generated by $B_s$ and $B_t$. Suppose that each $B_i$ is self-biadjoint, i.e. there are natural isomorphisms \[\Hom_\CC(B_i \ot M,N) \cong \Hom_{\CC}(M,B_i \ot N)
\; \textrm{ and } \; \Hom_\CC(M \ot B_i,N) \cong \Hom_{\CC}(M,N \ot B_i).\] Then we call $\CC$ a \emph{potential categorification} of $\HB_\infty$. \end{defn}

For a potential categorification, there is an obvious $\Zvv$-linear map from $\HB_\infty$ to the Grothendieck ring $[\Kar(\CC)]$ of the Karoubi envelope, sending $b_i \mapsto [B_i]$ and
$v$ to the grading shift.

We now specify what it means for a potential categorification of $\HB_\infty$ to factor through $\HB_m$ for $m<\infty$. Consider \eqref{twocolorbreln} as an equality of the two sides
(ignoring $b_{w_0}$ in the middle). One can transform this into an equality with only positive integer coefficients, by adding terms to both sides. Then, one can construct a ``categorified
version" of this relation, replacing the product $b_{\ul{\hat{n}}_s} = \cdots b_t b_s$ with the corresponding tensor product $\cdots \ot B_t \ot B_s$, and replacing the sum with the direct
sum.

\begin{example} When $m=3$, the categorified relation has the form $(B_s \ot B_t \ot B_s) \oplus B_t \cong (B_t \ot B_s \ot B_t) \oplus B_s$. \end{example}

\begin{defn} If $\CC$ is a potential categorification of $\HB_\infty$ and satisfies the categorified version of \eqref{twocolorbreln} for some $m<\infty$, we call $\CC$ a \emph{potential
categorification} of $\HB_m$ instead. \end{defn}

For a potential categorification of $\HB_m$, the map $\HB_\infty \to [\Kar(\CC)]$ clearly factors through the quotient $\HB_m$.

Let $\HB$ be the relevant Hecke algebra, either $\HB_\infty$ or $\HB_m$, and let $W$ be the corresponding Coxeter group. For an expression $\ul{w}=i_1 i_2 \ldots$, let $BS(\ul{w})$ denote
the corresponding tensor product $B_{i_1} \ot B_{i_2} \ot \cdots$, so that $BS(\emptyset)$ is the monoidal identity.

Any potential categorification $\CC$ induces a semi-linear pairing on $\HB$, via \[(b_{\ul{w}},b_{\ul{x}}) \mapsto \grdrk_{\Bbbk} \Hom_{\CC}(BS(\ul{w}),BS(\ul{x})).\] Here, $\grdrk$ denotes
the graded rank. We do not assume that Hom spaces are free as $\Bbbk$-modules, though their graded rank is still well-defined (say, as the maximal number of linearly independent vectors
over $\Bbbk$). In similar fashion, we could define a semi-linear pairing using the graded rank over $R$, where $R$ is any graded $\Bbbk$-algebra for which composition in $\CC$ is
$R$-linear. The elements $b_s$ and $b_t$ are self-biadjoint.

\begin{defn} We say that a potential categorification $\CC$ of $W$ satisfies the \emph{Soergel Categorification Theorem} or \emph{SCT} if the following properties hold for $\Kar(\CC)$:
\begin{itemize} \item For a reduced expression $\ul{w}$ there is a unique summand $B_{\ul{w}} \sumset BS(\ul{w})$ which does not appear in $BS(\ul{y})$ for any shorter expression $\ul{y}$.
\item For any two reduced expressions $\ul{w}$ and $\ul{w}'$ for the same element, there is a canonical morphism $BS(\ul{w}) \to BS(\ul{w}')$ which induces an isomorphism
$\phi_{\ul{w},\ul{w}'} \co B_{\ul{w}} \to B_{\ul{w}'}$. Moreover, these isomorphisms are compatible in the sense that $\phi_{\ul{w}',\ul{w}''} \circ \phi_{\ul{w},\ul{w}'} =
\phi_{\ul{w},\ul{w}''}$. Therefore, there is a single object $B_w$ which is canonically isomorphic to each summand $B_{\ul{w}}$, independent of the reduced expression for $w$. We will
never use the notation $B_{\ul{w}}$ again. \item The set $\{B_w\}_{w \in W}$ forms a complete list of non-isomorphic indecomposables in $\Kar(\CC)$, up to grading shift. \item The map $\HB
\to [\Kar(\CC)]$ is an isomorphism, and it induces the standard pairing on $\HB$. \end{itemize} If in addition one has $[B_w]=b_w$, we say that the \emph{(analog of the) Soergel
conjecture} is true for $\CC$. (Of course, Soergel made his conjecture about a very specific potential categorification of $\HB$, defined in characteristic zero; this terminology is not
meant to imply any claims on Soergel's behalf.) \end{defn}

\begin{lemma} \label{homspacelemma} Let $\CC$ be a potential categorification of $\HB$. Let $\epsilon$ be any trace map on $\HB$. Suppose that \begin{itemize} \item For each $w \in W$ there
is an object $B_w$ in $\Kar(\CC)$ for which $b_w \mapsto [B_w]$. The biadjoint of $B_w$ is $B_{w^{-1}}$. \item The categorified version of the relation in Claim \ref{allrecursions1} holds,
decomposing $BS(\ul{x})$ for a reduced expression into direct sums of various $B_w$. \item The Hom spaces $\Hom_{\CC}(B_e,B_w)$ are free $\Bbbk$-modules for all $w \in W$. (More generally,
we may assume they are free graded $R$-modules, for $R$ as above.) \item The graded rank of $\Hom_{\CC}(B_e,B_w)$ over $\Bbbk$ (resp. over $R$) is equal to $\epsilon(b_w)$. \end{itemize}
Then we may deduce that all Hom spaces between various objects $B_w$ and $BS(\ul{x})$ in $\CC$ are free as $\Bbbk$-modules (resp. as $R$-modules), and the semi-linear pairing induced from
the categorification agrees with that induced by $\epsilon$. \end{lemma}

\begin{proof} Using biadjointness and direct sum decompositions, we see that any Hom space between various $BS(\ul{x})$ or $B_w$ is isomorphic to a direct sum of Hom spaces
$\Hom(B_e,B_w)$ for various $B_w$. Therefore the freeness of $\Hom(B_e,B_w)$ implies the freeness of all Hom spaces. The combinatorics of biadjointness and decomposition in
$\CC$ are the same as the combinatorics of $\omega$ and the additive relations in $\HB$ when determining $(x,y)$, so that the final statements are obvious. \end{proof}

\begin{cor} \label{homspacecor} Suppose that $\CC$ satisfies the conditions of Lemma \ref{homspacelemma} for the standard trace $\epsilon$, as the graded rank of Hom spaces over a graded
ring $R$. We assume that $R$ is concentrated in non-negative degree, and consists of scalars $\Bbbk$ in degree $0$, and that $\Bbbk$ is a local ring. Then each object $B_w$ is
indecomposable and $\CC$ is already Karoubian. The category $\CC$ satisfies the SCT and the Soergel conjecture is true. \end{cor}

\begin{proof} Calculations with the bilinear form imply that \[\grdrk_R \Hom_{\CC}(B_w, B_x) = \delta_{w,x} + v \Z[v].\] This is sufficient to imply that $\{B_w\}$ form a list of pairwise
non-isomorphic indecomposable objects, up to shift. For instance, $B_x$ is indecomposable because its endomorphism ring must be local. Since every $BS(\ul{x})$ splits into these
indecomposables, our list must be complete. \end{proof}

\begin{cor} \label{fullyfaithfulfunctor} Suppose that $\CC$ and $\DC$ are two such categories as in Lemma \ref{homspacelemma}, and are both Karoubian. Suppose $\FC \colon \CC \to \DC$ is
an additive graded $R$-linear monoidal functor sending $B_i$ to $B_i$ (which implies that $B_w$ is sent to $B_w$). Suppose that $\FC$ induces isomorphisms of Hom spaces $\Hom(B_e,B_w)$
for all $w$, or alternatively of $\Hom(B_e,BS(\ul{w}))$ for every reduced expression $\ul{w}$. Then $\FC$ is an equivalence. \end{cor}

\begin{proof} Left to the reader. \end{proof}

\subsection{{\bf The Hecke Algebroid}}\label{heckealgebroid}

\subsubsection{{\bf Definitions}}

The Hecke algebroid is a general construction for Coxeter groups. See \cite{WilSSB} for more details.

It will be useful to distinguish between the index $s \in S$ and the parabolic subset $\{s\} \subset S$. Later in this paper we will be assigning a color to each index, blue to $s$ and red
to $t$. We use these colors to assign names to parabolic subsets of $S$: $b = \{s\}$ is blue, $r = \{t\}$ is red, $p = \{s,t\}$ is purple, and $\emptyset$ is white. Note that all parabolic
subsets are finitary, with the exception of $p$ in the case $m=\infty$. Given a parabolic subset $J$ we let $b_J \define b_{w_J}$.

The Hecke algebroid $\HG$ is a $\Zvv$-linear category with objects labelled by finitary parabolic subsets. The $\Zvv$-module $\Hom(J,K)$ is the intersection in $\HB$ of the
left ideal $\HB b_J$ with the right ideal $b_K \HB$. Composition from $\Hom(J,K) \times \Hom(L,J) \to \Hom(L,K)$ is denoted by $\star$, and is defined using renormalized multiplication:
\[x \star y = \frac{xy}{\qJ}. \] This makes sense, because we can write $x = x' b_J$ and $y = b_J y'$, so that $xy = \qJ x' b_J y'$.

It is clear that $b_J$ is the identity element of $\End(J)$. It is also clear that $\End(\emptyset) = \HB$ as an algebra. Whenever $J \subset K$, $b_K$ is in both the right and left ideal
of $b_J$, so there is an inclusion of ideals yielding $\Hom(K,L) \subset \Hom(J,L) \subset \HB$. This inclusion is realized by precomposition with $b_K \in \Hom(J,K)$. A similar statement
can be made about $\Hom(L,K) \subset \Hom(L,J)$ and postcomposition with $b_K \in \Hom(K,J)$.

\subsubsection{{\bf Presenting the Hecke algebroid as a quiver algebroid}}\label{presentinghecke}

Whenever $J \subset K$ we may view $b_K$ as an element both of $\Hom(J,K)$ and $\Hom(K,J)$. The collection of these morphisms for various $J \subset K$ will generate $\HG$. Moreover,
whenever $J \subset K \subset L$ it is clear that $b_L \star b_K = b_L$, as a composition $\Hom(K,L) \times \Hom(J,K) \to \Hom(J,L)$. Similarly $b_K \star b_L = b_L \in \Hom(L,J)$.
Therefore, $\HG$ is generated by the morphisms $b_K$ when $K \setminus J$ is a single index. We take these generators and view them as arrows in a path algebroid.
\[
\begin{tikzcd}
&b \arrow[leftrightarrow]{ld} \arrow[leftrightarrow]{rd} \\
\emptyset \arrow[leftrightarrow]{rd} && p \arrow[leftrightarrow]{ld} \\
&r
\end{tikzcd}
\]
When $m=\infty$ the parabolic subset $p$ is not finitary, so there are only 3 vertices and 2 doubled edges in this quiver.

We denote a path between parabolic subsets using an underline, analogous to our notation for expressions. By $b_{\ul{\emptyset bprpb}}$ we mean the morphism which follows the path from $b$
up to $p$ and eventually to $\emptyset$. For instance, $b_{\ul{r}}$ would be the identity morphism of $r$. Abusing notation, let $\ul{\hat{k}}_r$ denote the path $\ul{\cdots r \emptyset b
\emptyset r}$, which passes through $\emptyset$ exactly $k-1$ times, starting in $r$, and ending in either $r$ or $b$ depending on parity.

\begin{prop} The following relations on paths define the Hecke algebroid $\HG$ as a quiver algebroid. Only the first relation (together with its color switch) is needed for $m=\infty$.
\begin{subequations}
\begin{equation} \label{bisqsing} b_{\ul{r\emptyset r}} = (v+v^{-1}) b_{\ul{r}} \end{equation}
\begin{equation} \label{bw0sqsing} b_{\ul{prp}} = \frac{\qW}{v+v^{-1}} b_{\ul{p}} \end{equation}
\begin{equation} \label{upup} b_{\ul{\emptyset r p}}=b_{\ul{\emptyset b p}} \end{equation}
\begin{equation} \label{downdown} b_{\ul{p r \emptyset}}=b_{\ul{p b \emptyset}} \end{equation}
\begin{equation} \label{twocolorrelnsing} b_{\ul{i p r}} = \sum_{n} d^n_m b_{\ul{\hat{n}}_r} \end{equation}
\end{subequations}
In this final equation, $i$ is $r$ if $m$ is odd, and $b$ if $m$ is even.
\end{prop}

\begin{proof} Relation \eqref{bisqsing} follows from \eqref{bisq}, and relation \eqref{bw0sqsing} follows from \eqref{bw0sq}. Relations \eqref{upup} and \eqref{downdown} are obvious, since
both paths merely give $b_p$. Relation \eqref{twocolorrelnsing} follows from \eqref{twocolorbreln}. Thus all the relations do hold in $\HG$, and there is a map from
this quiver algebroid to $\HG$. It is easy to see that this map is surjective. Morphism spaces in $\HG$ are free $\Zvv$-modules of known rank, so it remains to find a spanning set for
paths in the quiver having the appropriate size. This is a simple exercise for the reader. \end{proof}

\subsubsection{{\bf Traces on the Hecke algebroid}}\label{hecketraces}

A \emph{trace} on an algebroid (over $\Zvv$) is a $\Zvv$-linear map $\epsilon_X \colon \End(X) \to \Zvv$ for each object $X$, such that $\epsilon_X(ab)=\epsilon_Y(ba)$ whenever $a \in
\Hom(Y,X)$ and $b \in \Hom(X,Y)$.

\begin{claim} Any trace on the Hecke algebroid is determined by $\epsilon_{\emptyset}$.\label{tracedetermined} \end{claim}

\begin{proof} Because of its defining property, the trace of an endomorphism (of any object) which can be expressed as a path going through $\emptyset$ is determined by
$\epsilon_{\emptyset}$. The identity of every object in $\HG$ is given (up to a scalar) by a path through $\emptyset$, and so the same is true for any endomorphism. \end{proof}

It is not hard to show that the standard trace on $\HB$ extends to a trace on $\HG$, also called the \emph{standard trace}.

One can construct a notion of a potential categorification of $\HG$, analogous to that found in section \ref{techniques}. The endomorphism category of the $\emptyset$ object will be a
potential categorification of $\End_\HG(\emptyset) = \HB$. One can state properties analogous to the SCT and the Soergel conjecture. The upshot of Claim \ref{tracedetermined} is that, in
line with Corollary \ref{fullyfaithfulfunctor}, a map of potential categorifications of $\HG$ is an equivalence if it induces an equivalence of potential categorifications of $\HB$. We
will not bother to formalize these arguments now; they will be put into practice in section \ref{indecomps}.

\subsubsection{{\bf Induced trivial representations}}\label{inducedreps}

The \emph{(left) trivial representation} $\TM_W$ of $\HB$ is the free rank one $\Zvv$-module where $b_s$ and $b_t$ both act by the scalar $\qtwo$. When $W$ is finite, this can be embedded
inside the regular representation as the left ideal of $b_p$. Suppose that $J$ is a finitary parabolic subset, and let $\HB_J \subset \HB$ be its Hecke algebra. Less obviously, the
$\HB$-module $\Ind(\TM_J)$ is embedded inside $\HB$ as the left ideal of $b_J$. Inside $\HG$, this is precisely the Hom space $\Hom(J,\emptyset)$, as a module over $\End(\emptyset)=\HB$.
Similar statements can be made about right modules and $\Hom(\emptyset,J)$.

%% file: SoergelCatfn.tex
\subsection{{\bf Quantum Numbers}}\label{qnum}

Let $[2]_q = q + q^{-1} \in \Zqq$, and more generally, $[m]_q = \frac{q^m - q^{-m}}{q-q^{-1}}$ for $m \in \Z_{\ge 0}$. Thus $[1]_q = 1$ and $[0]_q = 0$. We typically omit the subscript. When $q = e^{i \theta}$, we have $[2]^2_q = 4 \cos^2 \theta$, a familiar value from the usual theory of Cartan matrices.

We are interested in quantum numbers as algebraic integers, not in terms of their original expression as polynomials in $q$. Let $\delta$ be an indeterminate, which will play the role of
$[2]$. The computations in this section take place within the ring $\Z[\delta]$, which can be thought of as a subring of $\Zqq$ under the specialization $\delta \mapsto [2]$. Every quantum
number can be expressed as a polynomial in $\delta$, as was demonstrated in Claim \ref{allrecursions2}. A $\Z[\delta]$-algebra is an algebra $\Bbbk$ with a distinguished element $\d \in
\Bbbk$, which we will also denote $[2]$. For such an algebra, the elements $[m] \in \Bbbk$ are also well-defined.

When $n$ divides $m$, $[n]$ divides $[m]$ in $\Z[\delta]$. When $m$ is odd, $[m]$ is equal to some even polynomial in $[2]$. When $m$ is even, $[m]$ is equal to some odd polynomial in $[2]$. For $m \ge 3$ there is a minimal polynomial $Q_m \in \Z[z]$ such that $Q_m([2]^2)$ divides $[m]$ but not $[n]$ for any $n<m$.

\begin{example} \[\begin{array}{ccc} m=3: & [3] = [2]^2 - 1, & Q_3=z-1. \\
	m=4: & [4] = [2]^3 - 2 [2], & Q_4 = z-2. \\
	m=5: & [5] = [2]^4 - 3 [2]^2 + 1, & Q_5 = z^2 - 3z + 1. \\
	m=6: & [6] = [2]^5 - 4 [2]^3 + 3[2], & Q_6 = z-3. \end{array}\] \end{example}

Let us discuss the algebraic conditions on $\Z[\delta]$ which correspond to the specialization of $q^2$ to a root of unity. We write $\z_n$ for an arbitrary primitive $n$-th root of unity,
viewed as an algebraic integer. In other words, for any element $x$ in an arbitrary ring, we say that $x = \z_n$ if $x$ is a root of the $n$-th cyclotomic polynomial. For any $\Zqq$-algebra
and $m \ge 3$ we have \[q^2 = \z_m \iff Q_m([2]_q^2)=0 \iff [m]_q = 0 \textrm{ and } [k]_q \ne 0 \textrm{ for } 0<k<m.\] The case $m=2$ is special, in that $q^2 = \z_2 \iff [2]=0$, which is
an equation in $[2]$ not in $[2]^2$ (this will cause some issues in the appendix). Once again, we are not interested in viewing quantum numbers as polynomials in $q$, but this discussion
should instead serve to justify why one might wish to consider the algebraic conditions $[m]=0$ and $[k] \ne 0$ for $k<m$.

We now discuss some of the algebraic implications of the fact that $[m]=0$. It is easy to deduce (say, using the same recursive formulae in reverse) that whenever $[m]=0$, one has $[m-k] =
[m-1] [k]$. Therefore $[m-1]^2 = 1$. When $m$ is odd, $[2]$ divides $[m-1]$ and is therefore also invertible. Similarly, one has $[m+k] = [m+1][k]$, so that $[2m-1] = -1$ and $[2m]=0$.
Unfortunately, the converse is less pretty.

\begin{claim} \label{claim-qnumber-bullshit} Suppose that $[2m]=0$ and $[2m-1]=-1$. Then $2[m]=0$ and $[2][m]=0$. If $m$ is odd then $[m]=0$. However, allowing for 2-torsion, it is
possible that $[m] \ne 0$ when $m$ is even. \end{claim}

\begin{proof} That $2[m]=0$ and $[2][m]=0$ both follow from $[2m-k] = -[k]$. One can deduce from $[2m-k]=-[k]$ that $[m-1]^2=1$. If $m$ is odd then $[2]$ is invertible, so that
$[m]=0$. The ring $\ZM[\d]/(2\d,\d^2)$ provides an example where $[4]=0$ and $[3]=-1$ but $[2] \ne 0$. \end{proof}

We now see an essential difference between the even and odd cases. Suppose that $[m]=0$ and $[k] \ne 0$ for $k<m$, and that $\Bbbk$ is a domain. When $m$ is even, one can use the above
claim to deduce that $[m-1]=1$. However, when $m$ is odd, both $[m-1]=1$ and $[m-1]=-1$ are possible. In fact, this investigation into the algebraic conditions on $[m-1]$ is also inspired
by roots of unity; now one considers the implications of setting $q$ (rather than $q^2$) to a root of unity. It is easy to observe that $[m-1]=1$ when $q=\z_{2m}$ and $[m-1]=-1$ when
$q=\z_m$. When $m$ is even, $q^2 = \z_m$ already implies $q = \z_{2m}$ (also, $\z_{2m}=-\z_m$ and the two are indistinguishable in characteristic 2), as one might expect from the above
discussion. When $m$ is odd there are two distinct possibilities. Note also that when $m$ is odd, there is a splitting $Q_m(\d^2) = P_m(\d) P_m(-\d)$ for some polynomial $P_m \in \Z[\d]$,
with $q = \z_{2m} \iff P_m([2]_q)=0$. This polynomial $P_m$ determines the value of $[m-1]$.

Now for one final aside. When $m$ is odd and $[m]=0$, all quantum numbers are actually generated by $[2]^2$ and the unit $[m-1]$, since $[2]= [m-1][m-2]$ and $m-2$ is an even polynomial in
$[2]$. When $m$ is even and $[m]=0$, there is no guarantee that $[2]$ is invertible or that the ideal of $[2]^2$ contains $[2]$.

\subsection{{\bf Realizations}}\label{reflrep}

Fix a Coxeter system $(W,S)$. As before, elements of $S$ will be called \emph{indices} or \emph{colors}. The following definitions are taken from joint work with G. Williamson.

\begin{defn} \label{defnsymrealization} Let $\Bbbk$ be a commutative ring. A \emph{symmetric realization} of $(W,S)$ over $\Bbbk$ is a free, finite rank $\Bbbk$-module $\hg$, together with
subsets $\{ \a_s^\vee \; | \; s \in S\} \subset \hg$ and $\{ \a_s \; | \; s \in S\} \subset \hg^* = \Hom_{\Bbbk}(\hg,\Bbbk)$ called \emph{simple co-roots and simple roots}. The
\emph{Cartan matrix} $A=(a_{s,t}) _{s,t \in S}$ of a realization is defined by $a_{s,t}=\langle \alpha_s^\vee, \alpha_t \rangle \in \Bbbk$. This data must satisfy: \begin{enumerate} \item
$a_{s,s} = 2$ for all $s \in S$; \item $a_{s,t} = a_{t,s}$ for all $s,t \in S$; \item for any $s,t \in S$ with $m_{st}<\infty$, if $\Bbbk$ is given a $\Z[\d]$-algebra structure by the map $\d \mapsto -a_{s,t}$ (i.e. where $[2]
= -a_{s,t}$), then $[m_{st}]=0$; \item the assignment $s(v) \define v - \langle v, \alpha_s\rangle \alpha_s^{\vee}$ for all $v \in \hg$ yields a representation of $W$. \end{enumerate} We
will often refer to $\hg$ as a realization, however the choice of $\{ \alpha_s^\vee \}$ and $\{\alpha_s \}$ is always implicit. \end{defn}

\begin{example} Let $\Bbbk = \RM$ and let $\hg$ be spanned by $\{\a_s^\vee\}$ for $s \in S$. Let $\a_s$ be defined by the formula $a_{s,s}=2$ and $a_{s,t} = - 2 \cos(\frac{\pi}{m_{st}})$
for $s \ne t$, with the convention that $a_{s,t} = -2$ when $m_{st}=\infty$. This is the \emph{reflection representation} of $(W,S)$, as defined in Humphreys \cite{Humphreys}. The reader
uninterested in general realizations is welcome to use this realization by default. \end{example}

There is a contragredient action of $W$ on $\hg^*$. It is given on the span of the simple roots by the formula \[s(\a_t) = \a_t - a_{s,t} \a_s.\]

Note that the Cartan matrix need not determine the realization, since we do not assume that $\{ \a_s^\vee \}$ spans $\hg$. We also have not assumed that the simple roots or simple co-roots
are linearly independent. There are many degenerate possibilities which this definition permits. In characteristic 2 one might have that $\a_s=0$, or that $s$ acts trivially on $\hg^*$.
When $\a_s$ and $\a_t$ are collinear, one could have $s=t$ when acting on $\hg^*$ (one would also require $a_{s,t}= \pm 2$). We will soon make assumptions which eliminate some of these
degenerate possibilities. Given a realization over $\Bbbk$ and a homomorphism $\Bbbk \to \Bbbk'$ we obtain a realization over $\Bbbk'$ by base change; this can easily change the kernel of
the $W$-action on $\hg^*$.

Whenever a pair $s, t \in S$ is understood, we give $\Bbbk$ the structure of a $\Z[\d]$-algebra with $-[2] = a_{s,t}$.

Let us discuss the relationship between conditions (3) and (4).

\begin{claim} \label{repofW} Suppose that $\a_s$ and $\a_t$ are not collinear. The action of $s$ and $t$ preserves the space spanned by $\a_s$ and $\a_t$ in $\hg^*$. For any $m \ge 2$, the
action of $(st)^m$ on this span is trivial if and only if $[2m]=0$ and $[2m-1]=-1$. \end{claim}

\begin{proof} It is not difficult to show inductively that (in the basis $\{\a_s,\a_t\}$) we have \begin{equation} \label{stformula} (st)^k = \left( \begin{array}{rr} {[}2k+1] & -[2k] \\
{[}2k] & -[2k-1] \end{array} \right). \end{equation} In order for this matrix to be the identity for $k=m$, one requires $[2m]=0$ and $[2m-1]=-1$. \end{proof}

As discussed in Claim \ref{claim-qnumber-bullshit}, the fact that $[2m]=0$ and $[2m-1]=-1$ follows from and is usually equivalent to $[m]=0$, but fails to imply $[m]=0$ in certain cases.
If $\Bbbk$ is a domain or $m$ is odd then $[m]=0$.

Moreover, consider the action of $(st)$ on the span of $\{\a_s,\a_t,\a_u\}$. A similar computation shows that \begin{equation} \label{stonuformula} (st)^k(\a_u) = \a_u + ([k]^2 a_{s,u} +
[k][k+1] a_{t,u}) \a_s + ([k][k-1] a_{s,u} + [k]^2 a_{t,u}) \a_t. \end{equation} For $(st)^m$ to be trivial $[m]=0$ clearly suffices. When $[m] \ne 0$, the condition that $(st)^m$ is
trivial is rather restrictive, implying that $[m] a_{t,u}= [m] a_{s,u}= 2[m] = [2][m]=0$. However, extremely degenerate possibilities do exist.

Thus, condition (3) implies that $st$ has order dividing $m_{st}$ when acting on the span of the roots. However, this does not imply condition (4), because $\hg$ may be larger than this span. In the other direction, if $\Bbbk$ is a domain, or $m_{st}$ is odd for each pair $s,t$, then (4) implies (3). However, (3) and (4) are logically independent in general.

The definition of a non-symmetric realization will be postponed to the Appendix, though we will still discuss them briefly in the main text. One allows for the possibility $a_{s,t} \ne
a_{t,s}$. The only subtlety will be altering the condition that $[m_{st}]=0$.

\begin{defn} We call a realization \emph{faithful in rank 2} if, in the action on $\hg$ or $\hg^*$, the order of each $s \in S$ is 2, and for each $s,t \in S$ the order of $st$ is
$m_{st}$. This is the only property of realizations we will consider in this paper which is not preserved under base change.

Remember that when $[m]=0$ one has $[m-1]^2 = 1$. We call a symmetric realization \emph{balanced symmetric} if for each $s,t \in S$ with $m_{st}<\infty$ one has $[m_{st}-1]=1$. Refining
this notion, we call a symmetric realization \emph{even-unbalanced} (resp. \emph{odd-unbalanced}) if there is some $s,t \in S$ with $m_{st}$ even (resp. odd) and $[m_{st}-1] \ne 1$. We
call it \emph{even-balanced} (resp. \emph{odd-balanced}) otherwise. \end{defn}

The definition of a balanced non-symmetric realization will be postponed to the Appendix. Note that being symmetric, being balanced, and being faithful in rank 2 are all properties
determined by the action of dihedral subgroups. Being symmetric (resp. balanced) is preserved under base change of the ring $\Bbbk$, though being faithful in rank 2 is not. For the rest of
this paper, we abusively write \emph{faithful} to indicate faithful in rank 2.

\begin{remark} \label{rmk-balanced} We warn the reader now that the behavior of even-balanced and odd-balanced realizations are drastically different! This theme will recur in many of the
computations below. As an overarching principle, realizations that are odd-unbalanced are acceptable; one can work with them (as we do in the appendix), only they require more bookkeeping
than completely balanced realizations. However, realizations that are even-unbalanced are a nightmare, and we typically rule them out. The most important difference will be addressed in
section \ref{rouandrotation}. Due to the calculations of the previous section, the assumption that a realization is faithful and that it is even-balanced are very similar, though logically independent. \end{remark}

Let us give some examples where $\hg$ is spanned by $\a_s^{\vee}$ and $\a_t^{\vee}$, so that the realization is determined by a $2 \times 2$ Cartan matrix.

\begin{example} Suppose that $\Bbbk = \ZM$, with $a_{s,t}=-3$ and $a_{t,s}=-1$. This is a faithful non-symmetric realization of the dihedral group $W_6$. It is balanced, because
$[5]=1$ when $[2]^2 = a_{s,t} a_{t,s} = 3$. It can also be viewed as a non-faithful realization of $W_{6k}$ for any $k$, or of $W_\infty$. It is balanced for $W_{6k}$ precisely when
$k$ is odd, because $[6k-1] = (-1)^{k+1}$. \end{example}

\begin{example} Take the same Cartan matrix as the previous example, but base change to $\Bbbk = \FM_2$. Now one has a symmetric balanced realization of $W_6$ which is not faithful,
factoring instead through $W_3$. In characteristic 2, any symmetric realization is balanced, because $[m_{st}-1] = \pm 1$. \end{example}

\begin{example} Now suppose that $\Bbbk = \ZM$, with $a_{s,t} = a_{t,s} = -1$. This is a faithful symmetric balanced realization of $W_3$. It can also be viewed as a symmetric realization
of $W_6$ which, unlike the previous example, is not balanced since $[5]=-1$. \end{example}

\begin{example} Suppose that $a_{s,t}=a_{t,s} =-(q+q^{-1})$ in $\CM$, for $q$ is a primitive $m$-th root of unity with $m$ odd. Then $[km]=0$ and $[km-1]=-1$ for all $k \ge 0$, and $\hg$
will not be a balanced realization of $W_{mk}$ for any $k$. \end{example}

Given any realization and a choice of invertible scalars $\l_s \in \Bbbk$ for each $s \in S$, one obtains a new realization by \emph{root-rescaling}: rescaling $\a_s \mapsto \l_s \a_s$ and
$\a_s^\vee \mapsto \l_s^{-1} \a_s^\vee$. This amounts to conjugating $A$ by a diagonal matrix. This procedure will not preserve symmetric realizations unless $\l_s = \pm 1$.

In the main body of this paper, we will only consider balanced symmetric realizations. Root-rescaling will rarely preserve balanced symmetric realizations, and not all realizations are
balanceable or symmetrizable by root-rescaling. For a treatment of non-symmetric or non-balanced realizations, see the Appendix. When discussing Soergel bimodules, one will have to make
the additional assumption that the realization is faithful.

The \emph{universal} balanced symmetric realization of the infinite dihedral group is defined with $\Bbbk=\Z[\d]$, letting $\hg$ be the span of $\a^\vee_s$ and $\a^\vee_t$, and setting $a_{s,t}=-\d$. The \emph{universal} balanced symmetric realization of the finite dihedral group is defined analogously with $\Bbbk = \Z[\d] / ([m]=0, [m-1]=1)$.

\subsection{{\bf Assumptions on the realization}}\label{assumptions}

There is an unwritten assumption on the base ring $\Bbbk$, arising from the fact that there exists a (symmetric) realization defined over $\Bbbk$. Namely, it must contain an algebraic
integer $[2]$ for which $[m]=0$, for any $m = m_{st} < \infty$. For instance, if $m_{st}=5$ then $\Bbbk$ must contain the golden ratio.

We will let $\pa_s$ denote the map $\hg^* \to \Bbbk$ given by evaluation at $\a_s^\vee$. We have $\pa_s(\a_t)=a_{s,t}$. One can see that $f-s(f) = \pa_s(f) \a_s$ is collinear with $\a_s$.
Clearly $\pa_s(f)=0$ implies that $f$ is $s$-invariant. So long as $\a_s \ne 0$, $f$ is $s$-invariant if and only if $\pa_s(f)=0$. However, it is possible in characteristic 2 that
$\a_s=0$, in which case everything is $s$-invariant.

The most important assumption one can make about a realization is the following.

\begin{assumption} (Demazure surjectivity) \label{ass:Demazure Surjectivity} The map $\a_s \co \hg \to \Bbbk$ is surjective, for all $s \in S$, and the map $\pa_s \co \hg^* \to \Bbbk$ is
surjective, for all $s \in S$. \end{assumption}

This forbids the possibility that $\a_s = 0$. It guarantees the existence of some $\a \in \hg^*$ with $\pa_s(\a)=1$. Then for any $f \in \hg^*$ we see that $f - \pa_s(f) \a$ is
$s$-invariant. Moreover, $\a$ is not $s$-invariant.

Demazure surjectivity will be essential to most of the arguments in this paper. It does not hold for the universal realization of the infinite dihedral group, because $2$ and $[2]$ do not
generate the unit ideal in $\Z[\d]$. It does hold for $s$ and $t$ whenever $m_{st}$ is odd, because $[2]$ is invertible. The easiest way to ensure that Demazure surjectivity holds is to
invert $2$. However, Assumption \ref{ass:Demazure Surjectivity} can hold even when the Cartan matrix has zero columns (such as when $m_{st}=2,4$ in characteristic two), because $\hg$ need
not be spanned by the simple roots. In fact, by enlarging $\hg$ and $\hg^*$ and adjusting $\a_s$ and $\a_s^\vee$ accordingly, one can always create a realization with the same Cartan
matrix for which Demazure surjectivity does hold.

\begin{assumption} (Local non-degeneracy) \label{itsinvertible} Whenever $m_{st} < \infty$, $4-a_{s,t} a_{t,s}$ is invertible in $\Bbbk$. \end{assumption}

Note that local non-degeneracy implies Demazure surjectivity for both $s$ and $t$, when $m_{st}$ is finite.

\begin{assumption} (Lesser invertibility) \label{allinvertible} Choose any dihedral parabolic subset $\{s,t\}$. For all $k < m_{st}$, $[k]$ is invertible in $\Bbbk$. Moreover, the
realization is faithful. \end{assumption}

The fact that $[k]$ is invertible for $k<m_{st}$ implies faithfulness except in degenerate situations (e.g. $\a_s$ and $\a_t$ are collinear).

All of these assumptions are preserved under base change.

\subsection{{\bf Positive Roots for dihedral groups}}\label{roots}

For the reflection representation there is a notion of positive roots (for any Coxeter group). When dealing with a more general realization, one can still define an analogous multi-set of
\emph{positive roots}, using the same formulae. Because the realization need not be faithful, these roots may overlap.

Let $f_{s,k} = \hat{k}_t(\a_s) \in \hg^*$, for $k \ge 0$. A formula for $f_{s,k}$ can be deduced from the proof of Claim \ref{repofW}. Let $\LC_s$ be the set consisting of $f_{s,k}$ for $k
\ge 0$.

When $m = \infty$ and the realization is faithful, every $f_{s,l}$ is distinct from every other, and from every $f_{t,k}$. Regardless, we let $\LC$ be the multi-set union of $\LC_s$ and
$\LC_t$ when $m = \infty$.

When $m = 2k$ is even, we have that $f_{s,k-1} = f_{s,k}$ and $f_{t,k-1} = f_{t,k}$. Let $\LC = \{ f_{s,l},f_{t,l} \}_{0 \le l \le k-1}$.

When $m$ is odd, we have $f_{s,m-1} = [m-1] \a_t$. For a balanced realization therefore we have $f_{s,m-1-l} = f_{t,l}$ for all $0 \le l \le m-1$. Let $\LC = \{ f_{s,l} \}_{0 \le l \le
m-1}$.

\begin{remark} For an odd-unbalanced realization there are ambiguities of scalar when defining the positive roots, coming from the invertible factor $[m-1]$. The choice of positive roots
cannot be made canonically from the simple roots. See the Appendix for more details. \end{remark}

\subsection{{\bf Frobenius extensions}}\label{frobenius}

For more background on Frobenius extensions, see \cite{EWFrob}.

\begin{defn} A \emph{(commutative) Frobenius extension} is an inclusion $A \subset B$ of commutative rings where $B$ is a free finite-rank $A$-module, equipped with an $A$-linear
\emph{trace} map $\pa^B_A \co B \to A$ such that the pairing $(f,g) \mapsto \pa^B_A(fg)$ is perfect. In other words, there exists a basis $\{f_i\}$ and a dual basis $\{f_i^*\}$ of $B$
over $A$ such that $\pa^B_A(f_i f_j^*) = \delta_{i,j} 1_A$. \end{defn}

Whenever one has a Frobenius extension, one has four canonical bimodule maps: the inclusion $\iota^B_A \co A \to B$ and the trace $\pa^B_A \co B \to A$ of $A$-bimodules; and the multiplication $\mu^B_A \co B \ot_A B \to B$ and comultiplication $\Delta^B_A \co B \to B \ot_A B$ of $B$-bimodules. The comultiplication satisfies $\Delta^B_A(1) = \sum_i f_i \ot
f_i^*$, and this sum is independent of the choice of dual bases. These four maps are the units and counits for the biadjunction of $\Ind^B_A$ with $\Res^B_A$.

Let $\LM^B_A = \mu^B_A \Delta^B_A(1) = \sum_i f_i f_i^* \in B$. We may also use Sweedler notation, so that $\LM^B_A = \Delta^B_{A\ (1)} \Delta^B_{A\ (2)}$. It is clear that
$\pa^B_A(\LM^B_A) = n 1_A$, where $n$ is the rank of $B$ as an $A$-module.

A Frobenius extension can be \emph{graded}, in the sense that $A,B$ are graded rings, $\pa^B_A$ has degree $-2\ell$, and there exist homogeneous dual bases. We call $\ell$ the
\emph{degree} of the extension. In this case, $\Ind$ is shifted-biadjoint to $\Res(\ell)$, in that the right and left adjoints of $\Ind$ are isomorphic to $\Res(\ell)$ after
shifting by $\pm \ell$. After the appropriate grading shifts, $\pa$ and $\iota$ are maps of degree $-\ell$, and $\mu$ and $\Delta$ are degree $+\ell$.

\begin{defn} If $(A \subset B,\pa^B_A)$ and $(B \subset C,\pa^C_B)$ are Frobenius extensions, then $(A \subset C,\pa^B_A \circ \pa^C_B)$ is a Frobenius extension. We say that this chain of
Frobenius extensions is \emph{compatible}, because $\pa^C_A = \pa^B_A \pa^C_B$ and $\iota^C_A = \iota^C_B \iota^B_A$. A more complicated system of Frobenius extensions is called
\emph{compatible} if every subchain is compatible. \end{defn}

Let $R = \textrm{Sym}(\hg^*)$ denote the polynomial ring of the realization. It is a graded ring, with $\deg \a_s = \deg \hg^* = 2$, and it admits a
homogeneous action of $W$. For a parabolic subset $J$ let $R^J$ denote the ring of invariants under $W_J$.

Our goal for the remainder of this chapter is to give the collection of rings $R^J$ for finitary $J$ the structure of a (compatible) \emph{Frobenius hypercube}. The subsets of $S$ form a
hypercube and a poset, though we only consider the subposet of finitary subsets. Recall that $\ell(I)$ denotes the length of the longest element of $W_I$. For each edge $J \subset I = J
\coprod \{j\}$ there is a Frobenius extension $R^I \subset R^J$ of degree $\ell(I)-\ell(J)$. We denote the corresponding bimodule maps by $\Delta^J_I$ instead of $\Delta^{R^J}_{R^I}$ (and
similarly for $\mu$, $\iota$, $\pa$, and $\LM$). We omit $I = \emptyset$ from the notation, so that $\LM_I$ refers to the product-coproduct for the extension $R^I \subset R$. We omit set
notation when it is obvious: $R^s$ instead of $R^{\{s\}}$, $R^{s,t}$ instead of $R^{\{s,t\}}$.

The overall compatibility of this system depends on the compatibility over each face of the hypercube. Compatible hypercubes of Frobenius extensions were studied in \cite{EWFrob}, where
various general facts were proven. For instance, $\LM^C_A = \LM^C_B \LM^B_A$ for any subchain. In our case, the polynomial $\LM_I$ will be the product of the positive roots for $W_I$, and
therefore $\LM^J_I$ will be the product of the roots for $W_I$ which are not roots for $W_J$.

However, this Frobenius hypercube structure only exists under certain assumptions on the realization. As an example, for any $s \in S$ the map $\pa_s$ will extend to a map $R \to R^s$
giving the trace for the extension $R^s \subset R$. Clearly Demazure surjectivity is required in order for this extension to be Frobenius, because any Frobenius trace is surjective.

\begin{remark} When $I$ is not finitary, the extension $R^I \subset R$ is not Frobenius, or even finite. The two rings have difference transcendence degrees. \end{remark}

\begin{remark} The action of $W$ on $\hg$ or $\hg^*$ yields a collection of rings $R^J$. A choice of simple roots and co-roots encodes the Frobenius structures for the ring extensions $R^s
\subset R$. A choice of all positive roots essentially encodes the Frobenius structure for the entire hypercube. As mentioned previously, for balanced realizations a choice of all
positive roots can be made canonically from a choice of simple roots, but for unbalanced realizations this is not true. Instead of starting with the data of the Cartan matrix, one should
start with the data of the entire Frobenius hypercube. See the Appendix for more details. \end{remark}

The rings $R^J$ are determined only by the representation of $W$ on $\hg$, and not by the additional structure of a choice of simple roots and co-roots. As a result, the properties of
these ring extensions are determined not by the Coxeter group $W$, but by the Coxeter quotient which acts faithfully (i.e. choose the smallest $m_{st}$ for which $[m_{st}]=0$). For the
rest of this chapter minus a few remarks, we assume the realization is faithful. Whether there exists a Frobenius extension only depends on $\hg$, though the actual Frobenius structure
itself is fixed by the additional data of the simple roots.

\subsection{{\bf Reflection invariants}}\label{reflinvt}

We have already in section \ref{assumptions} defined a map $\pa_s \co \hg^* \to \Bbbk$, and we want to extend it to a trace map $\pa_s \co R \to R^s$. One way to do this is with the formula
$\pa_s(f) = \frac{f - s(f)}{\a_s}$, which expresses $\pa_s$ as a \emph{divided difference operator} or \emph{(simple) Demazure operator}. Another way is to use the \emph{twisted Leibniz
rule} $\pa_s(fg) = \pa_s(f)g + s(f) \pa_s(g)$. Both of these methods require something extra to imply that they are well-defined; we discuss the first approach.

Traditionally, when working over a field of characteristic $\ne 2$, one can assert that the $s$-antiinvariants are a free $R^s$-module generated by $\a_s$. Therefore, $f - s(f)$ has $\a_s$
as a factor, and $\pa_s$ is well-defined. Already one can see how this assertion makes no sense in characteristic $2$, when invariants and anti-invariants are identical. The assertion is
also false in a ring where $2$ is not prime. Nonetheless, with mild assumptions it is still true that $\a_s$ divides $f-s(f)$, for which we need a better argument.

Let us assume Demazure surjectivity, so that there is some linear $\a$ with $\pa_s(\a)=1$. For any $t \in \SC$ we know that $a_{s,t} \a - \a_t$ is symmetric, so any element of $R$ can be
expressed as a polynomial in $\a$ whose coefficients are $s$-invariant polynomials. Because $\a^2 = \a (\a + s(\a)) - \a s(\a)$, where $\a + s(\a)$ and $\a s(\a)$ are symmetric, we see
that any $f \in R$ can be written uniquely as $f = g + h \a$ for $g,h \in R^s$. Now clearly $f - s(f) = h\a_s$, and defining $\pa_s(f) = h$ makes perfect sense. Continuing this
calculation, it is not hard to show that $R^s$ is the polynomial ring generated by the linear terms $a_{s,t} \a - \a_t$ (one of these is redundant) and the quadratic term $\a_s^2$.

Clearly $\{1,\a\}$ and $\{-s(\a),1\}$ form a dual basis for $\pa_s$, making $R$ into a Frobenius extension over $R^s$ of degree $1$. Clearly $\LM_s = \a - s(\a) = \a_s$.

There are two natural choices of $\a$, though both require some assumptions. When $2$ is invertible we often use $\a = \frac{\a_s}{2}$, which is notable because $\{1,\frac{\a_s}{2}\}$ is
the only possible self-dual basis. On the other hand, when local non-degeneracy holds and $a_{s,t} = -[2]$ there is another useful choice, which is $\a = \w_s = \frac{2\a_s +[2]
\a_t}{4-[2]^2}$. This term is uniquely defined in the span of $\a_s$ and $\a_t$ by the fact that $\pa_s(\w_s)=1$ and $\pa_t(\w_s)=0$, and such an element exists (for both $s$ and $t$) iff
local non-degeneracy holds. This is obviously not the right definition of a ``fundamental weight" outside of the dihedral case, since $\pa_u(\w_s)$ need not be zero, but it will
suffice for our purposes in this paper.

Now for the first important consequence of this Frobenius extension: a categorification of \eqref{bisq}.

\begin{claim} \label{RistwiceRi} Letting $B_s \define R \ot_{R^s} R (1)$, we have an isomorphism of graded $R$-bimodules \begin{equation} B_s \ot B_s \cong B_s(1) \oplus B_s(-1).
\label{BiBi} \end{equation} \end{claim}

\begin{proof} This is clear from the $R^s$-bimodule isomorphism $R \cong R^s \oplus R^s (-2)$. After all, one has \[B_s \ot B_s \cong R \ot_{R^s} R \ot_{R^s} R(2) \cong R \ot_{R^s} R^s
\ot_{R^s} R(2) \oplus R \ot_{R^s} R^{s} \ot_{R^s} R(0) \cong B_s(1) \oplus B_s(-1).\] Explicitly, the isomorphism from left to right sends $1 \ot g \ot 1 \mapsto (\partial_s(\a g) \ot
1,\partial_s(g) \ot 1)$, and the isomorphism from right to left sends $(1 \ot 1,0) \mapsto 1 \ot 1 \ot 1$ and $(0,1 \ot 1) \mapsto 1 \ot -s(\a) \ot 1$. \end{proof}

\begin{remark} \label{notsurjective} We have seen that $\pa_s \co R \to R^s$ is a Frobenius trace if and only if $\pa_s$ is surjective. When $\pa_s$ is not surjective we can still ask
whether $R^s \subset R$ is Frobenius with some other trace map $\pa$. If the image of $\pa_s$ forms a nonzero principal ideal generated by $c \in \Bbbk \subset R^s$, then $\pa =
\frac{\pa_s}{c}$ makes sense even when $c$ is not invertible (if $\Bbbk$ is a domain), and this will be a Frobenius trace. The statements in this paper can be modified to deal
with this situation accordingly. However, when the image of $\pa_s$ is not a principal ideal (as in the universal case for the infinite dihedral group), there is little one can do.
\end{remark}

\begin{remark} \label{oneparameterfamily} There is at least a one-parameter family of Frobenius structures for $R^s \subset R$, given by root-rescaling, sending $\LM_s \mapsto \l_s \LM_s$
and $\pa_s \mapsto \l_s^{-1} \pa_s$. This family of Frobenius structures is the only one worth considering because of its other desirable properties: $\pa$ kills $R^s$, $\LM_s$ is
anti-invariant, etcetera. \end{remark}

\subsection{{\bf Dihedral invariants}}\label{dihinvt}

In this section we will be investigating invariant subrings under dihedral parabolic subgroups. Fix a pair $s,t \in S$ and let $a_{s,t}=-[2]$. Continue to assume the realization is
faithful.

We would like to investigate under which conditions $R^{s,t} \subset R$ is a Frobenius extension. As an illustrative example to the reader, we first give a description of the
$W_{s,t}$-invariants and the $W_{s,t}$-antiinvariants that live within the polynomial ring $\Bbbk[\a_s,\a_t]$. Let $\LM$ denote the product of the positive roots. We will soon show that, when $R^{s,t} \subset R$ is a Frobenius extension, $\LM$ is its product-coproduct.

\begin{claim} \label{clm:dihinvt} Suppose that $R = \Bbbk[\a_s,\a_t]$. If $m = \infty$ and $[2] = \mp 2$, then $R^{s,t} = \Bbbk[\a_s \pm \a_t]$. If $m = \infty$ and $A$ is non-degenerate
(i.e. $4-[2]^2$ is invertible) then $R^{s,t} = \Bbbk[z]$ for $z = \a_s^2 -[2] \a_s \a_t + \a_t^2$. If $A$ is non-degenerate and $m<\infty$ then $R^{s,t} = \Bbbk[z,Z]$ where $Z$ is a product
of $m$ linear factors as described below. When $m$ is infinite, there are no $W_{s,t}$-antiinvariants. When $m$ is finite the $W_{s,t}$-antiinvariants are freely generated over $R^{s,t}$ by
$\LM$. \end{claim}

\begin{proof} This is mostly a brute force calculation, and much of it is well-known. Let us remark on what happens when $m < \infty$. Clearly $\LM$ is antiinvariant, since $s$
permutes the positive roots except $\a_s$, which it sends to $-\a_s$. One may choose $Z$ to be the product of the $W$-orbit of $\w_s$ (resp. $\w_t$) or, if these elements exist in $\hg^*$,
to be the product of the positive roots of the dihedral group $W_{2m}$ which are not roots of $W_m$. \end{proof}

\begin{remark} When $R$ is not generated by $\a_s$ and $\a_t$, it will typically be the case that $R^{s,t}$ is generated by the elements $z$ and $Z$ above, as well as additional linear
terms. One can always guarantee this when $4-[2]^2$ is invertible, by a simple calculation. However, when $[2] = \pm 2$, the ring $R^{s,t}$ can be more complicated. We do not have a
general statement to make. \end{remark}

One can already see that when $m=\infty$, the subring $R^{s,t}$ has a strictly smaller transcendence degree than $R$, and thus $R^{s,t} \subset R$ can not be a Frobenius extension. For the
rest of this section we assume that $m < \infty$. We let $w_0$ be the longest element of this parabolic subgroup. Our next goal is to define the Frobenius trace $R \to R^{s,t}$.

\begin{claim} The simple Demazure operators satisfy the braid relation \[ \ubr{\pa_s \pa_t \ldots}{m} = \ubr{\pa_t \pa_s \ldots}{m} \define \pa_{s,t}. \] This composition is called a
\emph{(higher) Demazure operator}, and it maps $R$ to $R^{s,t}$. \end{claim}

\begin{proof} The braid relations for Demazure operators are well known. They can also be shown by a straightforward calculation. Both sides can be expressed as a sum where each term is
$\pm \frac{w(f)}{\pi}$ for some $w \in W_{s,t}$ and $\pi$ some product of roots. One can match the terms on each side of the equality using the observations in section \ref{roots}. Because
of the braid relation, the image of $\pa_{s,t}$ is in the kernel of both $\pa_s$ and $\pa_t$, and is therefore in $R^{s,t}$. \end{proof}

\begin{remark} When the realization is odd-unbalanced, the braid relation only holds up to scalar. For the even-unbalanced case, see Remark \ref{foolishremark} below. \end{remark}

If $\pa_{s,t}$ is a Frobenius trace $R \to R^{s,t}$ then there is a compatible square of Frobenius extensions, where $R^{s,t} \subset R^s$ has trace map $\pa^s_{s,t} = \pa_{\hat{m-1}_t}$
and $R^{s,t} \subset R^t$ has trace map $\pa^t_{s,t} = \pa_{\hat{m-1}_s}$.

Just because $\pa_s$ and $\pa_t$ are individually surjective does not mean that $\pa_s \pa_t$ is surjective. This would require that $R^s \into R \to R^t$ is surjective. When $\hg^*$ is
spanned by the simple roots, this is only the case when $\w_t \in R^s$ is defined, i.e. when local non-degeneracy holds.

\begin{prop} With Demazure surjectivity, local non-degeneracy, and lesser invertibility, $\pa_{s,t} \co R \to R^{s,t}$ is a Frobenius trace. \label{past-surjective} \end{prop}

\begin{proof} (Sketch) This is a brute force computation. It is sufficient to show that $\pa^s_{s,t}$ is a Frobenius trace. Consider the basis $\{1,\w_t,\w_t^2,\ldots,\w_t^{m-1}\}$ of
$R^s$ over $R^{s,t}$. It is rather cute to calculate that $\pa^s_{s,t}(\w_t^{m-1}) = [m-1]!$, so we leave it as an exercise in the Leibniz rule. One can show combinatorially that a dual
basis exists when $[m-1]!$ is invertible, by calculating the dual basis inductively. \end{proof}

\begin{remark} Just as in Remark \ref{notsurjective}, one need not invert $4-[2]^2$ or $[m-1]!$ to guarantee the existence of some Frobenius trace. Instead, it may be possible to divide
$\pa_{s,t}$ by $[m-1]!$ formally. \end{remark}

\begin{remark} \label{foolishremark} Because of the braid relation, we can define $\pa_w$ for any $w \in W$. Note that $\pa_{\hat{k}_s} = 0$ for any $k > m$. In particular, if the
realization is not faithful for this dihedral group (so that $[m]=0$ for $m<m_{st} < \infty$) then $\pa_{w_0}=0$, and the braid relation holds for foolish reasons. It is not a Frobenius
trace map, of course. It is impossible for a graded ring extension $A \subset B$ to be a Frobenius extension of two different degrees at once, so that if $R^{s,t} \subset R$ is a Frobenius
extension of degree $m$, it will not be one of degree $m_{st}$. Thus when the realization is even-unbalanced, it is usually not faithful, and therefore the braid relation on Demazure
operators holds for foolish reasons. \end{remark}

\begin{remark} It seems likely that the assumptions in Proposition \ref{past-surjective} can be weakened, but I have not done the computation. We have also not bothered to calculate the
condition for $\pa_{w_J}$ to give a Frobenius trace $R \to R^J$ beyond the dihedral case. When this happens, we have a Frobenius (partial) hypercube including all finitary $R^I$.
\end{remark}

\begin{remark} A nice formula for dual bases of the Frobenius extensions $R^I \subset R^J$ is unknown to the author, even in type $A$ where the situation is far better studied. Dual bases
for $\C[x_1,\ldots,x_n]$ over $\C[x_1,\ldots,x_n]^{S_n}$ are presented in \cite{KLMS}, though this is not quite the same as type $A$, which is the subring given by traceless polynomials.
Note that Schubert polynomials do \emph{not} form dual bases for $\C[x_1,\ldots,x_n]$, because $\pa_{w_0}(f_i f_j^*) = \delta_{i,j}$ only modulo positive degree symmetric polynomials.
\end{remark}

Though the results below continue to apply to any dihedral parabolic subgroup, we will now assume $W$ is dihedral, and write $R^W = R^{s,t}$, and use similar notation like $\pa_W =
\pa_{s,t}$.

\begin{thm} \label{frobsquare} Take all three assumptions, with $m < \infty$. Then $R^W \subset R^s, R^t \subset R$ is a graded Frobenius square.
Therefore, $R(m)$ is a free $R^W$-module of graded rank $\qW$, and $R^i(m-1)$ is a free $R^W$-module of graded rank $\frac{\qW}{(v+v^{-1})}$. Any dual bases satisfy the following
properties. Starting at \eqref{funnyhere}, we will take elements in $R^s$ and include them in $R$, in order to apply $\pa_t$.
\begin{subequations}
\begin{equation} \LM_W = \LM \end{equation}
\begin{equation} \pa_W(\LM) = 2m \end{equation}
\begin{equation} \LM_W^s = \frac{\LM}{\a_s} \end{equation}
\begin{equation} \pa^s_W(\LM^s_W) = m \end{equation}
\begin{equation} \Delta^s_{W(1)} \partial_t(\Delta^s_{W(2)}) = \partial_t(\Delta^s_{W(1)})\Delta^s_{W(2)} = \frac{\LM}{\a_s \a_t} \label{funnyhere} \end{equation}
\begin{equation} \label{foobar} \Delta^s_{W(1)} \ot \partial_t(f\Delta^s_{W(2)}) = \partial_s(f\Delta^t_{W(1)}) \ot \Delta^t_{W(2)} \ \ \in R^s \ot_{R^W} R^t \textrm{ for any } f \in R 
\end{equation}
\end{subequations}
(The last two equations use Sweedler notation.) In particular, the map $R \to R^s \ot_{R^W} R^t$ sending $f \mapsto \Delta^s_{W(1)} \ot \partial_t(f \Delta^s_{W(2)}) = \partial_s(f \Delta^t_{W(1)}) \ot \Delta^t_{W(2)}$ is
well-defined, $R^s$-linear on the left, and $R^t$-linear on the right. \end{thm}

\begin{proof} The equations above hold in general for any square of Frobenius extensions, as shown in \cite{EWFrob}. That paper requires a technical condition, that dual bases for $R$ over
$R^s$ can be chosen such that one basis lies entirely in $R^t$; we have already described how this follows from local non-degeneracy in the proof of Proposition
\ref{past-surjective}. The only interesting piece of data is that $\LM_W=\LM$, the product of the positive roots, from which the other facts can be deduced. To show this, note that
$\pa_s(\LM_W) = \pa_s(\LM^s_W \a_s) = 2 \LM^s_W = \frac{2\LM_W}{\a_s}$. This is only possible if $\LM_W$ is $s$-antiinvariant. It is $t$-antiinvariant by the same argument, so for degree
reasons it must be equal to a scalar multiple of $\LM$. A calculation shows that $\pa_W(\LM)=2m$, so they agree precisely. \end{proof}

As a consequence, we can also categorify \eqref{bibw0} and \eqref{bw0sq}.

\begin{cor} Letting $B_W \define R \ot_{R^W} R (m)$ we have the following isomorphisms: \begin{equation} \label{BiBW} B_s \ot B_W \cong B_W \ot B_s \cong (v+v^{-1}) B_W \end{equation}
\begin{equation} \label{BWBW} B_W \ot B_W \cong \qW B_W \end{equation} \end{cor}

Choosing any basis $\{f_k\}$ with dual basis $\{f^*_k\}$, the $\qW$-many projection maps from $R$ to $R^W$ are $g \mapsto \partial_W(g f_k)$ and the inclusion maps are $g \mapsto gf^*_k$.
These maps, applied to the middle factor in $R \ot_{R^W} R \ot_{R^W} R$, give you the projections and inclusions in \eqref{BWBW} as well. To deduce \eqref{BiBW} we write $R \ot_{R^s} R
\ot_{R^W} R$ as $R \ot_{R^s} R \ot_{R^s} R^s \ot_{R^W} R$ and reduce the second factor of $R$ as in Claim \ref{RistwiceRi}.

\subsection{{\bf Soergel bimodules and variants}}\label{soergelintro}

We continue to assume the realization is faithful, and that Demazure surjectivity holds. We have already defined the $R$-bimodules $B_W$ (when $m$ is finite) and $B_i$ for $i \in \{s,t\}$.
These can be used to define a number of full subcategories of $(R,R)$-bimodules, and a number of full sub-2-categories of $\Bim$. Recall that $\Bim$ is the 2-category whose objects are
rings $A$, and for which $\Hom_{\Bim}(A,B)$ is the category of $(B,A)$-bimodules, with the obvious tensor structure giving the composition of 1-morphisms.

\begin{defn} The category $\BSBim$ is the (non-additive, non-graded) full monoidal subcategory of $(R,R)$-bimodules generated by $B_s$ and $B_t$. Given a
sequence $\ul{w}=i_1 i_2 \ldots i_k$, we write $BS(\ul{w}) = B_{i_1} \ot_R B_{i_2} \ot_R \cdots \ot_R B_{i_k}$. These are called \emph{Bott-Samelson bimodules}. We write
$\HOM(BS(\ul{w}),BS(\ul{y}))$ for the graded vector space of $R$-bimodule maps from $BS(\ul{w})$ to any shift of $BS(\ul{y})$. It is a graded $R$-bimodule itself. \end{defn}

\begin{defn} Suppose that $m < \infty$ and $R^W \subset R^s, R^t \subset R$ is a Frobenius square. The category $\fooBim$ is the (non-additive, non-graded) full monoidal subcategory of
$(R,R)$-bimodules generated by $B_s$ and $B_t$ and $B_W$. Objects are called \emph{generalized Bott-Samelson bimodules}. We use similar conventions as for $\BSBim$. \end{defn}

\begin{defn} The category $\SBim$ is the graded Karoubi envelope of $\BSBim$. That is, it is the full additive monoidal subcategory of graded $(R,R)$-bimodules containing all grading
shifts, direct sums, and direct summands of Bott-Samelson bimodules. Objects are called \emph{Soergel bimodules}. \end{defn}

Though it is not immediately obvious, $\fooBim \subset \SBim$, so that $\SBim$ is also the Karoubi envelope of $\fooBim$.

\begin{defn} The 2-category $\SBSBim$ is the (non-additive, non-graded) full sub-2-category of $\Bim$, whose objects are the rings $R^J$ for $J \subset S$ finitary, and whose 1-morphisms
are generated by the $(R^I,R^J)$-bimodule $\Ind^I_J = R^I$ and the $(R^J,R^I)$-bimodule $\Res^I_J = R^I(\ell(J)-\ell(I))$ whenever $I \subset J$. Objects are called \emph{singular
Bott-Samelson bimodules}. \end{defn}

\begin{defn} The 2-category $\SSBim$ is the graded Karoubi envelope of $\SBSBim$. Objects are called \emph{singular Soergel bimodules}. \end{defn}

It is not hard to see that the endomorphism category of $\emptyset \subset S$ inside $\SBSBim$ is $\fooBim$ (or rather, they have the same additive graded envelope) when $m<\infty$, and
$\BSBim$ when $m = \infty$. Therefore the endomorphism category $\emptyset$ inside $\SSBim$ is $\SBim$.

If we wish to emphasize the base ring, we may write $\BSBim_{\Bbbk}$. The definition of $\BSBim$ depends on the realization, and for any base change $\Bbbk \to \Bbbk'$ there is an obvious
functor $\BSBim_{\Bbbk} \ot_{\Bbbk} \Bbbk' \to \BSBim_{\Bbbk'}$. However, this functor is by no means an equivalence! For instance, specializing $q$ to a root of unity and thus passing
from a faithful realization of an infinite dihedral group to a non-faithful one will add new morphisms between Bott-Samelson bimodules. In addition, taking the Karoubi envelope does not
commute with base change, which may create additional idempotents.

\subsection{{\bf Soergel and Williamson Categorification Theorems}}\label{soergelcatfn}

We summarize the main theorems of Soergel and Williamson, as they apply to the dihedral case.

\begin{defn} A realization over a field $\Bbbk$ of characteristic $\ne 2$ is \emph{reflection-faithful} if an element preserves a codimension 1 hyperplane of $\hg$ if and only if it is a
reflection, and moreover if two distinct reflections preserve distinct hyperplanes. \end{defn}

In particular, any reflection-faithful realization is faithful, and satisfies all three assumptions. However, the reflection representations of affine Weyl groups are not
reflection-faithful.

\begin{thm} (See Soergel \cite{Soe5} Theorems 1.10, 4.2, 5.5, 6.16) Let $m \ge 2$ or $m=\infty$. Let $\Bbbk$ be an infinite field of characteristic $\ne 2$, and fix a reflection-faithful
representation of $W_m$ over $\Bbbk$. Then the SCT and the Soergel conjecture (see section \ref{techniques}) hold for $\BSBim$. \end{thm}

Soergel's results apply in great generality to other Coxeter groups. In the general case, one can still show the SCT but not the Soergel conjecture. Soergel's results and techniques have
been generalized by Libedinsky \cite{LibRR} to some non-reflection-faithful realizations, including the reflection representation of any Coxeter group. We speak of a \emph{Soergel
realization} to imply that the SCT can be quoted from the literature. As discussed in the introduction, the sequel \cite{EWGR4SB} will use the diagrammatic presentation of morphisms to
give an alternate proof of the SCT in greater generality.

\begin{thm} (See Williamson \cite{WilSSB} Theorems 7.5.1, 7.4.1 and others) Fix a Soergel realization. There is a functor from $\HG$ to the Grothendieck category of $\SSBim$, sending $v$
to the grading shift, $b_J \in \Hom(J,K)$ to $[\Ind^J_K]$ for $J \subset K$, and $b_J \in \Hom(K,J)$ to $[\Res^J_K]$. This map is an isomorphism, and over the empty parabolic it restricts
to the isomorphism above from $\HB$ to $[\SBim]$. There is an analogous formula for calculating the size of 2-morphism spaces in $\SSBim$, involving the standard trace map on $\HG$.
\end{thm}

%% file: TemperleyLieb.tex
%
%

\subsection{{\bf The Uncolored Temperley-Lieb category}}

The (uncolored) \emph{Temperley-Lieb algebra} on $n$ strands $TL_n$ is an algebra over $\Z[\d]$ which can be realized pictorially. It has a basis given by crossingless matchings with $n$
points on bottom and $n$ on top. Multiplication is given by vertical concatenation of diagrams, and by replacing any closed component (i.e. circle) with the scalar $-\d$. We denote the
crossingless matching representing the identity element by $\1_n$.

\begin{example} An element of $TL_{10}$: $\ig{1}{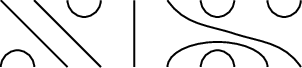}$ \end{example} 

The Temperley-Lieb algebra is part of the \emph{Temperley-Lieb category} $\mTL$, a monoidal category whose objects are $n \in \N$, pictured as $n$ points on a line, and where the
morphisms from $n$ to $m$ are spanned by crossingless matchings with $n$ points on bottom and $m$ on top. Composition is given by concatenation and resolving circles, as usual, so that
$\End(n)=TL_n$.

Let $U=U_q(\mf{sl}_2)$ be the quantum group of $\mf{sl}_2$. Let $V_k$ be the irreducible representation with highest weight $q^k$, and let $V=V_1$. In the introduction we remarked that the
Temperley-Lieb category governs the intertwiners between tensor products of $V$, so that $\Hom_{\mTL}(n,m) \ot_{\Z[\d]} \Q(q)=\Hom_U(V^{\ot n},V^{\ot m})$. Under this base change, $\d
\mapsto [2]_q$, and we use quantum numbers interchangeably with the polynomials in $\d$ that express them.

\begin{prop} \label{TLhasidemps} The Temperley-Lieb algebra $TL_n$, after extension of scalars, contains canonical idempotents which project $V^{\ot n}$ to each isotypic component.
It contains (non-canonical) primitive idempotents refining the isotypic idempotents, which project to each individual irreducible component. Given a choice of primitive idempotents,
$TL_n$ contains maps which realize the isomorphisms between the different irreducible summands of the same isotypic component. These maps can be defined in any extension of $\Z[\d]$ where
the quantum numbers $[2], [3], \ldots, [n]$ are invertible. \end{prop}

\begin{proof} The only part of this proposition which is not tautological is the statement about invertible quantum numbers. This follows from recursion formulas for the idempotents, some
of which can be found below. For more on recursion formulas and coefficients see \cite{FKh}. \end{proof}

The highest non-zero projection, from $V^{\ot n}$ to $V_n$, is known as the \emph{Jones-Wenzl projector} $JW_n \in TL_n$, having been studied independently by Jones \cite{Jon3} and by
Wenzl \cite{Wenzl}. Here are some examples of Jones-Wenzl projectors.

\begin{example} $\ig{1}{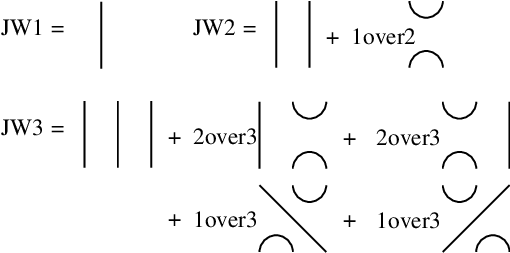}$ \end{example}

\begin{claim} \label{killingJW} Jones-Wenzl projectors satisfy the following properties: \begin{itemize} \item $JW_n$ is the unique map which is killed when any cap is applied on top or
any cup on bottom, and for which the coefficient of $\1_n$ is $1$. \item The ideal generated by $JW_n$ in $TL_n$ is rank 1, since any other element $x \in TL_n$ acts
on $JW_n$ by the coefficient of $\1_n$ in $x$. \item $JW_n$ is invariant under horizontal and vertical reflection. 
\item $JW_n$ can be defined if and only if the quantum binomial coefficients ${n \brack k}$ are invertible for all $0 \le k \le n$ (these are also polynomials in $\d$).
\item There is a recursive formula due to Wenzl \cite{Wenzl} which holds for $n \ge 1$: \begin{equation} \ig{1}{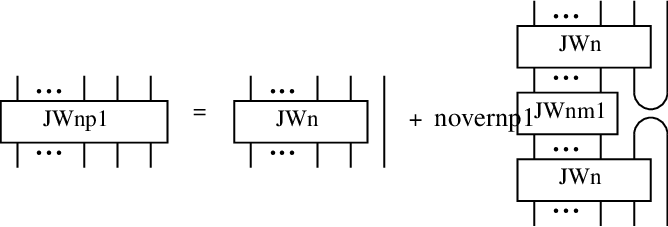}
\label{JWrecursive1}. \end{equation} \item There is an alternate recursive formula (\cite{FKh}, Theorem 3.5),
which sums over the possible positions of cups, and follows quickly from \eqref{JWrecursive1}: \begin{equation} \ig{1}{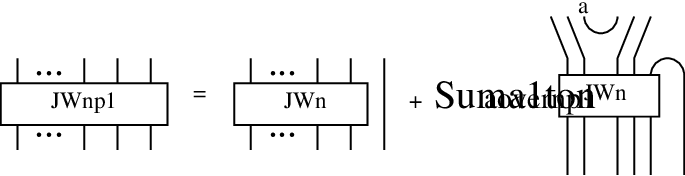} \label{JWrecursive2}. \end{equation} \end{itemize}
\end{claim}

In $\Kar(\mTL)$ we let $V_n$ denote the image of $JW_n$ and $(\cdot) \ot V$ denote the functor of adding a new line on the right (i.e. it sends the object $n$ to $n+1$, and it acts on
morphisms by adding a new line to the right). The recursive formula \eqref{JWrecursive1} gives a diagrammatic proof of the following obvious proposition.

\begin{prop} \label{KarTL} Suppose that all quantum numbers are invertible. The Karoubi envelope $\Kar(\mTL)$ of $\mTL$ has indecomposables $V_n$, $n \in \N$. These satisfy $V \ot V_0
\cong V_0 \ot V \cong V_1$ and $V \ot V_n \cong V_n \ot V \cong V_{n+1} \oplus V_{n-1}$ for $n \ge 1$. \end{prop}

The proofs of the above facts are standard. Some references on the Temperley-Lieb algebra and planar algebras include \cite{GooWen,BMPS,FKh,Mor}. Formulae for Jones-Wenzl projectors were
produced in \cite{Wenzl} and \cite{FKh}; the paper \cite{Mor} has a more detailed version. Finally, the statement that Jones-Wenzl denominators are defined when quantum binomial
coefficients are non-zero is a folklorish result which does not seem to appear in the literature. A representation-theoretic justification was recently explained to me by Ben Webster on
mathoverflow.net.

\subsection{{\bf Roots of Unity and Rotation}}\label{rouandrotation}

Suppose that the Jones-Wenzl projector $JW_{m-1} \in TL_{m-1}$ is well-defined. In particular, $[m-1]$ is invertible. One can ask when $JW_{m-1}$ is \emph{negligible}, which is equivalent
for Jones-Wenzl projectors to the statement that \[\ig{1}{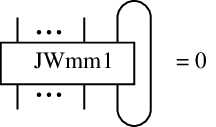}.\] A related question is whether the Jones-Wenzl projector has a ``rotational eigenvalue," i.e. whether rotating
$JW_{m-1}$ by a single strand will yield a scalar multiple of $JW_{m-1}$. Clearly $JW_{m-1}$ is negligible if and only if it has a rotational eigenvalue, if and only if it is killed by all
caps on top, bottom, or sides. See \cite{GooWen} for more details about negligibility.

\begin{example} $JW_2$ is negligible iff $[2]^2 = 1$ iff $[3]=0$. On the other hand, $JW_3$ is only negligible if $[4]=0$ and $[3]=1$. If instead $[4]=0$ and $[3]=-1$ (and $[2]=0$), $JW_3$
is not negligible. \end{example}

The diagram $\ig{.5}{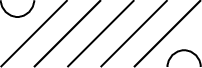}$ is the rotation of the identity map by one strand, and using \eqref{JWrecursive2} one can calculate that its coefficient in $JW_{k}$ is precisely
$\frac{1}{[k]}$. If there is a rotational eigenvalue for $JW_{m-1}$, it is is precisely $\frac{1}{[m-1]}$. If there is a rotational eigenvalue, it must be an $(m-1)$-th root of unity (to
preserve the coefficient of the identity map), and it must also be $\pm 1$ (by a variety of arguments). It particular, if $m$ is even then the rotational eigenvalue of $JW_{m-1}$ must be
1.

\begin{claim} If $m$ is odd, then $JW_{m-1}$ is negligible if and only if $[m]=0$, and it has rotational eigenvalue $[m-1] = \pm 1$. If $m$ is even, then $JW_{m-1}$ is negligible if and
only if $[m]=0$ and $[m-1]=1$. \end{claim}

This claim is the primary reason to assume that the realization is balanced: to guarantee the existence of certain rotationally-invariant Jones-Wenzl projectors. Moreover, it demonstrates
the key difference between even-unbalanced and odd-unbalanced realizations. Even-unbalanced situations do not have rotational eigenvalues.

Suppose that $[m]=0$. Using the fact that $[m-r]=[m-1][r]$ for $0 \le r \le m$, one can see that ${m-1 \brack k}$ is actually just a power of $[m-1]$, so it is invertible. Therefore
$JW_{m-1}$ is well-defined.

\subsection{{\bf The Two-colored Temperley-Lieb Category}}\label{twoTL}

Any embedded 1-manifold will divide the plane into regions which can be colored alternately with 2 colors. Let us assume these two colors are red and blue. We may construct a variation on
the Temperley-Lieb algebra by coloring the regions. The \emph{two-color Temperley-Lieb 2-category} $\mTTL$ has two objects, red and blue. It has two generating 1-morphisms: a map from red
to blue, and a map from blue to red. The 2-morphisms are the $\Z[\d]$-module spanned by appropriately-colored crossingless matchings. Multiplication is defined as in $\mTL$.

\begin{example} An element of $\Hom(rbrbrb,rbrb)$: $\ig{1}{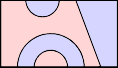}$ \end{example}

\begin{remark} There is an ``asymmetric" version of $\mTTL$ where the evaluation of a circle depends on the color of the interior, but the product of the two circles is thought of as
$[2]^2$. This is akin to an asymmetric Cartan matrix, and is described in the Appendix. \end{remark}

We use notation for 1-morphisms in $\mTTL$ analogous to our notation for reduced expressions in the dihedral group. The 1-morphism with 5 strands which passes through colors $rbrbrb$ will
be denoted either as $\hat{6}_b$ or ${}_r \hat{6}$, and will be called \emph{right-$b$-aligned} or \emph{left-$r$-aligned}. There is a color-switch isomorphism between
$\End(\hat{n}_b)$ and $\End(\hat{n}_r)$; both will be called the \emph{two-colored Temperley-Lieb algebra} $2TL_{n-1}$.

\begin{example} An element of $2TL_{10}$: $\ig{1}{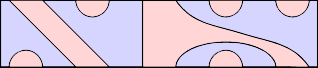}$ \end{example}

The Jones-Wenzl projectors $JW_n$ carry over to $2TL_n$ for either choice of alignment. Proposition \ref{TLhasidemps} generalizes obviously to $\mTTL$, where one remembers that each object
in $\Kar(\mTL)$ appears twice in $\Kar(\mTTL)$, once with each alignment. For instance, $V_0$ is replaced by two distinct indecomposables, $V_b$ and $V_r$, whose identity map is
represented by the empty diagram with the region colored blue or red respectively. Similarly, $V_1$ is replaced by $V_{rb}$ and $V_{br}$, and so forth.

\begin{prop} \label{Kar2TL} Suppose that all quantum numbers are invertible. The Karoubi envelope $\Kar(\mTTL)$ of $\mTTL$ has indecomposables corresponding to nonempty alternating
sequences of red and blue. We denote by $(\cdot) \ot r$ the functor of placing a line on the right, changing a right-$b$-aligned object into a right-$r$-aligned one. This satisfies
$V_{\hat{1}_b} \ot r \cong V_{\hat{2}_r}$ and $V_{\hat{k}_b} \ot r \cong V_{\hat{k+1}_r} \oplus V_{\hat{k-1}_r}$ for $k \ge 2$. \end{prop}

Note that we can label these indecomposables by the elements of $W_{\infty} \setminus \{e\}$. The multiplication rule given here is a categorification of the rule given in Claim
\ref{KLrecursion}.

When $m$ is even, the horizontal flip of the right-$b$-aligned Jones-Wenzl projector $JW_{m-1}$ is the right-$r$-aligned $JW_{m-1}$. When $[m]=0$ (and $[m-1]=1$ for $m$ even), both
versions of $JW_{m-1}$ are negligible. Rotating the right-$b$-aligned $JW_{m-1}$ by one strand will yield $[m-1]$ times the right-$r$-aligned $JW_{m-1}$.

\begin{remark} We know that irreducible representations of $\sl_2$ or $U_q(\sl_2)$ come in two kinds: even and odd dimensional. These are distinguished, for instance, by the action of the
center of $SL_2$ (the \emph{central character}), or by the image of the highest weight in $\Lambda_{\textrm{wt}} / \Lambda_{\textrm{rt}} \cong \Z / 2\Z$. This decomposition is compatible
with the tensor product, so that $U_q(\sl_2)$-rep is actually $\Z/2\Z$-graded-monoidal. Tensoring with the standard representation will switch between even and odd irreducibles, just as it
switches here between the two colors. This gives a representation-theoretic meaning for the two-colored Temperley-Lieb algebra. Even and odd are usually distinguished by the fact that the
trivial representation is even, but for us there is no difference between red and blue. Thus $\mTTL$ encodes a $2$-categorical version of $U_q(\sl_2)$ where we remember the central
character, but forget which character is trivial. This will generalize in the construction of quantum algebraic Satake in type $A$, see \cite{EQAGS}. \end{remark}

\subsection{{\bf Coxeter Lines and associated polynomials}}\label{coxlines}

\begin{defn} Given a colored crossingless matching in $\mTTL$, its \emph{associated monomial} will be $\a_s^a \a_t^b \in R$, where $a$ is the number of blue regions and $b$ the number of
red regions. Given an arbitrary 2-morphism in $\mTTL$, its \emph{associated polynomial} will be obtained by writing it in the basis of crossingless matchings and taking the appropriate
linear combination of monomials. \end{defn}

Note that the associated polynomial is defined only for crossingless matchings, not for crossingless matchings with circles. That is, a circle evaluates to $-[2]$, not to an extra copy of
$\a_s$ or $\a_t$. We will use associated polynomials in a crucial way in sections \ref{sec-grothinfty} and \ref{functortobimfinitesing}, to show that certain morphism spaces between
Soergel bimodules are non-zero.

We are interested in the associated polynomial of a Jones-Wenzl projector. Let us assume that all quantum numbers are invertible (up to the point we are interested in). Recall that we have
defined the positive roots $\LC_s$ and $\LC_t$ for $m = \infty$ in section \ref{roots}. We now place a ``snakelike" order on these roots: $f_{s,0} < f_{t,0} < f_{t,1} < f_{s,1} < f_{s,2} <
f_{t,2} < \ldots$. There is an alternative $t$-aligned snakelike order which begins with $f_{t,0} < f_{s,0} < f_{s,1} < \ldots$. One feature of this order is that the first $m$ roots
enumerate precisely the $m$ roots of the reflection representation of the corresponding finite dihedral group, when $[m]=0$.

Let $\LM^{(s)}_{k}$ be the product of the first $k$ roots in the $s$-aligned snakelike order. Clearly $\LM^{(s)}_k = \LM^{(t)}_k$ for $k$ even, so we may ignore the superscript. When
$[m]=0$ and $[m-1]=1$, $\LM^{(s)}_m = \LM^{(t)}_m$ regardless of the parity of $m$.

\begin{prop} \label{coxeterlinesprop} Suppose that all quantum numbers are invertible. The associated polynomial of $JW_{2k-1}$ is $\LM_{2k}$ times an invertible scalar specified below.
The associated polynomial of the blue-aligned (resp. red-aligned) $JW_{2k}$ is $\LM^{(s)}_{2k+1}$ (resp. $\LM^{(t)}_{2k+1}$) times an invertible scalar specified below. \end{prop}

\begin{cor} \label{coxeterlinescor} When $[m]=0$ and $[m-1]=1$, the negligible Jones-Wenzl $JW_{m-1}$ has associated polynomial precisely equal to $\LM$, the product of all the positive
roots for the corresponding finite dihedral group $W_m$. \end{cor}

\begin{remark} This proposition is the analog of Section 3.7 in \cite{ETemperley}. One should think about $\LM$ in this context not as a polynomial but in terms of the ideal it generates.
This ideal cuts out in $\mf{h}$ the union of all the reflection-fixed lines, i.e the lines $f=0$ for $f$ a root. In \cite{ETemperley} we examine Soergel bimodules in general type $A$, and
obtain an analogous non-principal ideal in $\Bbbk[\a_1,\a_2,\ldots,\a_n]$. This ideal cuts out the \emph{Coxeter lines} in $\mf{h}_{\mf{sl}_n}$ (there called the ``Weyl lines"), which are
the lines given by \emph{transverse intersections} of reflection hyperplanes in $\mf{h}$. While it is not obvious in this dihedral setup for reasons of dimension, what we are doing is
creating an ideal which cuts out \emph{lines}, not one which cuts out codimension one reflection hyperplanes. \end{remark}

\begin{proof} The proposition is clearly true for $JW_0$, with no scalar. We now work inductively using \eqref{JWrecursive2}. To get from the associated polynomial of the right-red-aligned
$JW_{2k}$ to that of the right-blue-aligned $JW_{2k+1}$ we need to multiply by \[ \frac{1}{[2k+1]} (\sum_{a=1, 2k+1-a \textrm{ even }}^{2k+1} [a] \a_s + \sum_{a=1,2k+1-a \textrm{ odd
}}^{2k+1} [a] \a_t).\] Written another way, this is \[ \frac{1}{[2k+1]} ([k+1][k+1]\a_s + [k][k+1]\a_t) = \frac{[k+1]}{[2k+1]} ([k+1]\a_s + [k]\a_t).\] The term in parentheses is precisely
the $2k+1$-st root in the $s$-aligned snake order. Similarly, to go from the right-blue-aligned $JW_{2k-1}$ to the red-aligned $JW_{2k}$ we multiply by $\frac{[k]}{[2k]}([k]\a_s + [k+1]
\a_t)$, which is a scalar multiple of the $2k$-th term of the $s$-aligned snake order.

Therefore, the associated polynomial of the left-blue-aligned $JW_{m-1}$ is equal to the product of the first $m$ roots in the $s$-aligned order, as well as $\frac{[1]}{[1]}
\frac{[1]}{[2]} \frac{[2]}{[3]} \frac{[2]}{[4]} \frac{[3]}{[5]} \frac{[3]}{[6]} \cdots$ where the final term has denominator $[m-1]$. Using $[k]=[m-1][m-k]$, the overall product is
$\frac{[m-1]!}{[m-1]!} [m-1]^{d} = [m-1]^{d}$, where $d$ is the floor of $\frac{m-1}{2}$. This is invertible; in the balanced case, it is $1$. \end{proof}

%% file: InftyDiag.tex
Let us fix a realization $\hg^*$ of the infinite dihedral group. It has a $2 \times 2$ Cartan matrix indexed by $S = \{s,t\}$, with $a_{s,t} = a_{t,s} = -[2]$ in a specialization $\Bbbk$
of $\Z[\d]$. We identify $s$ with the color blue, and $t$ with the color red. The finitary parabolic subset $b=\{s\}$ is also colored blue, $r=\{t\}$ is colored red, and $\emptyset$ is
colored white. We assume Demazure surjectivity (Assumption \ref{ass:Demazure Surjectivity}) everywhere below, though we remark on what can be said in its absence.

Below we will define diagrammatic categories which encode morphisms between Bott-Samelson bimodules (and singular Bott-Samelson bimodules). Let us reiterate a key point from the
introduction. We only encode those morphisms which appear for a generic realization of the infinite dihedral group. In contrast, when the realization is not faithful, the action of the
infinite dihedral group on $\hg^*$ factors through a finite dihedral group. As a consequence, the category of Soergel bimodules will categorify the Hecke algebra of the finite dihedral
group, and it will have additional, non-generic morphisms. Nonetheless, the diagrammatic category defined in this chapter will categorify the Hecke algebra of the infinite dihedral group,
regardless of whether the realization is faithful or not, as proven in \cite{EWGR4SB}. In this chapter, we prove the SCT and the Soergel conjecture for (faithful) realizations of the
infinite dihedral group where lesser invertibility holds.

\subsection{{\bf Diagrammatics and Rotation}}\label{diagrammatics}

There are numerous excellent introductions to diagrammatics for cyclic (i.e. pivotal) monoidal categories and $2$-categories, such as chapter 4 of \cite{LauSL2}. For an example of a
diagrammatic category which is self-biadjoint, see \cite{EKh}.

We make only one remark, also made in \cite{EKh}. \emph{Cyclicity} states that taking any morphism and using adjunction maps to rotate it by 360 degrees will not change the morphism.
Cyclicity is required to draw morphisms on a plane, because any symbol we use to depict the morphism is evidently invariant under 360 degree rotation. However, consider a morphism with
boundary $B_s \ot B_s \ot B_s$, reading around the circle. It is possible to rotate this morphism by 120 degrees, and cyclicity is no guarantee that this will not change the morphism. If
it is the case that 120 degree rotation does not change the morphism, then one may depict the morphism using a diagram which is 120 degree rotation invariant, such as a trivalent vertex.

One may pose the same question for any $2$-morphism whose boundary admits symmetry. For example, in the balanced case one can draw the negligible Jones-Wenzl projector as a vertex, as it
is invariant under all viable rotations, color-switches, and reflections.

When one gives a diagrammatic category by generators and relations, and generators are drawn to have some non-trivial rotational invariance, then there is a hidden relation (called the
\emph{isotopy relation}) which states that rotating that $2$-morphism does nothing. This relation will go unstated, but will need to be checked when applying functors into
non-diagrammatic categories.

%
\subsection{Singular Soergel Bimodules: $m=\infty$}
\label{sec-singinfty}
%
%

In this section we introduce a diagrammatic 2-category $\DGti$ which is supposed to represent $\SBSBim$ for the infinite dihedral group. We define a fully faithful 2-functor $\mTTL \to
\DGti$ (though the proof of faithfulness is postponed until section \ref{sec-grothinfty}). We also define a 2-functor $\DGti \to \SBSBim$, which is fully faithful for faithful realizations
(again, the proof is postponed).

\subsubsection{{\bf Definitions}}\label{SingDefns}

\begin{defn} A \emph{(singular) Soergel diagram} for $m=\infty$ is an isotopy class of a particular kind of decorated 1-manifold with boundary, properly embedded in the planar strip $\R
\times [0,1]$ (i.e. the boundary of the manifold is embedded in the boundary of the strip). The regions cut out by this 1-manifold are colored by finitary parabolic subsets of $S$, in such
a way that two adjacent regions differ by a single index (i.e. $b$ is not adjacent to $r$, no color is adjacent to itself). One may place boxes inside any region, each decorated by a
homogeneous polynomial in the appropriate invariant subring. For instance, one can place a polynomial $f \in R^s$ inside a blue region, or a polynomial $f \in R$ inside a white region.

This region labeling determines a coloring and orientation of the 1-manifold itself, as follows. If two adjacent regions differ by the index $s$, then the component of the $1$-manifold
which separates them will be colored $s$. It will be oriented such that the larger parabolic subset is on the right hand side of the 1-manifold. Conversely, the coloring and orientation of
the 1-manifold determines the region labelings, but not all colorings and orientations are allowable (i.e. will lead to consistent region labels).

The boundary of the manifold gives two sequences of colored oriented points, the \emph{top} and \emph{bottom boundary}. Soergel 1-manifold diagrams are graded, where the degree of a
clockwise cup or cap is $+1$, the degree of a counterclockwise cup or cap is $-1$, and the degree of a box is the degree of the polynomial inside. Note that the degree of a Soergel diagram
is independent of the isotopy class, since by planar Morse theory, cups and caps are created in clockwise-counterclockwise pairs. \end{defn}

Many examples of Soergel diagrams are found in the following pages.

We think of a Soergel diagram as being the data of two oriented 1-manifolds, one blue and one red, which are not allowed to overlap (with some additional restrictions). In the next chapter
when we treat the case $m < \infty$, these manifolds will be allowed to intersect transversely.

To \emph{rotate} a diagram is to change which external regions are \emph{extremal}, i.e. which regions go to $\infty$ on the right or left. Rotation causes part of the boundary to switch
from top to bottom or vice versa. Rotating a singular Soergel diagram may change its degree! We will happily calculate using planar \emph{disk} diagrams, which are essentially equivalence
classes under rotation.

\begin{defn} Let $\DGti$ be the 2-category defined as follows. The objects are $\{\emptyset,b,r\}$. The 1-morphisms are generated by maps from $\emptyset$ to $b$ (resp. $r$) and back. A
path in the object space (e.g. $\ul{b\emptyset r \emptyset b \emptyset b \emptyset}$) uniquely specifies a 1-morphism. The 2-morphism space between 1-morphisms is the free $\Bbbk$-module
spanned by Soergel diagrams with the appropriate boundary (by convention, the bottom boundary is the source, and the top boundary the target), modulo the relations below. All relations
hold with the colors $b$ and $r$ switched. Hom spaces will be graded by the degree of the Soergel diagrams.

It will be an unwritten relation in this and future definitions that boxes containing polynomials will add and multiply as the polynomials do. Then we have:

\begin{subequations} \label{singularrelations}
\begin{equation} \label{circleis} \ig{1}{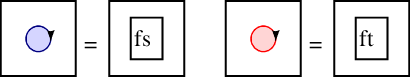}. \end{equation} \begin{equation} \label{slidepolyintoblue} \ig{1}{slidepolyintoblue} \textrm{ when } f \in R^s. \end{equation}
\begin{equation} \label{demazureis} \ig{1}{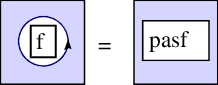}. \end{equation} \begin{equation} \label{circforcesamealt} \ig{1}{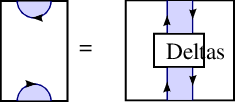}. \end{equation}
\end{subequations}

This ends the definition.
\end{defn}

Let us explain the symbol $\Delta_s$ in \eqref{circforcesamealt}. By placing polynomials in $R$ in a white region, and using the sliding relation \eqref{slidepolyintoblue}, it is clear
that there is an action of $R \ot_{R^s} R$ on diagrams with a blue strip separating white regions. The element $\Delta_s \in R \ot_{R^s} R$ is precisely the comultiplication element for
the Frobenius extension $R^s \subset R$. For example, one knows that $2 \Delta_s = \a_s \ot 1 + 1 \ot \a_s$, so that multiplying \eqref{circforcesamealt} by 2 one obtains \[ \ig{1}{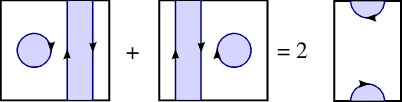}. \]

These four relations are standard for Frobenius extensions. See \cite{EWFrob} for more details.

By Demazure surjectivity, one can write any polynomial $f \in R$ as $g + h \a$ for $g,h \in R^s$ and $\pa_s(\a)=1$. From this it is easy to show the \emph{polynomial forcing relation}:
\begin{equation} \label{circforcegeneral}
{\labellist
\small\hair 2pt
 \pinlabel {$f$} at 18 25
 \pinlabel {$\pa_s(f)$} at 92 22
 \pinlabel {$s(f)$} at 169 22
\endlabellist
\centering
\ig{1}{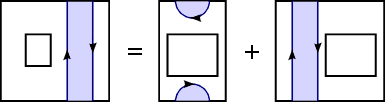}}
\end{equation} Clearly \eqref{circforcegeneral} implies \eqref{circforcesamealt} as well.

One can define the 2-category $\DGti$ without assuming Demazure surjectivity, using relation \eqref{circforcegeneral} instead of \eqref{circforcesamealt}. However, if $\pa_s$ is not
surjective it will be impossible to use this relation to resolve a ``broken strip," as in the left side of \eqref{circforcesamealt}, into a linear combination of diagrams with an unbroken
strip, as in the right side of \eqref{circforcesamealt}. One can only perform this operation up to torsion.

One should keep in mind that we do not yet know whether the map $R \to \End(\ul{\emptyset})$ is injective or even nonzero. We will eventually show it is an isomorphism.

Using Soergel diagrams without boxes, one can still express any polynomial $f \in R$ which is in the image of the sub-polynomial ring generated by $\a_s$ and $\a_t$, using colored circles
as in \eqref{circleis}. In the rest of this section, we describe a boxless presentation for $\DGti$ under the assumption that $\a_s$ and $\a_t$ generate $R$. We will also need to assume
Demazure surjectivity, which in this case is equivalent to the statement that the ideal $(2,a_{st}) \subset \Bbbk$ contains the unit.

Let us note the following relations among boxless diagrams, which follow easily from the relations above.

The \textbf{Empty Circle relation}:
\begin{equation} \label{ccwcirc} \ig{1}{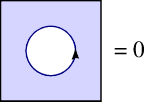} \end{equation}

The \textbf{Cartan relations}:
\begin{subequations} \label{thecartanrelations}
\begin{equation} \label{circeval} \ig{1}{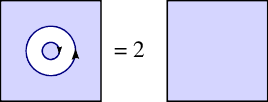} \end{equation}
\begin{equation} \label{circeval1} 	{
	\labellist
	\small\hair 2pt
	 \pinlabel {$-[2]$} [ ] at 75 26
	\endlabellist
	\centering
	\ig{1}{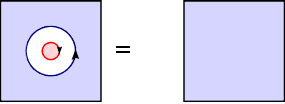}
	}
 \end{equation}
\end{subequations}

The \textbf{Circle Forcing relations}:
\begin{subequations} \label{thecircleforcingrelations}
\begin{equation} \label{circforcesame} \ig{1}{circforcesame} \end{equation}
\begin{equation} \label{circforce1} 	{
	\labellist
	\small\hair 2pt
	 \pinlabel {$-[2]$} [ ] at 75 25
	 \pinlabel {$[2]$} [ ] at 143 26
	\endlabellist
	\centering
	\ig{1}{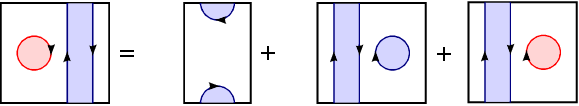}
	}
 \end{equation}
\end{subequations}

We claim that these five relations give an equivalent presentation of $\DGti$ under the assumptions above. The key goal is to give a well-defined notion of a box labeled by $f \in R^I$ in
a region colored $I$.

Using \eqref{circleis} as a convention (rather than a relation), we may place any polynomial in $\a_s$ and $\a_t$ in a white region, and thus by our assumption, any polynomial in $R$. It
is not difficult to use the circle forcing relations to prove the polynomial forcing relation. In particular, the circle forcing relations imply this fact for any linear combination of
$\a_s$ and $\a_t$, and the Leibniz rule is clear. Alternatively, \eqref{circforcegeneral} is easy to check for $\a_s^2$, and thus by our description of $R^s$ in section \ref{reflinvt}, it
holds for any $f \in R^s$. Writing an arbitrary $f \in R$ as $f = g + h \a$ for $g,h \in R^s$, the result is now clear from the linear case.

\begin{claim} Suppose that $\pa_s(\a)=1$. For any $f \in R$ we have \begin{equation} \label{onwaytodemazure} {
\labellist
\small\hair 2pt
 \pinlabel {$f$} [ ] at 24 25
 \pinlabel {$\pa_s(f) \a$} [ ] at 88 25
\endlabellist
\centering
\ig{1.5}{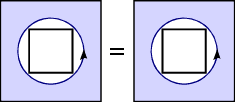}
}
\end{equation} \end{claim}

\begin{proof} The claim holds for any $f \in \Bbbk$, since in that case the LHS is zero by \eqref{ccwcirc}. The claim also holds for $f \in \hg^*$, by an easy application of the Cartan
relations.

Now begin with $f$ inside a counterclockwise circle, as in the LHS of \eqref{onwaytodemazure}. Inside the blue region, one can add a new counterclockwise circle containing $\a$ in its
white interior; this operation will not change the morphism. Apply \eqref{circforcegeneral} to force $f$ into the new white region. The first term is precisely the RHS of
\eqref{onwaytodemazure}. The second term is zero, because it contains an empty circle where $f$ once was. \end{proof}

Since $\pa_s$ is surjective, we can define what it means to place an element of $R^s$ in a blue region using the convention \eqref{demazureis}. The previous claim implies that this
convention is consistent. Now it is easy to prove \eqref{slidepolyintoblue}, using a similar proof to the previous claim. In fact, this shows that \eqref{slidepolyintoblue} is redundant
given \eqref{circforcegeneral} and \eqref{demazureis}, though it was needed to prove \eqref{circforcegeneral} given \eqref{circforcesamealt}.

\subsubsection{{\bf The functor to bimodules and evaluation}}\label{singfunctortobim}

\begin{defn} We give a (strict) 2-functor $\FG \co \DGti \to \SBSBim$. This 2-functor is the identity on objects. The maps from $\emptyset$ to $b$ and back correspond to $\Res_s$ and
$\Ind_s$ (defined in section \ref{frobenius}), respectively. To define the 2-functor on 2-morphisms we need only give the image of the boxes and the clockwise and counterclockwise cups and
caps. Boxes are sent to multiplication by a polynomial. Cups and caps are sent to the four structure maps of the Frobenius extension $R^s \subset R$: that is, the (blue) clockwise cap is
sent to multiplication $B_s = R \ot_{R^s} R(1) \to R$; the clockwise cup is sent to comultiplication $1 \mapsto \Delta_s$ as a map $R \to B_s$; the counterclockwise cap is sent to the
Demazure operator $\pa_s \co R \to R^s$, and the counterclockwise cup is sent to the inclusion $R^s \subset R$. \end{defn}

We could not have defined this functor without Demazure surjectivity, because then $\Delta_s$ would not exist.

\begin{claim} The above definition gives a well-defined 2-functor. \end{claim}

\begin{proof} The isotopy relations follow by properties of Frobenius extensions. The action of polynomials in various regions is clearly preserved by $\FG$, as are the relations
\eqref{singularrelations}. \end{proof}

Whenever $b$ (resp. $r$) appears in a $1$-morphism of $\DGti$, it is either on the far right or far left, or it is surrounded by $\emptyset$ on both sides. Therefore, one can take any
Soergel diagram and perform the following operations: \begin{itemize} \item If a blue region appears on the far right, place a new blue strand to create a white region on the far right. Do
the mirrored operation on the far left. Do the same with blue and red switched. Now the diagram has \emph{extremal white space}, in that $\emptyset$ appears on the right and left. \item
Whenever $\ul{\emptyset b \emptyset}$ appears inside the source (resp. target), precompose with a cup from $\ul{\emptyset} \to \ul{\emptyset b \emptyset}$ (resp. postcompose with a cap
$\ul{\emptyset b \emptyset} \to \ul{\emptyset}$). Do the same with the colors switched. \end{itemize} \igc{1}{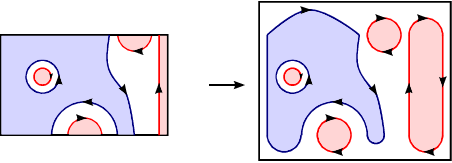} What remains is an endomorphism of $\ul{\emptyset}$. This
is sent by $\FG$ to a polynomial in $R$. This procedure, sending any 2-morphism space to $R$, is called the \emph{evaluation map}. For instance, a boxless diagram where every region is
external would be sent to $\a_s^a \a_t^b$, where $a$ was the number of blue regions and $b$ the number of red regions.

\subsubsection{{\bf Temperley-Lieb}}\label{TLandSing}

\begin{defn} We give a $\Z[\d]$-linear 2-functor from $\mTTL$ to $\DGti$ as follows. The 1-morphism in $\mTTL$ from blue to red is sent to the 1-morphism $\ul{r \emptyset b}$, and the
1-morphism from red to blue is sent to $\ul{b \emptyset r}$. Visually, the map on 2-morphisms takes a crossingless matching and widens each strand into a region labeled $\emptyset$, with
its boundary oriented counter-clockwise. \end{defn}

\igc{1}{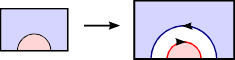}

\begin{claim} The 2-functor above is well-defined, and its image consists of degree $0$ maps. \end{claim}

\begin{proof} That the isotopy relations of $\mTTL$ are satisfied is obvious. That reduction of circles works follows from the Cartan relations. \end{proof}

\begin{claim} Composing this functor with the evaluation map, we get the associated polynomial of a $2$-morphism in $\mTTL$. \end{claim}

\begin{proof} This is obvious. \end{proof}

We will soon use this fact to prove the faithfulness of $\mTTL \to \DGti$, when the realization is faithful. Each Jones-Wenzl projector is sent to a nonzero 2-morphism in $\DGti$, because
Proposition \ref{coxeterlinesprop} implies that its evaluation is a product of roots (up to non-zero scalar).

%
\subsection{The category $\DCti$}
\label{sec-mdti}
%
%

Now we introduce a diagrammatic category $\DCti$ which is supposed to represent $\BSBim$ for the infinite dihedral group.

\subsubsection{{\bf Definitions and basics}}\label{defnsinfty}

\begin{defn} A \emph{Soergel graph} for $m=\infty$ is an isotopy class of a particular kind of graph with boundary, properly embedded in the planar strip (so that the boundary of the
graph is always embedded in the boundary of the strip). The edges in this graph are colored by either $s$ or $t$. The vertices in this graph are either univalent (dots) or trivalent, with
all three adjoining edges having the same color. The boundary of the graph gives two sequences of colors, the \emph{top} and \emph{bottom boundary}. Soergel graphs have a degree, where
trivalent vertices have degree $-1$ and dots have degree $1$. One can place a box labeled by $f \in R$ in any region. \end{defn}

Unlike the singular case, rotating a strand from the top boundary to the bottom does not affect the degree of the morphism, so we can consider Soergel graphs on the planar disk without
any degree issues.

\begin{defn} Let $\DCti$ be the $\Bbbk$-linear monoidal category defined herein. The objects will be finite sequences $\ul{w}=i_1 i_2 \ldots i_d$ of indices $s$ and $t$, with a
monoidal structure given by concatenation. The space $\Hom_{\DCti}(\ul{w},\ul{y})$ will be the free $\Bbbk$-module generated by Soergel graphs with bottom boundary $\ul{w}$ and top
boundary $\ul{y}$, modulo the relations below. All relations hold with the colors $s$ and $t$ switched. Hom spaces will be graded by the degree of the Soergel graphs.

The \textbf{Needle relation}:
\begin{equation} \label{needle} \ig{1}{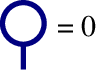} \end{equation}

The \textbf{Barbell relation}:
\begin{equation} \label{barbell} {
\labellist
\small\hair 2pt
 \pinlabel {$\a_s$} [ ] at 32 12
 \pinlabel {$\a_t$} [ ] at 88 12
\endlabellist
\centering
\ig{1}{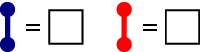}
} \end{equation}

The \textbf{Polynomial forcing relation}:
\begin{equation} \label{dotforcegeneral} \ig{1}{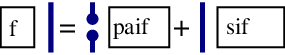} \end{equation}

The \textbf{Frobenius relations}:
\begin{subequations} \label{thefrobeniusrelations}
\begin{equation} \label{assoc1} \ig{1}{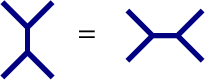} \end{equation}	
\begin{equation} \label{unit} \ig{1}{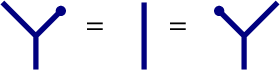} \end{equation}	
\end{subequations}

This ends the definition. \end{defn}

The hidden isotopy relations are below. They use cups and caps, which can be expressed in terms of trivalents and dots by a rotation of \eqref{unit}.
\begin{equation} \label{isotopydot} \ig{1}{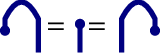} \end{equation} \begin{equation} \label{isotopytri} \ig{1}{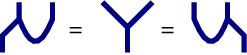} \end{equation}

The \textbf{Barbell Forcing relations} are implications of the above:
	\begin{subequations} \label{thedotforcingrelations}
	\begin{equation} \label{dotforcesame} \ig{1}{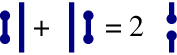} \end{equation}
	\begin{equation} \label{dotforcex} 	{
		\labellist
		\small\hair 2pt
		 \pinlabel {$-[2]$} [ ] at 32 13
		 \pinlabel {$+ [2]$} [ ] at 57 13
		\endlabellist
		\centering
		\ig{1.25}{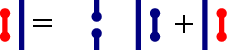}
		} \end{equation}
	\end{subequations}

The analogy between barbells and the circles in $\DGti$ are quite clear. In particular, if the polynomial ring $R$ is generated by $\a_s$ and $\a_t$, we may replace the polynomial forcing
relations with the barbell forcing relations. Barbells were called \emph{double dots} in some previous work.

A \emph{tree} is a connected graph with no cycles. By a successive application of \eqref{unit} and \eqref{assoc1}, one can reduce any tree to a minimal form depending on its boundary. If
the boundary is empty, the minimal form is either a barbell or the empty graph. If the boundary consists of one point, the minimal form is a single dot, called a \emph{boundary dot}. If
there are $n \ge 2$ boundary points, the minimal form has $n-2$ trivalent vertices and no dots; any two such trees are equivalent under \eqref{assoc1}. With the exception of the barbell,
we call such minimal trees \emph{simple}.

It follows immediately from the Frobenius and Needle relations that any blue cycle with an empty interior evaluates to $0$. Combining this with the polynomial forcing rule, a cycle
surrounding a polynomial $f$ can be replaced by the broken cycle with $\pa_s(f)$ outside. We call this procedure \emph{cycle reduction}. The resulting morphism does not depend on where
the cycle was broken or where the polynomial $\pa_s(f)$ is placed.
\begin{equation} \label{brokencycle} \ig{1}{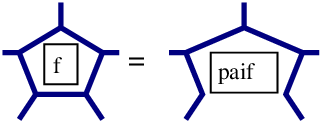} \end{equation}

\subsubsection{{\bf Functors}}\label{functors}

\begin{defn} \label{iotadefn} We give a monoidal functor $\iota \colon \DCti \to \Hom_{\DGti}(\emptyset,\emptyset)$, mapping to diagrams with extremal white space. On objects,
it sends $s$ to the path $\ul{\emptyset s\emptyset}$ and $t$ to the path $\ul{\emptyset t\emptyset}$. We define the functor on generators:

\igc{1}{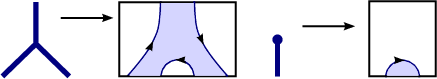} \end{defn}

\begin{claim} The above definition gives a well-defined functor. \end{claim}

\begin{proof} Both categories only consider pictures up to isotopy, so we may ignore questions of isotopy invariance. Relations \eqref{unit} and \eqref{assoc1} also correspond to mere
isotopies in $\DGti$. Relation \eqref{needle} follows from relation \eqref{ccwcirc}. A barbell in $\DCti$ goes to a clockwise circle in $\DGti$. The correspondence between the
polynomial forcing relations of $\DCti$ and those of $\DGti$ is clear. \end{proof}

\begin{prop} The functor $\iota$ is an \emph{isomorphism} of categories. \end{prop}

\begin{proof} Let us construct $\iota^{-1}$. Clearly any path from $\emptyset$ to itself will be composed out of the smaller loops $\ul{\emptyset s\emptyset}$ and $\ul{\emptyset
t\emptyset}$, so that we have a bijection of objects between sequences of indices $i$ and sequences of paths $\ul{\emptyset i\emptyset}$. This defines $\iota^{-1}$ on objects. Now take any
Soergel diagram with extremal white space, use \eqref{demazureis} to replace any polynomial in a blue (resp. red) region with a polynomial in a white region inside a counter-clockwise
circle, and deformation retract the shaded regions to some tree (or barbell) with the appropriate boundaries. The choice of tree is irrelevant.

Now we check this is a well-defined functor (if so, it is clearly an inverse). Checking that relations \eqref{circleis} and \eqref{circforcegeneral} are satisfied is easy. Any instance of
relation \eqref{demazureis} will yield an instance of \eqref{brokencycle}. Finally, \eqref{slidepolyintoblue} follows from \eqref{demazureis} and \eqref{circforcegeneral}. \end{proof}

\begin{defn} Let $\FC=\FC_\infty$ be the $\Bbbk$-linear monoidal functor from $\DCti$ to $\BSBim$ defined by composing $\FG \circ \iota$. The object $\ul{w}$ is sent to $BS(\ul{w})$.
\end{defn}

The \emph{evaluation map} on $\DCti$ is the map $\Hom_{\DCti}(\ul{w},\ul{y}) \to R$ which places a dot on every boundary edge to get a graph with empty boundary, and then applies the
functor $\FC$. This map commutes with the evaluation map on $\DGti$, via $\iota$ and its inverse.

Finally, consider a 2-morphism in $\DGti$ which need not have extremal white space. By adding lines on the left or right (as in the first step of the evaluation map) one can obtain a
diagram with extremal white space. This process is clearly not monoidal, for adding lines does not preserve horizontal multiplication, only vertical multiplication.

Applying this to the image of $\mTTL \to \DGti$, we have a functor $\mTTL \to \DCti$.
\igc{1}{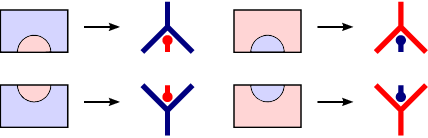}
Again, this is not a monoidal or $2$-functor, failing to commute with horizontal composition, as the following example shows.

\begin{example} \quad\\
	
$\ig{1}{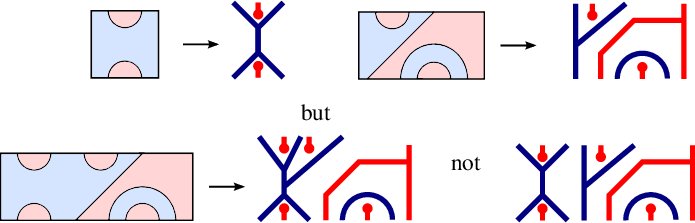}$ \end{example}

As an example, here are the images of the first few right-blue-aligned Jones-Wenzl projectors. Note that $JW_{m-1}$ has $m-1$ strands in the Temperley-Lieb context and $m$ strands in the
Soergel context.

\begin{example} \quad\\ \begin{center}
	
$\qquad \qquad{
\labellist
 \pinlabel {$JW_2 =$} [ ] at -24 113
 \pinlabel {$JW_3 =$} [ ] at 87 114
 \pinlabel {$+\; \; \frac{1}{[2]}$} [ ] at 165 115
 \pinlabel {$JW_4 =$} [ ] at -24 63
 \pinlabel {$+\; \; \frac{[2]}{[3]}$} [ ] at 60 64
 \pinlabel {$+\; \; \frac{[2]}{[3]}$} [ ] at 155 63
 \pinlabel {$+\; \; \frac{1}{[3]}$} [ ] at 60 21
 \pinlabel {$+\; \; \frac{1}{[3]}$} [ ] at 155 19
\endlabellist
\centering
\ig{.8}{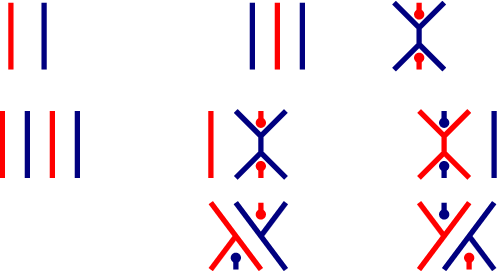}
}$ \end{center} \end{example}

We can modify the functor $\mTTL \to \DCti$ to obtain a map from $2$-colored crossingless matchings on the disk to Soergel graphs on the disk with alternating boundary, having degree
$+2$. Here are a few Jones-Wenzl projectors in this context.

\begin{example}\quad\\ \begin{center}
	
$\qquad \qquad{
\labellist
 \pinlabel {$JW_2 =$} [ ] at -24 110
 \pinlabel {$JW_3 =$} [ ] at 87 110
 \pinlabel {$+\; \; \frac{1}{[2]}$} [ ] at 175 110
 \pinlabel {$JW_4 =$} [ ] at -24 60
 \pinlabel {$+\; \; \frac{[2]}{[3]}$} [ ] at 70 60
 \pinlabel {$+\; \; \frac{[2]}{[3]}$} [ ] at 165 60
 \pinlabel {$+\; \; \frac{1}{[3]}$} [ ] at 70 15
 \pinlabel {$+\; \; \frac{1}{[3]}$} [ ] at 165 15
\endlabellist
\centering
\ig{.8}{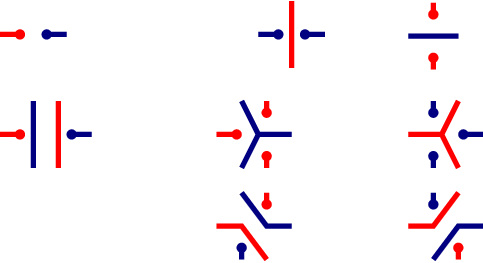}
}$  \end{center} \end{example}

\begin{notation} \label{degree2JWnotation} When we view the Jones-Wenzl projector as a map of degree $2$ with boundary $(st)^{m-1}$ reading around the circle, we will draw it as a circle
labelled by $JW$. For instance, $\ig{.7}{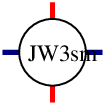}$. This map is \emph{not} rotation-invariant in general, so we can not draw it in a rotation-invariant way. \end{notation}

\subsubsection{{\bf Graphical manipulations, spanning sets, and minimal degrees}}\label{graphicalmanipulations}

We say that two subgraphs of a Soergel graph are \emph{adjacent} if no other part of the graph intervenes, i.e. if they lie in the same connected component of the plane minus the rest of
the graph. Given two adjacent dots of the same color, one can \emph{fuse} them into an edge; this operation decreases the degree of the graph by $2$. Given any two adjacent edges, one can
fuse them as follows: replace each edge with a trivalent vertex attached to a dot, as in \eqref{unit}, so that the dots are adjacent; then fuse the dots. The reverse operation is to
\emph{break} an edge, replacing it with two dots, and increasing the degree of the graph by $2$.

\igc{1}{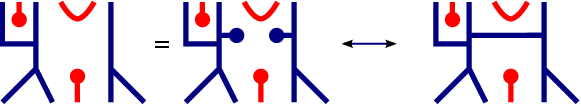}

Equation \eqref{dotforcegeneral} is what allows one to break and fuse lines in practice. When $\pa_s(f)=1$, \eqref{dotforcegeneral} implies that one can fuse two dots (or edges), at the
cost of placing linear polynomials in the adjoining regions. Conversely, \eqref{dotforcegeneral} allows one to force polynomials from one region to another, at the cost of possibly
breaking some edges.

\begin{prop} \label{treesspan} Any morphism in $\DCti$ can be written as a linear combination of graphs where each component is either a simple tree or a polynomial; moreover, all
polynomials are in the left-most region. Therefore, any morphism with no boundary reduces to a polynomial. \end{prop}

\begin{proof} By reducing trees to minimal form, a graph with no cycles and empty boundary reduces to a product of barbells, i.e. a polynomial. Now we induct on the number of cycles in a
graph. If a graph has a cycle, then the interior of that cycle is a graph with empty boundary, having fewer cycles than the original graph. An obvious inductive argument allows one to
reduce the interior to a polynomial, which can then be used to break the original cycle. Once one has reduced the graph to a collection of simple trees with polynomials, one can force all
the polynomials to the left, at the cost of breaking some edges. Breaking edges in a simple tree will result in additional simple trees and barbells, but can not add cycles. \end{proof}

(Disjoint unions of) simple trees with polynomials on the left do not constitute a basis, as the following equality shows.
\begin{equation} \label{notabasis} \ig{1}{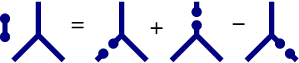} \end{equation}

\begin{cor} The endomorphism ring of $\emptyset$ in $\DCti$ is precisely $R$. \end{cor}

\begin{proof} The endomorphism ring is spanned by polynomials, so that $R$ surjects on to it. After applying $\FC$, we get the identity map $R \to \End_{\BSBim}(R)=R$. Therefore $R \cong
\End(\emptyset)$. \end{proof}

Let us consider morphisms of minimal degree (with a given fixed boundary), within the span of simple trees with polynomials. Clearly any polynomials can be removed, lowering the degree.
Fusing two edges will also lower the degree. If these edges are in the same simple tree, then fusing the edges will create an empty cycle, and yield the zero morphism. However, fusing
edges in two disjoint simple trees will merge them into a single tree. This motivates the following definition.

\begin{defn} Consider a graph, each component of which is a simple tree. The plane minus the red subgraph is split into connected components, each of which contains a (possibly empty)
blue subgraph. If every such blue subgraph is connected, and the same is true with the roles of red and blue reversed, then we call the graph \emph{maximally connected}. \end{defn}

A typical example is a graph in the image of the functor from $\mTTL$; a non-example can then be obtained by breaking an edge.

\begin{claim} \label{maxconisminimal} Fix a sequence of colors along the boundary of a planar disk. Morphisms represented by maximally connected diagrams with this boundary all have the
same degree, and are the minimal degree attainable for morphisms in $\DCti$ with that boundary. If the colors on the boundary alternate, this degree is $2$; for every repetition on the
boundary this degree is lowered by $1$. \end{claim}

\begin{example} Given boundary $brbrrrbbrb$, the minimal degree would be $-2$ because there are 4 repetitions (this sequence lies on a circle so the end is adjacent to the beginning).
\end{example}

\begin{proof} If there is more than one blue tree in a single component of the plane minus the red graph, then two of them can be fused, yielding a graph of smaller degree. Therefore a
minimal degree map must be maximally connected. It is a simple exercise to show that any maximally connected diagram has the appropriate degree. \end{proof}

\begin{claim} All morphisms are generated over maximally connected diagrams by placing polynomials in \textbf{any} of the regions. \end{claim}

\begin{proof} One can reach any disjoint collection of simple trees by breaking lines in a maximally connected diagram. One can break lines by adding polynomials, using
\eqref{dotforcegeneral} with $\pa_s(f)=1$ for instance. \end{proof}

We will soon show that maximally connected graphs form a basis for the space of minimal degree morphisms in $\DCti$. This could be accomplished by applying $\FC$ and using some
combinatorics, but we will give a cheaper proof soon. It is easy to see that the image of either map from $\mTTL$ (to degree $0$, or degree $2$) will be precisely the maximally connected graphs. 

\begin{remark} We have already noted in section \ref{threerecursions} that certain decomposition numbers $c^1_{2k-1}$ are Catalan numbers. These numbers also agree with the degree $2$
part of trace of an alternating monomial. This supports the claim that Temperley-Lieb diagrams form a basis. \end{remark}

\subsubsection{{\bf Pitchforks and Alldots}}\label{moremorphisms}

Let us call the following map a \emph{pitchfork}.
\igc{1}{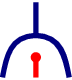}
Let us call the map from $\ul{w}$ to $\emptyset$ consisting entirely of boundary dots by the name \emph{all-dot}.
\igc{1}{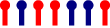}

Recall that ${}_t \ul{\hat{m}}$ denotes the alternating sequence of length $m$ which begins with $t$.

\begin{claim} \label{pitchfork} $\Hom({}_t \ul{\hat{m}},\emptyset)$ is generated (over polynomials in the extremal region) by maps which begin with pitchforks (i.e. they have a pitchfork
somewhere on the far bottom of the diagram), and by the all-dot. \end{claim}

\begin{proof} Without loss of generality, we may assume that our morphism is represented by a collection of simple trees. We will use induction on $m$, where the statement is obviously
true for $m\le 2$ (and there are no pitchforks). If the very first (red) strand is a boundary dot, then we may use the inductive hypothesis for $m-1$ on the remainder of the diagram.
Otherwise, the first strand connects to some other index $t$ on the boundary. Consider the first such index it connects to, and how this divides the graph into two regions.
\igc{1}{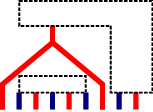}
By induction, the morphism in the inner region either begins with a pitchfork (and thus satisfies our criterion) or is the all-dot. But this latter possibility is within the span of
pitchforks as well.
\igc{1}{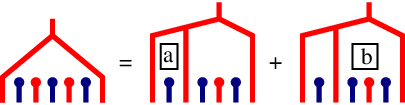}
We fused two red components using \eqref{dotforcegeneral}; here, $\{1,a\}$ and $\{b,1\}$ are dual bases of $R$ over $R^t$. The map with $b$ begins with a pitchfork, while the map with $a$ does not yet begin with a pitchfork. However, in the map with $a$, the region where $b$ is absent looks like what we began with, so by induction it is within the span of pitchforks. \end{proof}

Inside $\End(BS({}_t \ul{\hat{m}}))$ we have the Jones-Wenzl projector $JW_{m-1}$ living in degree $0$. The defining property of $JW$ says that it is killed by all pitchforks. Therefore,
in the Karoubi envelope we see that $\Hom(\Im(JW_{m-1}),\emptyset)$ is spanned (over $R$) by the all-dot.

%
\subsection{The Grothendieck group of $\DCti$ and $\DGti$}
\label{sec-grothinfty}
%

\subsubsection{{\bf Potential categorifications}}\label{splitii}

The only relation in the infinite dihedral Hecke algebroid is 
\begin{equation} \ul{i\emptyset i} \cong \qtwo \ul{i}, \label{iemptysetisplitting} \end{equation} for $i = r$ or $b$. This is categorified in $\DGti$ by a rotation of \eqref{circforcesamealt}, such as
\begin{equation} \ig{1}{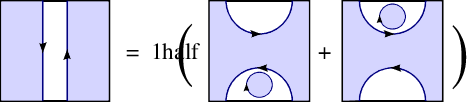} \label{ieiSingDecomp} \end{equation}
This splits the identity of $\ul{i\emptyset i}$ as a sum $i_1 p_1 + i_2 p_2$ of orthogonal idempotents, where $p_1 i_1$ and $p_2 i_2$ are the identity of $\ul{i}$. In this picture we have chosen the dual bases $\{1,\frac{\a_s}{2}\}$ and $\{\frac{\a_s}{2},1\}$, but one could define such an idempotent decomposition for any dual bases $\{1,\a\}$ and $\{-s(\a),1\}$. Similarly, in $\DCti$ we have the idempotent decomposition
\begin{equation} \label{iidecomp} \ig{1}{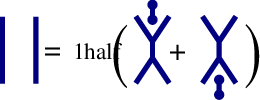} \end{equation}
which is analogous to the isomorphism \eqref{BiBi}. Thus $\DCti$ satisfies \begin{equation} ii \cong i\{1\} \oplus i\{-1\}. \label{iisplittingeqn} \end{equation}

Therefore, $\DCti$ is a potential categorification of $\HB$ (for $m= \infty$), and $\DGti$ is a potential categorification of $\HG$.

\begin{remark} In fact, $\DCti$ is a potential categorification of $\HB$ even without Demazure surjectivity. To obtain the isomorphism \eqref{iisplittingeqn}, we can use the following
idempotent decomposition \begin{equation} 	{
	\labellist
	\small\hair 2pt
	 \pinlabel {$a$} [ ] at 84 19
	 \pinlabel {$b$} [ ] at 188 19
	\endlabellist
	\centering
	\ig{1}{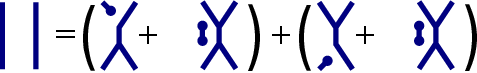}
	}
, \end{equation} where $a+b = -1$. These idempotents do not rely on the existence of a dual basis.

However, for $\DGti$ to be a potential categorification of $\HG$, one does require Demazure surjectivity. It is not difficult to show that the only morphisms $\ul{i\emptyset i}
\to \ul{i}$ are given by a cap with a polynomial, and the construction of appropriate maps $i_1$, $p_1$, $i_2$, and $p_2$ requires the existence of dual bases. \end{remark}

\subsubsection{{\bf The SCT and the Soergel Conjecture}}\label{indecomps}

\begin{defn} Assume lesser invertibility (Assumption \ref{allinvertible}), i.e. all quantum numbers are invertible. Let $B_e$ denote $BS(\emptyset)$, and for $w \ne e$ let $B_w \in
\Kar(\DCti)$ denote the image of the Jones-Wenzl projector in $BS(\ul{w})$ for a reduced expression $\ul{w}$. \end{defn}

The Temperley-Lieb algebra subsumes all degree $0$ endomorphisms of $BS(\ul{w})$ for nontrivial reduced expressions in $W$. Therefore, an idempotent which is primitive in the
Temperley-Lieb algebra is primitive in $\DCti$ as well. In particular, $B_w$ is indecomposable in $\Kar(\DCti)$, and $BS(\ul{w})$ splits into indecomposables in $\Kar(\DCti)$ exactly as
the corresponding object does in $\mTTL$. Therefore, the formula for decomposing $b_{\ul{w}}$ into the KL basis in Claim \ref{allrecursions1} is categorified, decomposing $BS(\ul{w})$ into
indecomposables. Finally, we note that inverting $[2]$ implies Demazure surjectivity, so that the functor $\FC$ is well-defined. Moreover, the splitting \eqref{iisplittingeqn} implies that
every Bott-Samelson bimodule decomposes into direct sums of various $B_w$ with shifts.

\begin{thm} \label{mainthminfty} Assume lesser invertibility. Then the SCT and the Soergel conjecture hold for $\DC$. For any Soergel realization (see section \ref{soergelcatfn}), the
functor $\FC$ is an equivalence, and the indecomposables $B_w$ go to the indecomposable Soergel bimodules that he also labels $B_w$. \end{thm}

\begin{proof} Let us compute $\Hom(B_w,B_e)$, which is equal to $\Hom(BS(\ul{w}),B_e) JW$ for a reduced expression $\ul{w}$, the precomposition of this Hom space with the Jones-Wenzl
projector. Let $\psi \in \Hom(BS(\ul{w}),B_e)$ denote the all-dot. We have already seen that $\Hom(B_w,B_e)$ is spanned by $\psi \circ JW$, in section \ref{moremorphisms}. Also, $\FC(\psi
\circ JW)$ is non-zero. After all, the evaluation map applied to $\psi \circ JW$ is the same as the evaluation map applied to $JW$, which is a nonzero product of roots. Since $R$ acts
freely on morphisms in $\SBim$, it must act freely on $\psi \circ JW$ as well. Therefore, $\Hom(B_w,B_e)$ is the free $R$-module of rank $1$ generated in degree $\ell(w)$.

We have now seen that $\DCti$ satisfies the conditions of Lemma \ref{homspacelemma} for the standard trace, because $\e(b_w)=v^{\ell(w)}$. For a Soergel realization, $\SBim$ also
satisfies these conditions. Therefore, Corollaries \ref{homspacecor} and \ref{fullyfaithfulfunctor} will finish the proof. \end{proof}

This also proves that maximally connected graphs are linearly independent, and form a basis for the minimal degree morphism space.

\begin{cor} Suppose that all quantum numbers are invertible in $\Bbbk$. The 2-category $\Kar(\DGti_\Bbbk)$ is a categorification of the Hecke category $\HG$, and categorifies the
standard trace. The 2-functor $\FG_\Bbbk$ is an equivalence for Soergel realizations. \end{cor}

\begin{proof} We proceed as in Lemma \ref{homspacelemma}. We have a functor from the Hecke category to the Grothendieck category of $\DGti$, and Hom spaces induce a semi-linear pairing on
$\HG$. Because a blue up arrow is clearly biadjoint to a blue down arrow, this pairing is determined by a trace on the category. Any trace on $\HG$ is determined by its values on
$\End_{\HG}(\ul{\emptyset})$ (see section \ref{hecketraces}). Moreover, we know the graded rank of all Hom spaces in $\Hom_{\DGti}(\ul{\emptyset},\ul{\emptyset})$ because it is equivalent
to $\DCti$. Therefore, the trace on $\End_{\HG}(\ul{\emptyset})=\HB$ agrees with the standard trace, and the trace on all of $\HG$ agrees with the standard trace. Moreover, $\FG$ induces
isomorphisms on Hom spaces in the category $\End_{\DGti}(\ul{e})$ since $\FC$ does.

The idempotents in $\mTTL$ give rise to idempotents in $\DGti$ which give a number of direct sum decompositions. In particular, we can find objects in the Karoubi envelope which descend
to the Kazhdan-Lusztig basis of $\HG$, and are therefore indecomposable and pairwise non-isomorphic. Therefore the map from $\HG$ to the Grothendieck category is an equivalence of
categories. Similar arguments show that $\FG$ is an equivalence. \end{proof}

\begin{remark} When some quantum numbers are not invertible, certain idempotents can no longer be defined, and various objects that usually decompose will no longer do so. How precisely
this affects the Grothendieck group is now an answerable question, supposing that one can answer the same question for the Temperley-Lieb algebra. \end{remark}

\begin{remark} \label{partialfaithful} Suppose that the quantum numbers $[k]$ for $k < m$ are invertible. One can use similar arguments to show that $b_w \mapsto [B_w]$ for $\ell(w) \le
m$, and $\FC$ will be fully faithful on Hom spaces between $\hat{\ul{k}}_s$ and $\hat{\ul{n}}_s$ for $k,n \le m$. \end{remark}

\subsubsection{{\bf Induced Modules}}\label{InducedInfty}

\begin{cor} Assume lesser invertibility. The category $\Hom_{\DGti}(\ul{\emptyset},\ul{s})$ categorifies the left ideal of $b_s$, which is the induced module
from the trivial representation of $\HB_s$. The module action is categorified by the monoidal action of $\End_{\DGti}(\ul{\emptyset}) = \DCti$. \end{cor}

Note that this Hom category can easily be described using a slight modification of $\DCti$, in a precise analogy to Chapter 4 in \cite{EInduced}. One may draw the usual Soergel graphs but
require that they end in a ``blue region". One adds a new morphism corresponding to a trivalent vertex with the blue region, and imposes relations corresponding to \eqref{assoc1} and
\eqref{unit}.

Explicitly, the new generator is $\ig{.5}{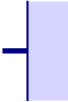}$. The new relations are \begin{equation} \label{InducedRelations} \ig{1}{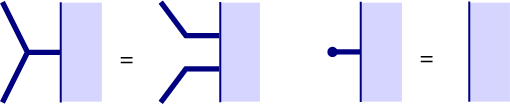} \end{equation} Any usual Soergel graph,
with the usual Soergel relations, may be drawn to the left of the blue region. 

It is quite easy to provide the equivalence of categories between this diagrammatic category and $\Hom_{\DGti}(\ul{\emptyset},\ul{s})$, in analogy to the functor $\iota$ above.

%% file: FiniteDiag.tex
In $\Kar(\DCti)$ we have found two idempotents corresponding to $B_{\hat{m}_s}$ and $B_{\hat{m}_t}$, and we have already shown (see Remark \ref{partialfaithful}) that they descend to
$b_{\hat{m}_s}$ and $b_{\hat{m}_t}$ in the Grothendieck group. Therefore, to obtain a potential categorification of $\HB_m$, one can formally add an isomorphism between $B_{\hat{m}_s}$ and
$B_{\hat{m}_t}$ to $\Kar(\DCti)$. Instead, we do the same thing before taking the Karoubi envelope, modifying $\DCti$ into a category $\DC_m$ by adding maps from $BS(\hat{\ul{m}}_s) \to
BS(\hat{\ul{m}}_t)$ and back which interact in a precise way with the various idempotents.

Given a representation of $W_\infty$ defined by a Cartan matrix over a base ring $\Bbbk$, it is enough to base change $\Bbbk$ so that $[m]=0$ in order to obtain a representation of $W_m$.
However, changing base for $\DCti$ is not sufficient to produce $\DC_m$ - one must add these new morphisms as well.

The morphism $JW_{m-1}$ and its behavior does depend strongly on $m$. However, having abstracted the Jones-Wenzl projector into a symbol, a remarkable thing occurs: the relations in
$\DC_m$ involving this new generator have a very simple form which is independent of $m$.

%
\subsection{Singular Soergel Bimodules: $m<\infty$}
\label{sec-singfinite}	
%
%

In order to define the diagrammatic version of $\SBSBim$, one should assume that $R^W \subset R^s, R^t \subset R$ is a Frobenius square.

\subsubsection{{\bf Definitions}}\label{defnsfinitesing}

\begin{defn} A \emph{singular Soergel diagram} for $m$ is an isotopy class of two oriented 1-manifolds with boundary properly embedded in the planar strip: one of each color, and they can
only intersect transversely. Moreover, there must be a consistent labeling of the regions between these edges. Regions may be labelled by parabolic subsets $\{\emptyset,r,b,p\}$. A line
can be colored $s$ or $t$, and separates two regions whose label differs by that element. The orientation is such that the larger parabolic subset is on the right hand of the 1-manifold. A
polynomial in $R^I$ can be placed in any region labeled by $I$. The boundary of the graph gives two sequences of colored oriented points, the \emph{top} and \emph{bottom boundary}. Not
every oriented colored 1-manifold gives rise to a consistent labeling of regions. Soergel 1-manifold diagrams are graded as in figure \ref{degreeofsing}. \end{defn}

\begin{figure}
\caption{Degrees of Soergel diagrams}
\igc{1}{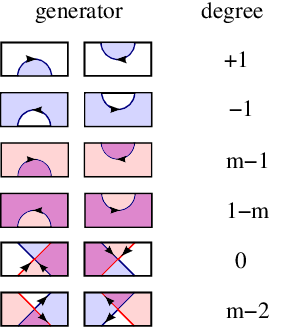}
\label{degreeofsing}
\end{figure}

If there is no ambiguity, we shorten the name to ``Soergel diagram." One can remember that the degree of a cup or cap is always ``in minus out," regardless of orientation, where we
associate to a region the length of the longest element of the corresponding parabolic subgroup (which is $0$, $1$, or $m$). The degree of a sideways crossing is ``big plus small minus middles," which in this case is always $m+0-1-1=m-2$. Note that the degree of a diagram is not defined on the planar disk, but we can still discuss disk diagrams with the same
caveats as before.

A picture is often worth a thousand words. We will try to be very clear about what we refer to in a picture. ``Lines" or ``strands" will refer to sections of the red or the blue
1-manifold. A ``blue circle" will denote a blue line in the shape of a circle (which can separate either a red region from a purple one, or blue from white), while a ``blue circular
region" will refer a blue region enclosed by a circle of indefinite color. Purple designates the parabolic subset $p=\{s,t\}$. A diagram without any purple regions is by definition a
Soergel diagram for $m=\infty$, as in Section \ref{sec-singinfty}. We call it an $\infty$-\emph{diagram}.

\begin{defn} Let $\DG_m$ be the 2-category defined as follows. The objects are $\{\emptyset,r,b,p\}$, thought of as parabolic subsets. The 1-morphisms are generated by maps from $I$ to
$J$ and whenever their difference is a single element. A path in the object space (like $\ul{b\emptyset bprpb\emptyset}$) uniquely specifies a 1-morphism. The 2-morphism space between
1-morphisms is the free $\Bbbk$-module spanned by Soergel diagrams with the appropriate boundary, modulo the relations below. Hom spaces will be graded by the degree of Soergel diagrams.

We have all the relations present in $\DGti$ (see section \ref{SingDefns}). This means that we have a functor $\DGti \to \DG_m$.

\begin{equation} \label{bigdemazure} \ig{1}{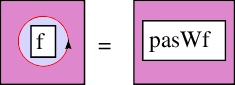} \end{equation}
\begin{equation} \label{bigcircle} \ig{1}{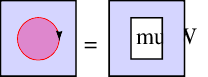} \end{equation}
\begin{equation} \label{slidepolyintopurple} \ig{1}{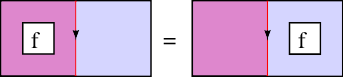} \textrm{ when } f \in R^W \end{equation}
\begin{equation} \label{bigcircforce} \ig{1}{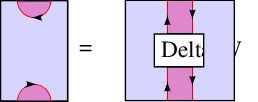} \end{equation}

These relations hold with colors switched. They are standard relations for the Frobenius extension $R^W \subset R^s$, and should remind the reader of the analogous relations for $\DGti$
for the extension $R^s \subset R$. Because of the relation \eqref{slidepolyintopurple}, there is a map from $R^s \ot_{R^W} R^s \to \End(\ul{bpb})$ given by placing boxes in the right and
left regions. The element $\Delta^s_W \in R^s \ot_{R^W} R^s$ is described in Theorem \ref{frobsquare}.

The next relation, an analog of Reidemeister II, says that like-oriented strands can be pulled apart.

\begin{equation} \label{singR2} \ig{1}{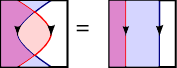} \end{equation}

We also have non-oriented Reidemeister II relations. The element $\pa \Delta_{st}$ represents the element $\Delta^s_{W(1)} \ot \pa_t(\Delta^s_{W(2)}) = \pa_s(\Delta^t_{W(1)}) \ot
\Delta^t_{W(2)} \in R^s \ot_W R^t$, as described in Theorem \ref{frobsquare}.

\begin{equation} \label{singR2nonoriented1} \ig{1}{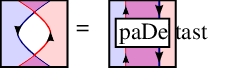} \end{equation} \begin{equation} \label{singR2nonoriented2} \ig{1}{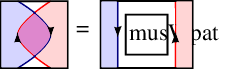} \end{equation}

All of the above relations hold in generality for Frobenius squares. See \cite{EWFrob} for more details. There is only one truly interesting relation, unique to the dihedral group. This
relation starts with $m-1$ alternating arcs of each color, oriented clockwise around an inner purple region. It replaces this with an $\infty$-diagram, the image of $JW_{m-1}$ under the
functor $\mTTL \to \DGti \to \DG_m$, which we still denote $JW_{m-1}$.

\begin{equation} \label{SingJWRelation} {
\labellist
\small\hair 2pt
 \pinlabel {$+$} [ ] at 137 133
 \pinlabel {$+ [2]$} [ ] at 88 22
 \pinlabel {$+ [2]$} [ ] at 148 22
 \pinlabel {$JW$} [ ] at 354 145
\endlabellist
\centering
\ig{1}{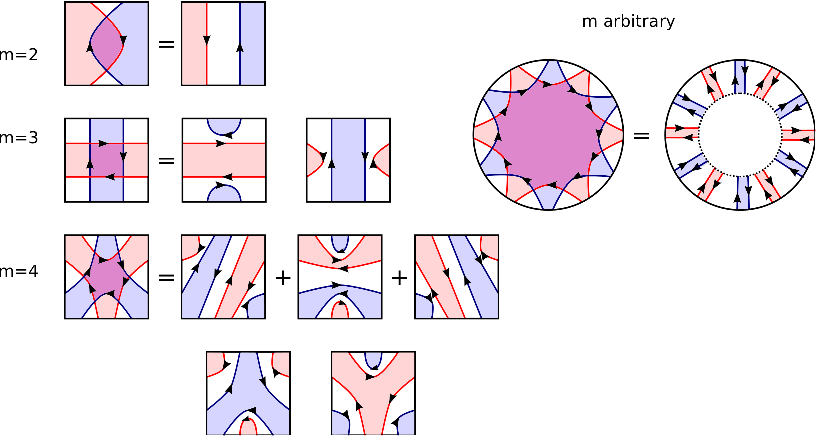}
}
 \end{equation}
\end{defn}

In writing down the coefficients in the Jones-Wenzl projectors above, we have already incorporated the fact that $[m-1]=1$ (as our realization is balanced). Thanks to the results of
section \ref{rouandrotation}, the morphism $JW_{m-1}$ is rotation invariant, like the LHS.

Also note that, as for $\DGti$, one can use boxless diagrams to represent any polynomial which lives in the invariant subring $R^I$. As before, we have not yet proven that the action of $R^I$ (by placing polynomials in a region labeled $I$) is faithful in $\DG_m$.

\subsubsection{{\bf Graph simplifications}}\label{depurplify}

\begin{lemma} \label{circleremoval} \emph{(The Circle Removal Lemma)} Any morphism in $\DG_m$ can be represented as a linear combination of diagrams with no closed components of either
color, but with polynomials in arbitrary regions. Any Soergel diagram with empty boundary reduces to a polynomial. \end{lemma}

\begin{proof} This is proven in \cite{EWFrob}, and is a general statement about Frobenius squares. \end{proof}

In other words, a collection of nested circles easily evaluates to a polynomial using the relations above, and a more complicated system of overlapping circles may be pulled apart using
Reidemeister II moves, and reduced to a polynomial as well. Note that this lemma holds regardless of the color of the region on the boundary, so that a closed Soergel diagram with white
exterior evaluates to a polynomial in $R$, and a closed Soergel diagram with blue exterior becomes a polynomial in $R^s$.

Now we attempt to simplify more complicated graphs with boundary.

\begin{notation} When given $k$ alternating arcs of each color oriented around an inner purple region as in \eqref{SingJWRelation}, we will denote the map by $v_k$. For instance, relation
\eqref{SingJWRelation} says that $v_{m-1} = JW_{m-1}$. In the example below, $k=6$. \end{notation}

\igc{1}{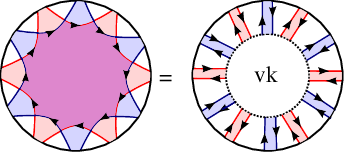}

The map $v_k$ can be positioned as a planar strip diagram in many ways, and its positioning determines its degree. However, when positioned as a map from $\ul{\emptyset b \emptyset r
\emptyset \ldots \emptyset}$ to $\ul{\emptyset r \emptyset b \emptyset \ldots \emptyset}$, it will have degree $2(m-k)$. We let $v_0$ denote a purple circle in a blue annulus in a white
region (or purple in red in white, these are equal). Using the relations above, $v_0$ is equal to $\LM$. It is obvious that any purple region which does not meet the boundary must have a
neighborhood equal to $v_k$ for some $k \ge 1$, or simply be a purple circular region, reducible to a polynomial.

If we place a colored cap on one of the colored sections of the boundary of $v_k$ for $k\ge 1$, we can use \eqref{singR2} to pull two strands apart, and obtain $v_{k-1}$ with an added
``trivalent vertex." \begin{equation}\label{vkprop} \ig{1}{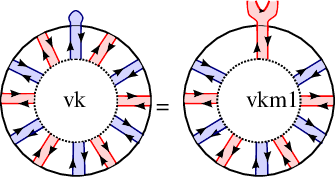} \end{equation} From this, it is easy to show that ``pitchforks" kill $v_k$, and double capping $v_k$ will yield
$v_{k-1}$. \igc{1}{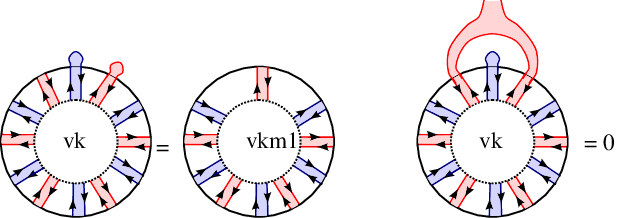}

Relation \eqref{SingJWRelation} says that $v_{m-1}$ can be de-purplified, i.e. rewritten as a sum of $\infty$-diagrams. Therefore, any $v_k$ for $k \le m-1$ can also be de-purplified.
After all, $v_k$ equals $v_{m-1}$ with a number of caps attached, and will be equal to $JW_{m-1}$ with a number of caps attached. Note that relation \eqref{singR2nonoriented2} is actually
a statement about $v_1$; in fact, this relation is redundant, merely stating that the RHS is what one obtains when one caps $JW_{m-1}$ almost everywhere.

{\bf Warning:} Remember, \eqref{SingJWRelation} says that $v_{m-1} = JW_{m-1}$, but it does not say that $v_k = JW_k$ for all $k$. Capping off $JW_{m-1}$ certainly does not yield
$JW_k$ for $k<m-1$. When working with diagrams for $m$, only $JW_{m-1}$ will be relevant, and $JW_k$ for $k<m-1$ will never be used except to assert the existence of well-behaved
idempotents.

For $k>m-1$, one can not use \eqref{SingJWRelation} to de-purplify $v_k$. In fact, $v_m$ is the smallest diagram which is not in the span of $\infty$-diagrams. Thankfully, all $v_k$ for
$k>m$ can be expressed using only $v_m$. In order to show this, and to aid further calculations, we introduce a family of auxiliary maps.

For $k \ge 3$, consider the following maps $C_k$ of degree $2(m-k)$, where the number of circles is $k-2$:

\igc{1}{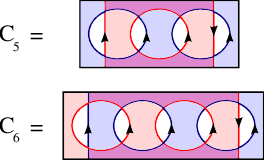}

Using the circle elimination lemma and polynomial manipulation rules, it is clear that this map must reduce to the action of polynomials on either side of $R^i \ot_{R^W} R^j$, where $i,j
\in \{s,t\}$ depend on the colors in the diagram.

\begin{claim} When $k>m$ we have $C_k = 0$. When $m=k$ (so that this map is degree $0$) we have \begin{equation} \label{ckis1} \ig{1}{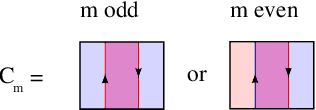} \end{equation} \end{claim}

\begin{proof} This map must reduce to some polynomial in $R^i \ot_{R^W} R^j$. For degree reasons, one already can deduce that $C_k=0$ when $k>m$, and that $C_m$ is a scalar. We need only
check that this scalar is $1$. The derivation goes as follows, for the case $m=4$, and the general proof is similar.

\igc{1}{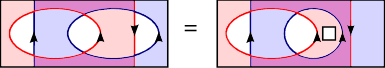}

The first step is to apply (\ref{singR2nonoriented1}) to obtain a sum of diagrams with boxes in the rightmost blue region, and a box in the red region as pictured. However, only one term
in this sum survives: the term where the box in the rightmost region is $1$. Were the box in the rightmost region of degree $>0$, then the diagram to the left of that region would have to
reduce to a polynomial of negative degree, equal to $0$. Therefore $C_m$ is equal to the RHS, where if $f$ is the polynomial dual to $1$ under $\pa^s_W$, then the element in the box is
$\pa_t(f)$.

Now we apply (\ref{slidepolyintoblue}), (\ref{singR2}), and (\ref{demazureis}) until the diagram only has a box in the left region.

\igc{1}{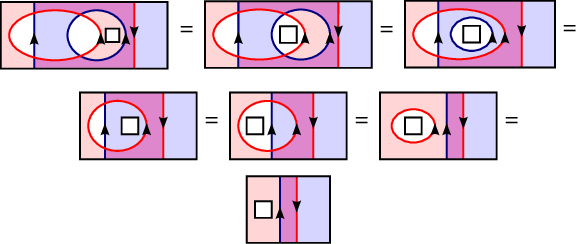}

We place $\partial_t(f)$ in the box in every diagram on the second row. But then $\partial_s(\partial_t(f))$ appears in the box on the third row, and
$\partial_t(\partial_s(\partial_t(f)))$ on the fourth row (and so forth, for $m>4$). Therefore, the final polynomial appearing is $\partial^s_W(f)$, which is $1$. \end{proof}

The following picture is a definition of what it means to ``attach $v_k$ to $v_l$ by $n$ colored bands," for $n \le k,l$. Obviously, the colors of the ``new" bands appearing on right and
left depend on the parity of $n$. This is the example $n=4$, $k=5$ and $l=6$.

{
\labellist
\small\hair 2pt
 \pinlabel {$v_k$} [ ] at 47 91
 \pinlabel {$v_l$} [ ] at 63 29
\endlabellist
\centering
\igc{.7}{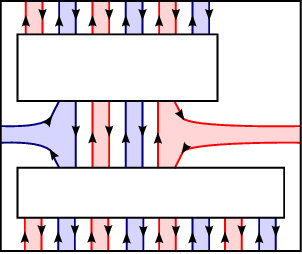}
}

This is the example $n=1$, $k=3$ and $l=4$.

{
\labellist
\small\hair 2pt
 \pinlabel {$v_k$} [ ] at 47 91
 \pinlabel {$v_l$} [ ] at 63 29
\endlabellist
\centering
\igc{.7}{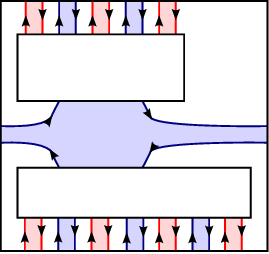}
}

\begin{claim} For $k \ge m$, if one attaches $v_k$ to $v_m$ by $m$ colored bands, one obtains $v_{k+1}$. \begin{equation} \label{gettinghighvk} \ig{1}{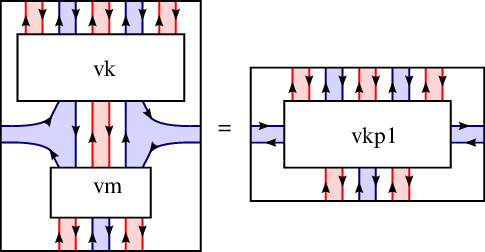} \end{equation}
\end{claim}

\begin{proof} The example below should make the general proof clear. When $m=2$, use \eqref{singR2nonoriented1} instead. \igc{1}{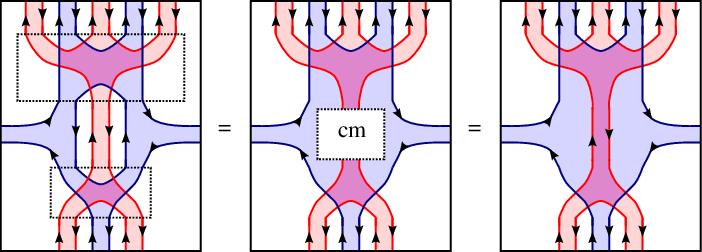}\end{proof}

\begin{claim} \label{connectingwithbands} Assume $k,l > n$. For $n>m$, attaching $v_k$ and $v_l$ along $n$ colored bands yields zero. For $n \le m$, attaching $v_k$
and $v_l$ along $n$ colored bands yields (a linear combination of) diagrams which look like $v_{k+l+1-n}$ with polynomials in various regions, having degree $2(m-n)$.
\end{claim}

\begin{proof} When $n \ge 3$, we can use the same argument as before, only now we resolve $C_n$ into polynomials in $R^i \ot_{R^W} R^j$. When $n > m$ the resulting polynomials are zero
for degree reasons. When $m \ge n \ge 3$, this technique yields $v_{k+l+1-n}$ with these polynomials on either side. When $n=2$, we apply \eqref{singR2nonoriented1} to the center of the
diagram, yielding $v_{k+l-1}$ times $\pa \Delta_W$. When $n=1$, we stretch the purple regions together through the blue region and apply \eqref{bigcircforce}, yielding $v_{k+l}$ times
$\Delta^s_W$. \end{proof}

\begin{lemma} \label{itssurjective} Suppose that a diagram in $\DG_m$ has no purple appearing on the boundary. Then it reduces to linear combinations of diagrams generated by
$\infty$-diagrams and $v_m$. \end{lemma}

\begin{proof} Now suppose that there is a purple region in the diagram. If this purple region is a circular region, it reduces to a polynomial by (\ref{bigcircle}). If this
purple region is $v_k$ for $1 \le k < m$ then we may de-purplify it, as previously discussed. If this purple region is $v_k$ for $k \ge m$ then we may express it using only copies of
$v_m$, by using \eqref{gettinghighvk} iteratively. This procedure will strictly decrease the number of purple regions labelled by $k \ne m$, and therefore we may eliminate all such purple
regions. \end{proof}

Note that if there is no purple on the boundary, then the boundary is a $1$-morphism in the image of $\DGti \to \DG_m$.

\begin{cor} \label{reducetoinfty} Suppose that a diagram has boundary $\ul{\emptyset i_1 \emptyset i_2 \emptyset \ldots \emptyset i_d \emptyset}$ when reading around the boundary of the
disk, where $\ii$ alternates between $b$ and $r$. If $d<2m$ then every diagram with that boundary can be de-purplified. \end{cor}

\begin{proof} Using the circle removal lemma, let us assume that there are no closed 1-manifold components of either color, so that each component connects to the boundary. We may also
assume that each purple region in the diagram is $v_k$ for $k \ge m$, since any region of the form $v_k$ for $k\le m-1$ can be de-purplified with \eqref{SingJWRelation}. Any strand
leaving a region $v_k$ must either meet the boundary, must meet another strand from the same $v_k$, or must meet a strand from a different purple region. Whenever two purple regions share
a strand, it is easy to show that (up to simple manipulations) they are connected with some number of colored bands, and thus can be fused using Claim \ref{connectingwithbands}. Thus we
can assume that no two purple regions share a strand. If any strand from $v_k$ loops back to the same purple region, simple planar arguments imply that $v_k$ must have a cap somewhere,
and thus can be reduced to $v_{k-1}$. Finally, one of these reductions can be performed so long as there is any purple region with $k \ge m$, because there are not enough strands on the
boundary to accomodate all the strands from $v_k$. \end{proof}

\subsubsection{{\bf The functor to bimodules}}\label{functortobimfinitesing}

\begin{defn} We define a 2-functor $\FG_m \colon \DG_m \to \SBSBim_m$ as follows. The objects of the two categories are already identified. For $I\subset J$, it sends the $1$-morphism
from $I$ to $J$ to $\textrm{Res}^I_J$, and the $1$-morphism from $J$ to $I$ to $\textrm{Ind}^I_J$. Cups and caps are sent to the appropriate Frobenius structure maps, as discussed in
section \ref{frobenius}. The upwards-pointing crossing goes to the canonical isomorphism $R \ot_{R^s} R^s \ot_{R^W} R^W \cong R \ot_{R^t} R^t \ot_{R^W} R^W$, which are both $R$ as an
$(R,R^W)$-bimodule; this isomorphism sends $1 \ot 1 \ot 1 \mapsto 1 \ot 1 \ot 1$. Similarly, the downwards-pointing crossing is a canonical isomorphism between $R$ and itself as an
$(R^W,R)$-bimodule. The sideways crossings are maps between $R$ and $R^s \ot_{R^W} R^t$ as $R^s-R^t$-bimodules, which are either multiplication or the map $f \to \Delta^s_{W,(1)} \ot
\partial_t(f\Delta^s_{W,(2)})$ discussed in Theorem \ref{frobsquare}. \end{defn}

\begin{prop} This 2-functor is well-defined. \end{prop}

\begin{proof} We must check the relations of the category, as well as isotopy relations. With the exception of \eqref{SingJWRelation}, all these relations (including the isotopy
relations) hold in more generality for squares of Frobenius extensions, as proven in \cite{EWFrob}. The only relation that needs to be checked is \eqref{SingJWRelation}, expressing
$v_{m-1}$ as $JW_{m-1}$.

By rotation, we may view the map $v_{m-1}$ as a map from $BS(\hat{\ul{2(m-1)}}_s) \to R$, which is killed by every pitchfork and has minimal degree. We have already seen that the space of
such maps is 1-dimensional, and that composing with an all-dot (which corresponds in the singular world to placing caps on every colored band) gives an \emph{injective} map from this
1-dimensional space into $\End(\emptyset)=R$. Therefore, we need only show that $v_{m-1}$ and $JW_{m-1}$ produce the same polynomial when capped off everywhere. But we have already shown
that $v_{m-1}$ capped off everywhere is $v_0 = \LM$, while $JW_{m-1}$ capped off everywhere is its associated polynomial, which is also $\LM$ (see section \ref{coxlines}). \end{proof}

\begin{cor} For any parabolic subset, $\End_{\DG_m}(I) \cong R^I$. \end{cor}

\begin{proof} We have already seen that all such diagrams reduce to polynomials, and the existence of the functor implies that there can be no additional relations between polynomials.
\end{proof}

\begin{cor} \label{woohoo} For $\ii$ an alternating sequence of length $d<2m$ and $X$ the corresponding 1-morphism $\ul{\emptyset i_1 \emptyset \ldots \emptyset i_d \emptyset}$ in $\DG_m$
or $\DGti$, the 2-functor $\DGti \to \DG_m$ induces isomorphisms on $\Hom(X,\ul{\emptyset})$. \end{cor}

\begin{proof} The map on Hom spaces induced by $\DGti \to \DG_m$ is surjective by Corollary \ref{reducetoinfty}. The $2$-functor $\FG$ is faithful on these objects, as in Remark
\ref{partialfaithful}, and it factors through $\DG_m$, so we also have injectivity. \end{proof}

\subsubsection{{\bf The Grothendieck Algebroid}}\label{grothsingagain}

\begin{thm} Assume Demazure surjectivity, local non-degeneracy, and lesser invertibility. The 2-category $\Kar(\DG_m)$ categorifies the Hecke algebroid $\HG_m$. If $\Bbbk$ is a Soergel ring then $\FG_m$ is an
equivalence after passage to the Karoubi envelope. \end{thm}

\begin{proof} First we show that $\DG_m$ is a potential categorification of $\HG_m$. We have presented $\HG_m$ by generators and relations in section \ref{presentinghecke}, and it is
clear how to define the map $\HG_m \to [\DG_m]$ on generators. Equation \eqref{bisqsing} follows as in $\DGti$, and equation \eqref{bw0sqsing} follows in the same way from
\eqref{bigcircforce}. The isomorphism $\ul{\emptyset bp} \cong \ul{\emptyset rp}$ required by equation \eqref{upup} is realized by the upwards-pointing crossing, and the same for
\eqref{downdown} and the downwards-pointing crossing. The most interesting relation is \eqref{twocolorrelnsing}. Let $j = b$ when $m$ is odd, and $r$ when $m$ is even. Temperley-Lieb
theory implies that $b_{\ul{jpb}}$ is sent to the image of $JW_{m-1}$, when viewed on the planar strip as an endomorphism of $\ul{j \emptyset \ldots r \emptyset b}$. Then
\eqref{SingJWRelation} says that this image actually factors as a map $\ul{j \emptyset \ldots r \emptyset b} \to \ul{jpb} \to \ul{j \emptyset \ldots r \emptyset b}$. The reader can finish
the deduction that $\ul{jpb}$ is isomorphic to this image.

Therefore, $\DG_m$ induces a trace on $\HG_m$, which can be identified by its values on $\End_{\HG_m}(\ul{\emptyset}) = \HB_m$. This trace $\e$ is determined by $\e(b_{\hat{\ul{k}}_s})$
and $\e(b_{\hat{\ul{k}}_t})$ for $k \le m$. These in turn are determined by the graded dimensions of the Hom spaces specified in Corollary \ref{woohoo} (giving the formula for all $k <
2m$), which agree with the graded ranks of Hom spaces in $\DGti$. Therefore, the trace induced by $\DG_m$ is equal to the standard trace. The remaining arguments proceed as usual.
\end{proof}

%
\subsection{The category $\DC_m$}
\label{sec-mdm}
%
%

In order to define the diagrammatic version of $\BSBim$, we need make fewer assumptions than for $\SBSBim$. In particular, we only need to assume Demazure surjectivity. However, the proof
of the SCT in this paper also relies on local non-degeneracy and lesser invertibility as well. For a more general proof, see \cite{EWGR4SB}.

\subsubsection{{\bf Definitions}}\label{defnsfinite}

\begin{defn} A \emph{Soergel graph} for $m<\infty$ is an isotopy class of a particular kind of graph with boundary, properly embedded in the planar strip (so that the boundary of the graph
is always embedded in the boundary of the strip). The edges in this graph are colored by either $s$ or $t$. The vertices in this graph are either univalent (dots), trivalent with all three
adjoining edges having the same color, or $2m$-valent with alternating edge colors. Polynomials in $R$ can be placed in any region. The boundary of the graph gives two sequences of colors,
the \emph{top} and \emph{bottom boundary}. Soergel graphs have a degree, where trivalent vertices have degree $-1$, dots have degree $1$, and $2m$-valent vertices have degree $0$.
\end{defn}

An \emph{$\infty$-graph} will be a graph without $2m$-valent vertices.

\begin{defn} Let $\DC_m$ be the $\Bbbk$-linear monoidal category defined herein. The objects will be finite sequences $\ul{w}$ of indices $s$ and $t$, with a monoidal structure given by
concatenation. The space $\Hom_{\DC_m}(\ul{w},\ul{y})$ will be the free $\Bbbk$-module generated by Soergel graphs with bottom boundary $\ul{w}$ and top boundary $\ul{y}$, modulo the
relations below. Hom spaces will be graded by the degree of the Soergel graphs.

We have all the relations that define $\DCti$ (see section \ref{defnsinfty}) as well as two new relations, called the \emph{two-color relations}, found in equations \eqref{assoc2} and
\eqref{dot2m}. They hold with the colors switched as well. It is difficult to draw these relations for all $m$ at once, since the number of strands entering a vertex changes. A circle
labelled $JW$ contains the Jones-Wenzl projector $JW_{m-1}$ as a degree $2$ map (see Notation \ref{degree2JWnotation}), and a circle labelled $\textrm{v}$ contains the $2m$-valent vertex.
A sequence of a few purple lines will indicate an alternating sequence of red and blue lines of the appropriate length (depending on $m$).

The new relations are \emph{two-color associativity}: 

\begin{equation} \label{assoc2} \ig{1}{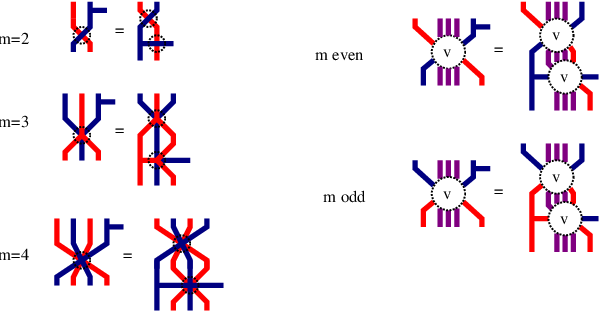} \end{equation}

and \emph{dotting the vertex}:

\begin{equation} \label{dot2m} {
\labellist
\small\hair 2pt
 \pinlabel {$+ [2]$} [ ] at 95 21
 \pinlabel {$+ [2]$} [ ] at 150 21
\endlabellist
\centering
\ig{1}{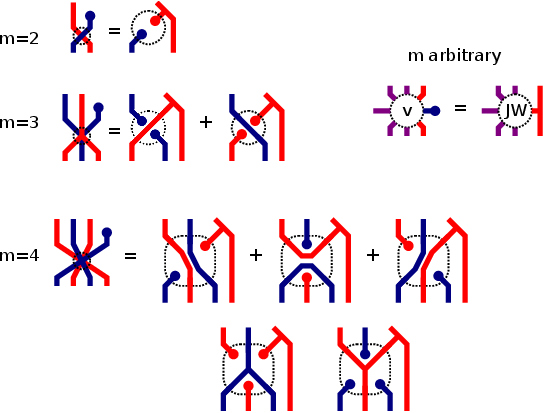}
}
 \end{equation}

\end{defn}

Each figure has some examples, with the $2m$-valent vertex in these relations circled. In writing the coefficients of the Jones-Wenzl projector, we have assumed that $[m-1]=1$.

Now we derive some other relations in $\DC_m$. The following two pictures become the same after an application of \eqref{assoc2}. \begin{equation} \label{tricommutes} \ig{1}{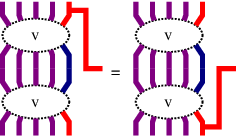}
\end{equation}

Relation \eqref{dot2m} implies that the $2m$-valent vertex is killed by any pitchfork. Thus we may deduce that \begin{equation} \label{vJWisv} \ig{1}{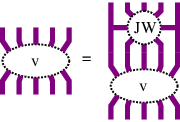} \end{equation} Only one term
in $JW$ survives: the term yielding the identity map. The other terms produce pitchforks.

\begin{claim} The following relation holds. (Note: this will be the idempotent which projects onto $B_{w_0}$.) \end{claim}

\begin{equation} \label{w0decomp} 	{
	\labellist
	\small\hair 2pt
	 \pinlabel {$+ [2]$} [ ] at 85 21
	 \pinlabel {$+ [2]$} [ ] at 142 21
	\endlabellist
	\centering
	\ig{1}{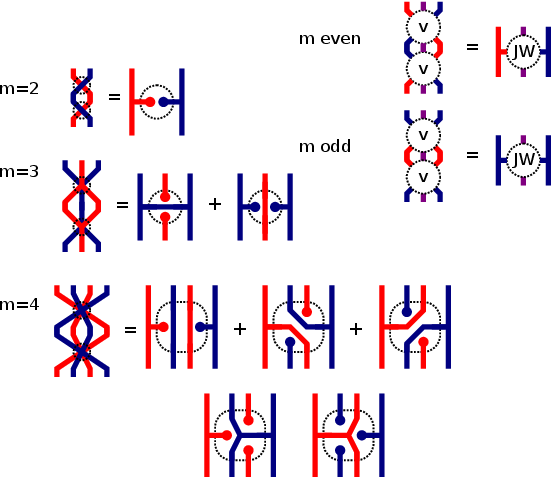}
	} \end{equation}

\begin{proof} We purposely present two proofs, for study in the Appendix. For the first, use in order: \eqref{unit}, \eqref{tricommutes}, \eqref{dot2m}, \eqref{vJWisv}, and \eqref{dot2m}.
The coloration is as though $m$ is even, although if $m$ is odd one only need change the color on the leftmost strands. \igc{1}{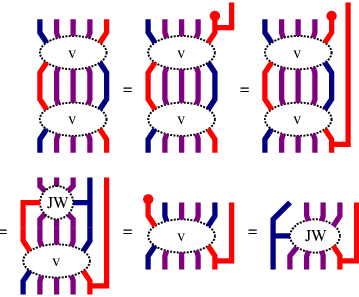}

For the second proof, merely apply a dot in the correct place to \eqref{assoc2}, in such a way that the RHS of \eqref{assoc2} becomes the LHS of \eqref{w0decomp}. \end{proof}

We invite the reader to compare relation \eqref{assoc2} with relation (3.7) in \cite{EInduced}.

\subsubsection{{\bf Functors}}\label{functorsfinitesing}

\begin{defn} Assume local nondegeneracy and lesser invertibility, so that $\DG_m$ is well-defined. We give a functor $\iota_m \colon \DC_m \to \Hom_{\DG_m}(\emptyset,\emptyset)$. On
objects, it sends $s$ to the path $\ul{\emptyset b \emptyset}$ and $t$ to the path $\ul{\emptyset r \emptyset}$. We define the functor on dots and trivalent vertices as in Definition
\ref{iotadefn}. The image of the $2m$-valent vertex is $v_m$.

\igc{1}{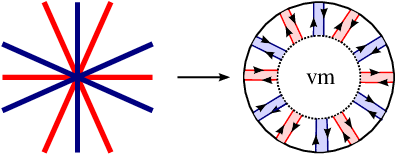} \end{defn}

\begin{claim} The above definition gives a well-defined functor. \label{welldefinedfunctormfinite} \end{claim}

\begin{proof} In Definition \ref{iotadefn} we already had a well-defined functor for $\infty$-graphs. We need only check the relations involving $2m$-valent vertices. Relation
\eqref{dot2m} follows from \eqref{vkprop} and \eqref{SingJWRelation}. Relation \eqref{assoc2} follows from \igc{1}{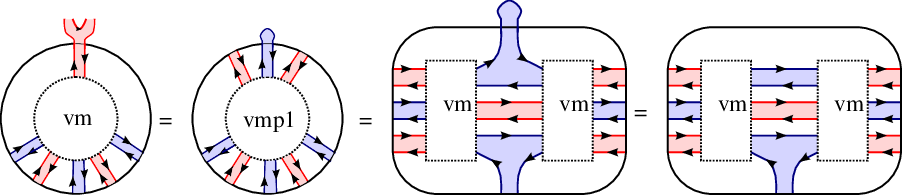} \end{proof}

\begin{claim} This functor is full. \end{claim}

\begin{proof} This is precisely the statement of Lemma \ref{itssurjective}. \end{proof}

Now $\iota_m$ is not essentially surjective, because there are loops based at $\emptyset$ which pass through $p$, but it will be surjective after passage to the Karoubi envelope. After
all, one can already see that $\ul{\emptyset p \emptyset}$ is the image of the Jones-Wenzl projector inside $BS(\hat{\ul{m}}_s)$, by rewriting this Jones-Wenzl projector using $v_{m-1}$.

\begin{defn} We define the functor $\FC_m \co \DC_m \to \SBim_m$ as the composition of $\iota_m$ and $\FG_m$. \end{defn}

Clearly this functor agrees with $\FC_\infty$ on $\infty$-graphs, since $\FG_m$ agrees with $\FG_\infty$ on $\infty$-diagrams. We have avoided giving an explicit formula for the image of
the $2m$-valent vertex, instead describing it as the image of $v_m$, which itself is a composition of numerous cups, caps, and crossings. It seems that a straightforward formula for this
composition is quite nasty in general. The formula for $m=2,3$ was given in \cite{EKh}.

\subsubsection{{\bf Graph reduction}}\label{graphreductionagain}

Relation \eqref{gettinghighvk} told us how to construct the singular Soergel graph $v_k$ out of $v_m$, and we can perform the same construction with Soergel graphs with the $2m$-valent
vertex replacing $v_m$. This yields a well-defined Soergel graph $v_k$ for all $k \ge 0$. It is not difficult to prove that $v_k$ is rotation-invariant for all $k$, using either
increasing or decreasing induction from $v_m$. Therefore, $v_k$ is killed by any pitchfork. It is not terrible to duplicate the results of Claim \ref{connectingwithbands}, though now the
proof should go the other way, starting with $v_{k+l+1-n}$ with some polynomials, and using the dot forcing rules to break lines and resolve until obtaining $v_k$ attached to $v_l$.

\begin{prop} \label{reducessans2m} Any morphism in $\DC_m$ on a planar disk with an alternating boundary of length $<2m$ reduces to a sum of $\infty$-graphs. \end{prop}

\begin{proof} This proof is entirely analogous to Claim \ref{reducetoinfty}, and we leave it as an exercise to the reader. \end{proof}

\begin{lemma} \label{polys} The functor $\FC_m$ induces an isomorphism $\End_{\DC_m}(\emptyset) \cong R$. \end{lemma}

\begin{proof} Proposition \ref{reducessans2m} shows that all maps reduce to boxes. The functor to bimodules gives us a surjective map from a rank 1 $R$-module to
$\End_{\SBim}(B_e)=R$, which must be an isomorphism. \end{proof}

The upshot is that the new relations \eqref{assoc2} and \eqref{dot2m} do not impose any new relations on polynomials.

\subsubsection{{\bf The Grothendieck group}}\label{grothgrpfinite}

There is clearly a map from $\HB_\infty \to [\Kar(\DC_m)]$, because \eqref{iisplittingeqn} still holds. One can define all the idempotents in $\DC_m$ coming from $\mTTL$ for elements $w
\in W_\infty$ of length $\le m$. We call the image of the corresponding Jones-Wenzl $B_w \in \Kar(\DC_m)$, as before. We make no claim yet that these are indecomposable or non-zero, but
we do have $b_w \mapsto [B_w]$ for $\ell(w) \le m$. Relation \eqref{w0decomp} implies that, in the Karoubi envelope, the $2m$-valent vertex is precisely an isomorphism from $B_{\hat{m}_s}
\to B_{\hat{m}_t}$, whose inverse isomorphism is its own rotation. Therefore, $\Kar(\DC_m)$ is a potential categorification of $\HB_m$.

\begin{thm} \label{mainthmfinite} Assume lesser invertibility and local non-degeneracy. Then the SCT and the Soergel conjecture hold for $\DC_m$. For any Soergel ring, $\FC_m$ is an
equivalence of categories, sending indecomposables $B_w$ to $B_w$. In addition, $\iota_m$ is fully faithful. \end{thm}

\begin{proof} As in Theorem \ref{mainthminfty}, this is a simple application of Lemma \ref{homspacelemma} and Corollary \ref{homspacecor}. All that remains is to check whether the trace
on $\HB_m$ induced by the map $\HB_m \to [\Kar(\DC_m)]$ agrees with the standard trace. As in the proof of Theorem \ref{mainthminfty}, using Proposition \ref{reducessans2m} it is easy to
show that the Hom space from $B_{\hat{k}_s}$ to $B_e$ is generated over $R$ by the all-dot, and that it is free of rank $1$ over $R$ comes from the functor $\FC_m$.

The functor $\FC_m$ factors through $\iota_m$ and thus $\iota_m$ is faithful.\end{proof}

%
\subsection{Thickening}
\label{sec-thick}	
%
%

Suppose that $m<\infty$ and continue to assume lesser invertibility and local non-degeneracy.

\subsubsection{Diagrams for $\fooBim$}\label{thickdiagrams}

Recall that $\fooBim$ is the full subcategory of $\SBim$ monoidally generated by $B_s$, $B_t$, and $B_W$. We present this category diagrammatically, in precise analogy with \cite{EInduced},
chapter 3.5.

\begin{defn} A \emph{thick Soergel graph} has edges labelled either $s$, $t$, or $W$ (purple). The new vertices (compared to a Soergel graph) are: trivalent with 3 purple edges (degree
$-m$); trivalent with two purple and one other (degree $-1$); univalent with one purple (degree $m$); and $(m+1)$-valent with one purple edge and the remainder alternating between $s$ and
$t$ (degree 0). \end{defn}

\begin{defn} The category $\DCfoo$ has morphisms given by thick Soergel graphs, with the relations of $\DC_m$ as well as the following relations.
\begin{equation} \label{defthick} \ig{1}{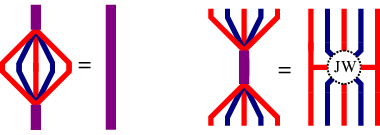} \end{equation}
\begin{equation} \label{defthicktri} \ig{1}{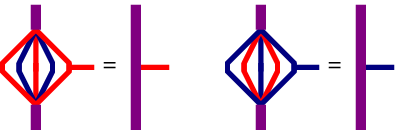} \end{equation}
\begin{equation} \label{defthickthicktri} \ig{1}{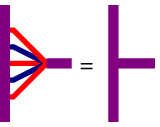} \end{equation}
\begin{equation} \label{defthickdot} \ig{1}{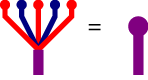} \end{equation}
\end{defn}

Relation \eqref{defthick} identifies the purple line as the image of the idempotent which picks out $B_{w_0}$ inside $BS(\hat{m})$. The remaining equalities identify the new generators as
pre-existing maps in $\SBim$. Therefore, the fact that this category is equivalent to $\fooBim$ is entirely obvious. We can also describe these morphisms within $\DG_m$ as follows.

\igc{1}{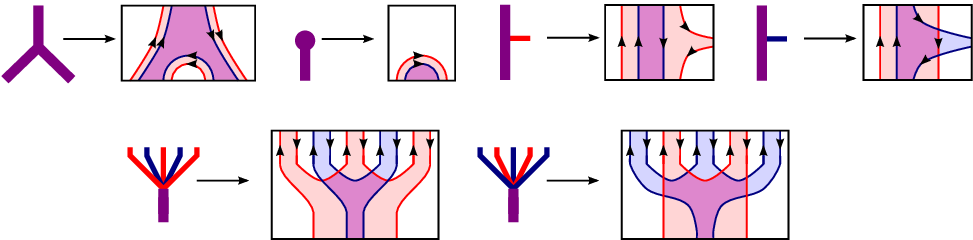}

We encourage the reader to check \eqref{defthickthicktri}, which will involve evaluating \eqref{ckis1} as in the proof of \eqref{gettinghighvk}.

This defines a functor from $\DCfoo$ to $\End_{\DG_m}(\emptyset)$, which sends the purple object to $\ul{\emptyset rpr\emptyset}$. There is an isomorphic functor sending purple to
$\ul{\emptyset bpb \emptyset}$, passing through blue instead of red.

\begin{prop} We also have the following equalities.
\begin{equation} \label{thickunits} \ig{1}{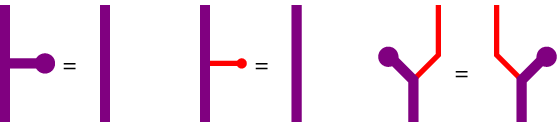} \end{equation}
\begin{equation} \label{thickassocs} \ig{1}{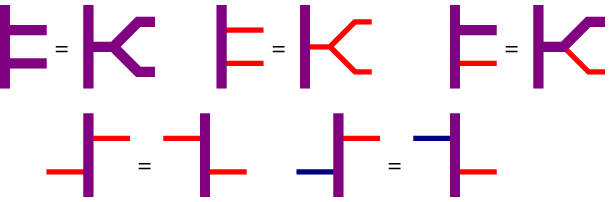} \end{equation}
\begin{equation} \label{thickneedle} \ig{1}{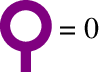} \end{equation}
\begin{equation} \label{thickdotforce} \ig{1}{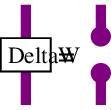} \end{equation}
\begin{equation} \label{thickmeets2m} \ig{1}{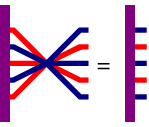} \end{equation}	
\end{prop}

\begin{proof} These are each very easy to show within $\DG_m$. Checking \eqref{thickmeets2m} also requires \eqref{ckis1} as above. The remaining relations then follow from isotopy or
Frobenius extension relations. \end{proof}

\begin{remark} Note also that neither the $2m$-valent vertex nor the $m+1$-valent vertex are actually required, in the presence of the other maps. For instance: \igc{1}{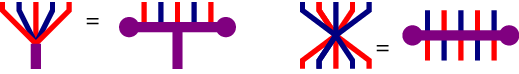}
\end{remark}

There are perhaps many more interesting equalities to find.

\subsubsection{Induced modules}\label{induced}

As in section \ref{InducedInfty}, we may represent $\Hom_{\DG_m}(\emptyset,I)$ for $I \subset \{s,t\}$ simply using Soergel graphs with a shaded region, in precise analogy with Chapter 4
of \cite{EInduced}. When $I=\{s\}$ or $I=\{t\}$, we have already described the answer for $\DGti$ in section \ref{InducedInfty}, and the result for $\DG_m$ is identical, except that we have
additional relations between Soergel graphs. When $I=\{s,t\}$, we require that graphs end in a ``purple region," and add new morphisms corresponding to trivalent vertices with the purple
region.

Explicitly, the new generators are \igc{1}{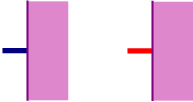} The new relations are \begin{equation} \label{InducedRelationsPurp} \ig{1}{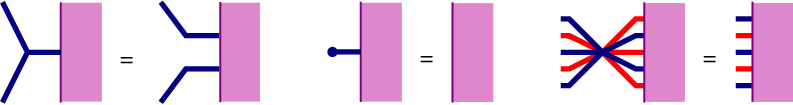}\end{equation} Any usual
Soergel graph, with the usual Soergel relations, may be drawn to the left of the purple region.

Again, the equivalence between these diagrams and certain singular Soergel diagrams is easy.

%
\subsection{Temperley-Lieb categorifies Temperley-Lieb}
\label{sec-TLcatsTL}	
%
%

Let $m < \infty$. Recall that the \emph{generalized Temperley-Lieb algebra} $TL_{W_m}$ of the finite dihedral group $W_m$ for $m>2$ is the quotient of $\HB_m$ by $b_{w_0}$. According to
the original definition in \cite{GraThesis}, one should take $TL_{W_2} = \HB_2$. However, for the purposes of this chapter, we let $TL_{W_2}$ be the quotient of $\HB_2$ by $b_{w_0}$.
Recall that the 2-sided ideal of $b_{w_0}$ is none other than the $\Zvv$-span of $b_{w_0}$.

Similarly, the \emph{generalized Temperley-Lieb algebroid} is the quotient of $\HG_m$ by morphisms factoring through the object $p$.

\begin{thm} \label{TLWcatfnthm} Consider the quotient of $\DC_m$ by the $2m$-valent vertex. This categorifies $TL_{W_m}$. The degree $0$ Hom spaces are given precisely by $\mTTL_{\negl}$.
Similarly, the quotient of $\DG_m$ by the purple region categorifies the generalized Temperley-Lieb algebroid of $W_m$. \end{thm}

Here is a sketch of the proof. There is clearly a map from $TL_W$ to the Grothendieck ring of this quotient, because the map from $\HB_m$ factors through the ideal $b_{w_0}=0$. Therefore,
this quotient induces a trace map on $TL_W$. We still have idempotents which yield objects $B_w$, $w \in W \setminus \{w_0\}$, and though we have not yet shown that these objects are
nonzero, we do know that all other objects can be expressed as direct sums of these. If they remain indecomposable and pairwise non-isomorphic in the quotient, then they will descend to a
basis of $TL_W$, and the map from $TL_W$ will be an isomorphism. This, in turn, will follow from the calculation of the trace on $TL_W$, because the graded rank of $\End(B_w)$ will be in
$1 + v\Z[v]$, and the graded rank of $\Hom(B_w,B_x)$ for $w \ne x$ will be in $v\Z[v]$.

In order to calculate the trace map, we must determine what elements of $\Hom(BS(\ul{w}),B_e)$ survive in the quotient, for each reduced expression $\ul{w}$. This is a diagrammatic
calculation.

For instance, what will be $\End(B_e)$ in the quotient category? We rephrase this in singular language: consider a Soergel 1-manifold diagram with a purple region and with empty
white boundary. We know that this reduces to a polynomial in $R$, but which polynomials can appear? We claim that the polynomials which appear are precisely the ideal generated by
$\LM$, and that therefore $\End(B_{\emptyset}) \cong R/\LM$.

We can use an argument similar to the proof of Proposition \ref{reducessans2m} to show that any diagram containing $v_m$ can be reduced to a diagram containing $v_k$ once for $k \ge m$,
and has no other purple regions in the same connected component. As in the proof that $\Hom(B_{w_0},\emptyset)=R$, one can show that the only way to get a nonzero map for that connected
component is to attach the all-dot to $v_k$, yielding $\LM$. This calculation is done for the case $m=3$ in \cite{ETemperley}.

Similarly, $\Hom(B_s,B_{\emptyset})$ should be isomorphic to $R/(\frac{\LM}{\a_s})$, where each polynomial is placed next to the blue dot. This kernel is generated by the $2m$-valent
vertex with all but one dot attached. Similarly, if $w = sts\ldots$ has length $k$ then $\Hom(B_w,B_\emptyset)$ will be generated by the all-dot, and the kernel will be generated by the
$2m$-valent vertex with $k$ strands attached to the boundary, and the remaining strands dotted.

Calculating $\Hom(B_w,B_e)$ for each $w$, we pin down the trace precisely. One then checks that the graded ranks of Hom spaces satisfy the desired properties above. This concludes the
sketch.

After investigation, we see that the morphism spaces in this quotient category, as $R$-modules, are supported on the union of the Coxeter lines. Moreover, one can associate to each $w \in
W$ a set of positive roots of size $l(w)$, such that $\Hom(B_w,B_e)$ is supported on the complement of those root hyperplanes.

%% file: ExoticDihedral.tex
In this appendix we discuss how to modify the statements and proofs of this paper to account for unbalanced and non-symmetric realizations. We do this by introducing two-colored quantum
numbers, which make the computations quite analogous to the symmetric case. However, one must keep track of additional data in order to construct a Frobenius hypercube. This appendix is
written so that it may be read in conjunction with the main paper, after reading the corresponding section there. For a more detailed version of many of these computations, see the
author's PhD thesis \cite{EThesis}.

\subsection{{\bf Two-Colored Quantum Numbers}}\label{2qnum}

Section \ref{qnum} gave a number of facts about quantum numbers, both inside the ring $\Z[\d]$ and inside the specialization where $q$ is a root of unity. Our new ring to replace $\Z[\d]$
will be $\Z[x,y]$. We identify the subring $\Z[xy]$ with the subring $\Z[\d^2]$ via $xy=\d^2$. We think of $x$ and $y$ as alternate versions of $[2]$, which need to be balanced in a more
complicated quantum number. A \emph{symmetric specialization} of $\Z[x,y]$ is a specialization factoring through the map to $\Z[\d]$ sending $x$ and $y$ to $\d$.

\begin{defn} We define \emph{two-colored quantum numbers}, which are elements of $\Z[x,y]$ analogous to the quantum numbers in $\Z[\d]$. Because $[2k+1]$ and $\frac{[2k]}{[2]}$ are both
inside $\Z[\d^2]$, we may express them as polynomials in $xy$. When $m$ is odd, we define $[m] = [m]_x = [m]_y$. When $m$ is even, we define $[m]_x = x \frac{[m]}{[2]}$ and $[m]_y = y
\frac{[m]}{[2]}$, where in both cases $\frac{[m]}{[2]}$ represents the corresponding polynomial in $xy$. \end{defn}

\begin{example} $[2]_x = x$ and $[2]_y = y$; $[3]_x=xy-1=[3]_y$; $[4]_x = x^2y - 2x$ and $[4]_y = xy^2 - 2y$. \end{example}

Each of the facts about quantum numbers has a two-colored analogue, derived in essentially the same way. Instead of the usual recurrence relation, one uses $[2]_x [m]_y = [m+1]_x +
[m-1]_x$ for $m \ge 1$. To ``specialize $q^2 = \z_m$ to a primitive root of unity," we set $Q_m(xy)=0$. When $m$ is odd, it seems more difficult to distinguish between the specializations
$q = \z_{2m}$ and $q=\z_m$, because $Q_m(xy)$ does not split. For example, when $m=3$, $Q_m(\d^2) = \d^2-1 = (\d+1)(\d-1)$, while $Q_m(xy) = xy-1$. However, when $m$ is even it is still
true that $Q_m(xy)=0$ implies that $[m-1]=1$.

When $m$ is even, $x[m]_y = y [m]_x$. In a domain with $xy \ne 0$ one has $[m]_x = 0 \iff [m]_y=0$. If $[m]_y=0$, we still have $[m-k]_x = [m-1] [k]_x$, so that $[m-1]^2 = 1$. If $[2m]_x =
[2m]_y = 0$ and $[2m-1]=-1$, one can deduce that $2[m]_x = [2]_y [m]_x = 0$ and $2[m]_y = [2]_x [m]_y = 0$, just as in Claim \ref{claim-qnumber-bullshit}.

When $m$ is odd and $[m]=0$, a great simplification occurs. Now $x$ is invertible because it divides $[m-1]_x$ and $[m-1]_x[m-1]_y = 1$. If $Q_m(xy)=0$ then there is actually a polynomial
in $xy$, namely $[m-2]$, which satisfies $P_m([m-2])=0$, the algebraic conditions to be equal to $[2]_q$ at $q = \z_{2m}$! We write $[2]_m$ to represent $[m-2]$, the element which behaves
like ``quantum 2" should. In other words, when $m$ is odd $\Z[x,y]/Q_m(xy)$ is no more than the extension of $\Z[[2]_m]$ by the invertible variable $\l = \frac{[2]_m}{x} = \frac{y}{[2]_m}$.
Note also that $[2]_x [m-1]_y = [m-2] = [2]_m$, so that $\l = [m-1]_y$ and $\l^{-1} = [m-1]_x$. Now there are two symmetric specializations, $\l=\pm 1$, which correspond to the
specializations $q=\z_{2m}$ and $q=\z_m$.

This rescaling factor $\l$ will appear numerous times below. There is no general definition of $\l$ in the case where $m$ is even or infinite, as there is no element in $\Z[x,y]/Q_m(xy)$
which behaves like $[2]_m$.

\subsection{{\bf Realizations}}\label{reflrepexotic}

\begin{defn} \label{defnrealization} Let $\Bbbk$ be a commutative ring and $(W,S)$ be a Coxeter system. A \emph{realization} of $(W,S)$ over $\Bbbk$ is a free, finite rank $\Bbbk$-module
$\hg$, together with a choice of simple co-roots and roots having a Cartan matrix $(a_{s,t}) _{s,t \in S}$ satisfying: \begin{enumerate} \item $a_{s,s} = 2$ for all $s \in S$; \item for any
$s,t \in S$ with $m_{st}<\infty$, if $\Bbbk$ is given a $\Z[x,y]$-algebra structure where $-x = a_{s,t}$ and $-y=a_{t,s}$, then $[m_{st}]_x=[m_{st}]_y=0$; \item the assignment $s(v) \define
v - \langle v, \alpha_s\rangle \alpha_s^{\vee}$ for all $v \in \hg$ yields a representation of $W$. \end{enumerate} \end{defn}

When analyzing a particular dihedral subgroup, we always use the convention that the Cartan matrix is \[A = \left( \begin{array}{cc} 2 & -x \\ -y & 2 \end{array}\right), \] living inside
some specialization $\Bbbk$ of $\Z[x,y]$. The determinant of this matrix is $4-xy$, which replaces $4-[2]^2$ in all previous formulas.

We define the action of $W$ on $\hg^*$ as before. The formulae for the action of $(st)^k$ go through with minor adaptations. The action on the span of $\a_s$ and $\a_t$ is given by
\begin{equation} \label{stformulaalt} (st)^k = \left( \begin{array}{cc} {[}2k+1] & -[2k]_x \\ {[}2k]_y & -[2k-1] \end{array} \right). \end{equation} The action on $\a_u$ is given by
\begin{equation} \label{stonuformulaalt} (st)^k(\a_u) = \a_u + ([k]_x[k]_y a_{s,u} + [k]_x[k+1]_x a_{t,u}) \a_s + ([k]_y[k-1]_y a_{s,u} + [k]_x[k]_y a_{t,u}) \a_t. \end{equation} In
particular, for a domain or modulo 2-torsion, $(st)^m$ is trivial if and only if $[m]_x = [m]_y = 0$. As in the symmetric case, outside of degenerate situations there is redundancy between
the fact that $[m_{st}]=0$ and the fact that there is an action of $W$.

\begin{defn} A realization is \emph{balanced} if for each $s,t \in S$ with $m_{st}<\infty$ one has $[m_{st}-1]_x = [m_{st}-1]_y = 1$. The notions of even-balanced, odd-balanced, etc are
easy to extrapolate. \end{defn}

There is an enormous difference between realizations which are unbalanced for $m$ even, and those which are unbalanced for $m$ odd! This will be a common theme. Let us momentarily consider
only dihedral realizations where $\Bbbk$ is a domain, and where $\hg$ is spanned by the coroots.

Suppose that $m$ is even. If the realization is faithful then $Q_m(xy)=0$ so that $[m-1]=1$, and the realization is automatically balanced. If the realization is not faithful it is
quite possible that it is not balanced. Note that a faithful realization need not be symmetrizable, but only because $\Bbbk$ may not contain enough scalars; if $\Bbbk$ is a field and $xy$
has a square root, the Cartan matrix is symmetrizable.

Suppose that $m$ is odd. If the realization is faithful then it is balanced if and only if it is symmetric and $[m-1]=1$ (i.e. if $\l=1$, where $\l$ was defined in the last section).
Root-rescaling by the diagonal matrix with entries $(\l,1)$ will yield a symmetric matrix. It will be balanced as a faithful dihedral realization, but need not be balanced when viewed as a
non-faithful dihedral realization of $W_{mk}$ (the dihedral group with $2mk$ elements, for some $k \ge 2$).

\begin{example} The best example of a non-symmetrizable non-balanceable Cartan matrix that still plays a significant role is \emph{exotic affine} $\sl_n$ for $n \ge 3$. When $n=4$, it is
given by the following matrix over $\Bbbk=\Z[q^\pm]$: \[A = \left( \begin{array}{ccccc} 2 & -1 & 0 & 0 & -q^{-1} \\ -1 & 2 & -1 & 0 & 0 \\ 0 & -1 & 2 & -1 & 0 \\ 0 & 0 & -1 & 2 & -q \\ -q &
0 & 0 & -q^{-1} & 2 \end{array} \right).\] Note that the Coxeter group which acts faithfully in rank 2 does not depend on the specialization of $q$, and thus agrees with the case $q=1$,
which is the usual affine $\sl_n$ Cartan matrix. This matrix is even-balanced and odd-unbalanced. The matrix for exotic affine $\sl_2$ is symmetric and balanced, given by \[A = \left(
\begin{array}{cc} 2 & -(q+q^{-1}) \\ -(q+q^{-1}) & 2 \end{array} \right). \] Now, of course, the Coxeter group which acts faithfully in rank 2 can change when $q$ is specialized. In
\cite{EQAGS} we will explain how Soergel bimodules for exotic affine $\sl_n$ give rise to a quantum Satake equivalence. \end{example}

\subsection{{\bf Realizations and roots}}\label{rootsanddem}

Outside of the discussion of roots, everything up to section \ref{dihinvt} can be followed verbatim.

When $m$ is even or infinite, the definition of positive roots is unchanged. However, when $m$ is odd, $f_{s,m-1} = [m-1]_y f_t = \l f_t$. Thus $f_{s,m-1-l} = \l f_{t,l}$ for all $0 \le l
\le m-1$. With positive roots defined only up to a scalar, one must make some conventional choices.

In the calculation of the associated polynomial of a Jones-Wenzl projector (see section \ref{coxlines}, and section \ref{subsec-exoticTL} below) it was useful to define a snakelike order
on the set of positive roots for the infinite dihedral group, either an $s$-aligned or a $t$-aligned snakelike order. The $s$-aligned version is $f_{s,0} < f_{t,0} < f_{t,1} < f_{s,1} <
f_{s,2} < f_{t,2} < \ldots$. 

\begin{defn} The \emph{$s$-aligned choice of roots} $\LC^{(s)}$ for the finite dihedral group $W_m$ are the first $m$ roots in the $s$-aligned snakelike order. The $t$-aligned choice of
roots is defined analogously. \end{defn}

It is easy to observe that the only difference between these choices occurs when $m=2k+1$ is odd, where $\LC^{(s)}$ contains $f_{s,k}$ and $\LC^{(t)}$ contains $f_{t,k}$. Letting
$\LM^{(s)}$ be the product of the elements of $\LC^{(s)}$, we have $\LM^{(s)} = \l \LM^{(t)}$. These are not the only choices of positive roots, of course, but they will be the most
convenient for our calculations. Choosing a set of positive roots for Coxeter groups of rank $\ge 3$ will require more bookkeeping.

The main conceptual difference between balanced and unbalanced Cartan matrices is the following claim.

\begin{claim} Suppose that $m=m_{s,t} < \infty$. Then the simple Demazure operators satisfy the braid relation $\ubr{\pa_s \pa_t \ldots}{m} = \ubr{\pa_t \pa_s \ldots}{m}$ when $m$ is
even. When $m$ is odd, $\ubr{\pa_s \pa_t \ldots \pa_s}{m} = \l^{-1} \ubr{\pa_t \pa_s \ldots \pa_t}{m}$. \end{claim}

\begin{proof} (Sketch) As in the suggested brute force proof for the symmetric case, we can write the iterated Demazure operator applied to $f$ as a sum of terms of the form $\frac{\pm
w(f)}{\pi}$ where $w \in W$ and $\pi$ is a product of $m$ roots. If one finds a formula for the products $\pi$ which appear for each side of the braid relation, and matches them (up to
scalar, using the identification of $f_{s,l}$ with $f_{t,m-l}$ in the odd case), one will end up with the desired result. \end{proof}

When $m$ is odd, there is a simpler proof, using the fact that the Cartan matrix is symmetrizable. Rescaling $\a_s$ by $\l$ (from the symmetric case) will perforce rescale $\pa_s$ by
$\l^{-1}$, and this rescaling factor will affect the LHS one more time than the RHS.

In particular, for $w \in W$ with multiple reduced expressions, one can only define the operator $\pa_w$ up to scalar.

\begin{claim} \label{choiceofposroots} When $m=2k$, $\ubr{\pa_s \pa_t \ldots \pa_t}{m}(\LM) = \ubr{\pa_t \pa_s \ldots \pa_s}{m}(\LM) = 2m$. When $m=2k+1$, $\ubr{\pa_s \pa_t \ldots
\pa_s}{m}(\LM^{(s)}) = 2m$ and $\ubr{\pa_t \pa_s \ldots \pa_t}{m}(\LM^{(t)})=2m$. \end{claim}

\begin{proof} An annoying exercise for the reader. \end{proof}

\subsection{{\bf Frobenius structures}}\label{frobstructs}

Given any Frobenius extension $A \subset B$ with trace $\pa$ and $\mu(\Delta(1))=\LM \in B$, there is a one-parameter family of Frobenius extension structures having trace $\l^{-1} \pa$
and coproduct-product $\l \LM$, for some invertible scalar $\l \in \Bbbk$. One can pin down this scalar by choosing $\pa$ or by choosing $\LM$; one determines the other by the requirement
that $\pa(\LM)=n$, the rank of the extension.

As discussed previously in the dihedral case, there will be no convenient Frobenius extension structure on $R^I \subset R$ when $m < \infty$ unless the realization is faithful for the
parabolic subgroup $W_I$. If it is faithful, however, then $\pa_{w_I}$ (up to scalar) is a Frobenius trace map. To define a Frobenius hypercube structure on the invariant rings $R^I
\subset R^J$, one should let $\pa^J_I \co R^J \to R^I$ be the Demazure operator associated to the relative longest element $w_I w_J$, which is only defined up to scalar. Then, one should
normalize these scalars so that the Frobenius hypercube is compatible.

In the unbalanced odd dihedral case, normalization is required. There is only one choice of reduced expression for $w_0 s$ and for $w_0 t$, and $\pa_{w_0 s} \pa_s = \l^{-1} \pa_{w_0 t}
\pa_t$.

\begin{defn} A \emph{Frobenius realization} is the data of a faithful realization of a Coxeter group $W$, together with a Frobenius hypercube structure. More precisely, for all finitary
subsets $I$ with $J = I \setminus \{i\}$, one fixes a reduced expression for $w_I w_J$ so that one has an unambiguous operator $\pa_{w_I w_J}$. Then one chooses scalars $\l^J_I$, and sets
$\pa^J_I = \pa_{w_I w_J} \l^J_I.$ These must satisfy \begin{itemize} \item $\l_s = 1$ for all $s \in S$. \item Whenever $K \subset J,J' \subset I$ is a square in the poset of $S$, one has
$\pa^J_I \pa^K_J = \pa^{J'}_I \pa^K_{J'}$. \end{itemize} \end{defn}

When the realization is balanced, there is a canonical choice of Frobenius realization, with $\l^J_I=1$ for all $J \subset I$. For a dihedral group, the difference between a Frobenius
realization and a usual realization is merely the choice of one arbitrary invertible scalar in $\Bbbk$.

Instead of fixing a family of scalars as in the definition above, it may be preferable to fix a system of positive roots for $W$. For a faithful realization, it is not difficult to see
that the lines spanned by the roots are well-defined, even though the choice of positive root within that line may not be.

\begin{defn} A \emph{root realization} is the data of a faithful realization of a Coxeter group $W$, together with a choice of positive roots. More precisely, for each distinct line
spanned by $w(\a_s)$ or $w(\a_t)$ in $\hg^*$, one chooses a non-zero vector to be the corresponding positive root. One requires $\a_s$ and $\a_t$ to be chosen. \end{defn}

Given a root realization, one can obtain a Frobenius realization as follows. Given $I \subset S$, a root is a \emph{root for $W_I$} when it lies on the line of $w(\a_s)$ for some $w \in
W_I$ and $s \in I$. Let $\LM^J_I$ be the product of the positive roots for $W_I$ that are not roots for $W_J$. Now fix the scalars $\l^J_I$ in order that $\pa^J_I(\LM^J_I)$ is the size of
$W_I/W_J$. It is easy to see that these scalars yield a compatible Frobenius hypercube. Given a Frobenius realization, there may be multiple choices of root realization giving rise to it.

(The discussion of the previous paragraph doesn't make sense when the size of $W_I/W_J$ is a zero-divisor in $\Bbbk$. In this case, there is usually not a Frobenius extension structure
anyway, except in the easy case where $m=2$. However, I am not entirely sure.)

In \cite{EQAGS}, a choice of positive roots is made for exotic affine $\sl_n$ in such a way that certain circular singular Soergel diagrams evaluate to quantum binomial coefficients. This
illustrates that some choices are more natural than others.

We have already defined two choices of root realization for the dihedral group, $\LC^{(s)}$ and $\LC^{(t)}$. Below we reformulate Theorem \ref{frobsquare} for the choice $\LC^{(s)}$.

\begin{thm} When $m$ is even, Theorem \ref{frobsquare} holds exactly as stated even in the asymmetric case, with the Frobenius traces given. When $m$ is odd, we give $R^W \subset R$ a
Frobenius structure with trace $\pa_W = \ubr{\pa_s \pa_t \ldots}{m}$ and coproduct-product $\LM_W = \LM^{(s)}$. We give $R^W \subset R^s$ a Frobenius structure with trace
$\pa^s_W = \ubr{\ldots \pa_s \pa_t}{m-1}$ and $\LM^s_W = \frac{\LM_W}{\a_s}$. We give $R^W \subset R^t$ a Frobenius structure with trace $\pa^t_W = \l \ubr{\ldots \pa_t \pa_s}{}$ and
$\LM^t_W = \frac{\LM_W}{\a_t}$. We give $R^s \subset R$ and $R^t \subset R$ the usual Frobenius structure. The result is a Frobenius square. \end{thm}

\subsection{{\bf Additional comments about realizations}}

It is easy to come up with formulae for $z,Z \in R^{s,t}$, and thus to find an explicit description of this invariant subring, in analogy to Claim \ref{clm:dihinvt}.

Soergel worked abstractly with the representation $\hg^*$, not fixing a basis or the Frobenius structures. Therefore, his techniques still apply in the unbalanced or non-symmetric setting,
so long as $\Bbbk$ is an infinite field of characteristic $\ne 2$. In \cite{EWGR4SB}, we give an independent proof of Soergel's results directly for the diagrammatic category, without
needing to use the equivalence $\FC$. This proof will work equally well for the non-symmetric or unbalanced case.

By working formally with universal non-symmetric specializations, one has recourse to the symmetric specialization. This is a useful tool. For instance, we know that various useful
elements of $\Z[x,y]$ (like the coefficients of Jones-Wenzl projectors below) are non-zero generically, because they specialize to non-zero elements.

%
\subsection{Non-symmetric Temperley-Lieb}
\label{subsec-exoticTL}
%
%

We now redefine the two-color Temperley-Lieb 2-category $\mTTL$ as having coefficients which lie in $\Z[x,y]$. A circle with red (resp. blue) interior evaluates to $-x$ (resp. $-y$).

Two-colored Jones-Wenzl projectors $JW_n$ exist in this generality as well, and its coefficients will have two-colored quantum numbers instead of usual quantum numbers. The recursion
formulae \eqref{JWrecursive1} and \eqref{JWrecursive2} can be generalized, using two-color quantum numbers. To modify \eqref{JWrecursive2}, replace $[n+1]$ with $[n+1]_x$ if the diagram
is right-blue-aligned, and replace $[a]$ with $[a]_x$ if the interior of the new cup is blue, and $[a]_y$ if the interior is red. We give examples of the first few right-blue-aligned
projectors.

\begin{equation} \label{JWnonsym}{
\labellist
\small\hair 2pt
 \pinlabel {$JW_1 =$} [ ] at -25 108
 \pinlabel {$JW_2 =$} [ ] at 94 108
 \pinlabel {$+ \frac{1}{x}$} [ ] at 160 108
 \pinlabel {$JW_3 =$} [ ] at -25 58
 \pinlabel {$+ \frac{y}{xy-1}$} [ ] at 75 58
 \pinlabel {$+ \frac{x}{xy-1}$} [ ] at 175 58
 \pinlabel {$+ \frac{1}{xy-1}$} [ ] at 75 20
 \pinlabel {$+ \frac{1}{xy-1}$} [ ] at 175 20
\endlabellist
\centering
\ig{1}{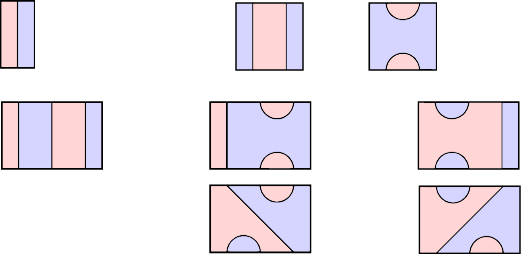}
} \end{equation}

To obtain the right-red-aligned projector, switch the colors and switch $x$ and $y$. The coefficient of the identity map should always be equal to $1$.

There are two specializations of the scalars in $\mTTL$ which occur most frequently in the literature: the \emph{spherical} specialization, identical to the symmetric specialization
$x=y=[2]$; and the \emph{lopsided} specialization, where $x=1$ and $y=[2]^2$. In fact, the general case is no more than a ``perturbation" of the spherical specialization, in the sense of
\cite{DasGhoGup}. Outside of these two specializations, references are difficult to find. The proofs, however, are completely analogous to the uncolored case.

\begin{prop} \label{2TLhasidemps} The canonical isotypic idempotents, the non-canonical primitive idempotents, and the intra-isotypic isomorphisms of Proposition \ref{TLhasidemps} all
have analogs in $2TL_n$ after localization. These maps are defined over any extension of $\Z[x,y]$ for which the two-color quantum numbers up to $[n]$ are invertible. \end{prop}

The results on the Karoubi envelope of $\Kar(\mTTL)$ are also analogous.

Whenever $[m]_x = [m]_y = 0$ (and $[m-1]=1$ for $m$ even) either two-colored $JW_{m-1}$ is rotation invariant by $2$ strands (or any color-preserving rotation). If we rotate the
right-blue-aligned $JW_{m-1}$ by one strand, one obtains the right-red-aligned $JW_{m-1}$ multiplied by a factor of $\l$. Rotating the right-red-aligned $JW_{m-1}$ by a strand, one obtains
the right-blue-aligned $JW_{m-1}$ multiplied by a factor of $\l^{-1}$.

Associated to a $2 \times 2$ Cartan matrix one has a specialization of the Temperley-Lieb 2-category. Even-unbalanced realizations behave poorly, in that their Jones-Wenzl operators
$JW_{m-1}$ are not rotation invariant; as already remarked, this can only occur over a domain for non-faithful realizations. Odd-unbalanced realizations behave well, but again call for
additional bookkeeping.

The statement and proof of Proposition \ref{coxeterlinesprop} adapt to the non-symmetric case as well. Now the scalar factor which appears is $\frac{[1]_x}{[1]_x} \frac{[1]_x}{[2]_y}
\frac{[2]_y}{[3]_x} \frac{[2]_y}{[4]_y} \frac{[3]_x}{[5]_x} \frac{[3]_x}{[6]_y} \cdots$.

%
\subsection{Diagrammatic modifications when $m=\infty$}
\label{subsec-exoticdiagramsinfty}
%
%

When $m=\infty$, the objects of study are the Frobenius extensions $R^s, R^t \subset R$. There are essentially no complications which arise from non-symmetric or unbalanced Cartan
matrices, as higher Demazure operators play no role. The definitions of $\DGti$ and $\DCti$ are entirely unchanged.

If one desires to define the boxless version of the category when $R = \Bbbk[\a_s,\a_t]$, one should adjust the Cartan relations and the circle forcing relations accordingly.

\begin{equation} \label{circforce1unb} \ig{1}{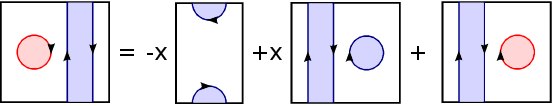} \end{equation}
\begin{equation} \label{circforce2unb} \ig{1}{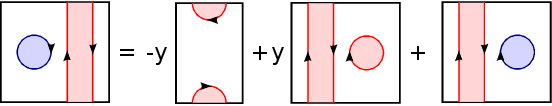} \end{equation}
\begin{equation} \label{circeval1unb} \ig{1}{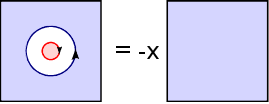} \end{equation}
\begin{equation} \label{circeval2unb} \ig{1}{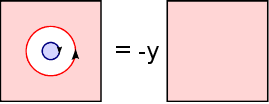} \end{equation}

Similar adjustments need to be made to the dot forcing relations in $\DCti$.

Of course, the non-symmetric version of the two-colored Temperley-Lieb category is to be used whenever appropriate. Whenever an example of a Jones-Wenzl projector is given, one must
replace the usual quantum numbers with two-colored quantum numbers, as in \eqref{JWnonsym}.

The only detail which has the slightest bit of subtlety is the proof that certain Hom spaces are non-zero by calculating the evaluation of the Jones-Wenzl projectors. In the non-symmetric
context, evaluation still gives a generically non-zero element, as can be checked with the symmetric specialization.

Aside from these minor changes, everything works verbatim!

\subsubsection{{\bf Diagrammatic modifications when $m<\infty$}}

We remind the reader that different assumptions are needed to define $\DG_m$ and to define $\DC_m$. The category $\DG_m$ depends on the existence of a Frobenius square, which requires the
realization to be faithful. As in the body of the paper, we guarantee this by assuming local non-degeneracy and lesser invertibility. To define the category $\DC_m$ one does not require
the realization to be faithful (though the proof of the SCT in this paper does), and one only needs Demazure surjectivity. One will still need to choose an arbitrary scalar, as though one
were choosing a Frobenius structure for $R^{s,t} \subset R$, even though for non-faithful realizations no such Frobenius structure exists. In both situations, one needs to assume the
realization is even-balanced; we require a 2-colored Jones-Wenzl projector which is rotation-invariant under color-preserving rotations.

First let us examine $\DG_m$. In \cite{EWFrob}, relations are presented for an arbitrary Frobenius square. For the Frobenius square structure induced by a symmetric Cartan matrix, these
give all the relations in section \ref{SingDefns} except the dihedral relation \eqref{SingJWRelation}. They are already expressed in a general format, and for a general Frobenius
realization, these relations are unchanged.

For relation \eqref{SingJWRelation}, the question arises: which Jones-Wenzl projector does one use? The LHS of \eqref{SingJWRelation} is invariant under rotation, and thus the RHS must be
as well. Therefore, one cannot simply use the right-blue-aligned $JW$ when blue appears on the right, and the right-red-aligned $JW$ when red appears on the right, because this is not
rotation-invariant. Instead, the RHS of \eqref{SingJWRelation} should be a rescaling of either Jones-Wenzl projector, with scalar determined not by the coloration but by the choice of
Frobenius structure. One chooses the scalar to be compatible with relation \eqref{singR2nonoriented2}. In other words, one chooses the (rescaling of the) Jones-Wenzl projector whose
evaluation is equal to the chosen product of roots $\LM$.

For example, when $m=3$ the evaluation of the right-blue-aligned Jones-Wenzl projector $JW_2$ from \eqref{JWnonsym} is $\a_s \a_t (\a_s + \frac{1}{x} \a_t)$. Since $xy=1$, this is the
product $\a_s \a_t t(\a_s)$, or $\LM^{(t)}$. Thus one should use this Jones-Wenzl projector if the chosen Frobenius structure on $R^{s,t} \subset R$ has product-coproduct $\LM^{(t)}$.

In particular, one cannot define the category $\DG_m$ when $m$ is even and the realization is unbalanced, because $JW_{m-1}$ is not rotation-invariant. When $m$ is odd, $JW_{m-1}$ is
rotation-invariant under color-preserving rotations, regardless of whether the realization is balanced or not.

In similar fashion, there are other scalars sprinkled throughout that one must keep track of. For instance, \eqref{ckis1} no longer holds on the nose, being true only up to scalar. This
scalar is the difference between $\pa^s_W$ and $\ldots \pa_s \pa_t$, as is clear from tracing the proof. As a consequence, there will be scalars involved attaching $v_k$ to $v_l$ along $m$
colored bands, but this does not affect any of the proofs. It may be a worthwhile exercise for the reader to confirm the following claim.

\begin{claim} We have \begin{equation} \label{ckisnt1} \ig{1}{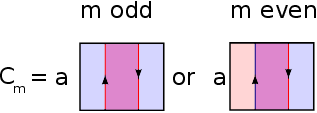} \end{equation} In this equation, the scalar $a$ is equal to the coefficient of the identity in the rotation of the
chosen Jones-Wenzl projector having the same alignment. \end{claim}

For example, if $m$ is odd and if we chose $\LM^{(t)}$ as our product of roots, the coefficient of the identity in the blue-aligned $JW$ is 1, and the coefficient in the red-aligned $JW$
is $\l$. Thus $a=1$ for the above-pictured blue-aligned $C_m$, but $a=\l$ for the red-aligned $C_m$. If $m$ is even, then the scalar $a$ does not depend on the coloration, because of
rotation-invariance.

Now consider the definition of $\DC_m$. Even though we do not assume the existence of a Frobenius square, we must still assume the existence of a rotation-invariant Jones-Wenzl projector
$JW_{m-1}$ (well-defined up to scalar), and we must fix a scalar multiple once and for all. Now \eqref{dot2m} holds as drawn, using that multiple of $JW$. We must modify \eqref{assoc2},
and consequentially \eqref{w0decomp}.

\begin{equation} \label{assoc2unbalanced} \ig{1}{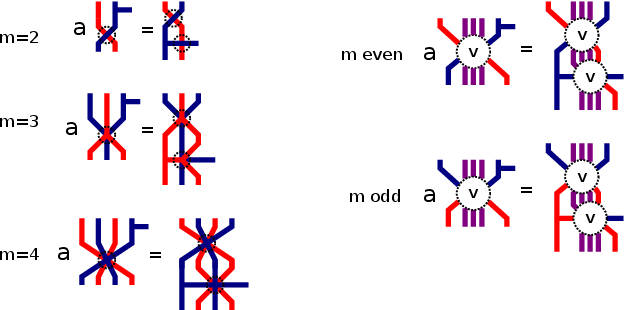} \end{equation}

\begin{equation} \label{w0decompunbalanced} \ig{1}{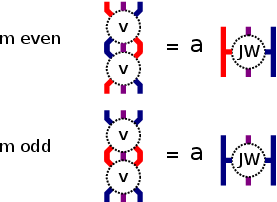} \end{equation}

Again, the scalar $a$ is the coefficient of the identity in the corresponding rotation of the chosen Jones-Wenzl. In the picture for $m$ odd for either relation, one takes the coefficient
of the identity in the red-aligned Jones-Wenzl. To check the consistency of this with \eqref{dot2m}, we recommend that the reader follow the two proofs of \eqref{w0decomp} using the new
unbalanced relations, and confirm that they agree. The reader can also check that \eqref{ckisnt1} matches with \eqref{assoc2unbalanced}, so that the functor $\DC_m \to \DG_m$ is still
well-defined.

Because of the scalar appearing in \eqref{w0decompunbalanced}, it is no longer the case that the doubled $m$-valent vertex is an idempotent. It is, however, an idempotent up to an
invertible scalar, independent of coloring. This invertible scalar is the product of the two possible values of $a$, for the two colorings.

There is an alternative approach to defining the diagrammatic category $\DC_m$, which sacrifices one measure of simplicity for another. There is a unique scalar multiple of the $2m$-valent
vertex which, when viewed as a map $BS(\hat{\ul{m}}_t) \to BS(\hat{\ul{m}}_s)$ after applying the functor to Bott-Samelson bimodules, will act on the lowest nonzero degree by sending $1
\ot 1 \ot \ldots \ot 1 \mapsto 1 \ot 1 \ot \ldots \ot 1$. One can draw this map as a $2m$-valent vertex where the vertex itself is colored blue. Similarly, there is a (different) scalar
multiple which ``preserves the 1-tensor" when viewed as a map $BS(\hat{\ul{m}}_s) \to BS(\hat{\ul{m}}_t)$ which we can draw as a $2m$-valent vertex with the vertex colored red. Each of
these maps is cyclic (i.e. invariant under 360 degree rotation, or even $\frac{360}{m}$ degree rotation), but one is not the rotation of the other.

With this convention, one must modify the relations \eqref{dot2m} and \eqref{assoc2} so that each keeps track of the color on the vertex, and must also add a new relation stating that
rotation of one version of the $2m$-valent vertex is equal to the other up to a scalar (this scalar, in fact, is $\l^{\pm 1}$). The reader can guess what the two new versions of
\eqref{dot2m} become. Relation \eqref{assoc2} will become \begin{equation} \label{assoc2alt} \ig{1}{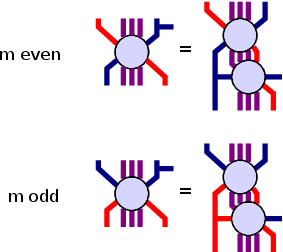} \end{equation} The key point in this coloration is that the
rightmost input of each $2m$-valent vertex in the diagram is red. In a version of this relation where the rightmost input is blue, one would color the vertex red instead. Similarly, one
can replace \eqref{w0decomp} with \begin{equation} \label{w0decompalt} \ig{1}{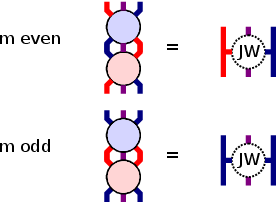} \end{equation} where the Jones-Wenzl in this relation has identity coefficient one.
The composition of these two vertices is thus a genuine idempotent.

The remainder of the study of $\DC_m$ works verbatim. The avid reader can figure out how to modify the sections on thickening and the Temperley-Lieb quotient accordingly.